\documentclass[12pt,a4paper]{amsart}
\usepackage{mathtools}
\usepackage{xcolor}
\usepackage{amsmath,amstext,amssymb,amsthm,amsfonts}
\usepackage[all]{xy} \xyoption{poly}

\newcommand{\nc}{\newcommand}

\nc{\one}{\mbox{\bf 1}}
\nc{\invtensor}{\underset{\leftarrow}{\otimes}}
\nc{\const}{\operatorname{const}}

\nc{\ad}{\operatorname{ad}}

\nc{\tr}{\operatorname{tr}}
\nc{\defect}{\operatorname{defect}}
\nc{\depth}{\operatorname{depth}}
\nc{\tp}{\operatorname{top}}
\nc{\rank}{\operatorname{rank}}
\nc{\corank}{\operatorname{corank}}
\nc{\codim}{\operatorname{codim}}
\nc{\sdim}{\operatorname{sdim}}
\nc{\mult}{\operatorname{mult}}
\nc{\ds}{\operatorname{ds}}
\nc{\tail}{\operatorname{tail}}
\nc{\howl}{\operatorname{howl}}
\nc{\spn}{\operatorname{span}}
\nc{\Sym}{\operatorname{Sym}}
\nc{\sym}{\operatorname{sym}}
\nc{\id}{\operatorname{id}}
\nc{\Id}{\operatorname{Id}}
\nc{\Ree}{\operatorname{Re}}
\nc{\hi}{\operatorname{hi}}
\nc{\htt}{\operatorname{ht}}
\nc{\at}{\operatorname{at}}
\nc{\str}{\operatorname{str}}
\nc{\Iso}{\operatorname{Iso}}
\nc{\Ker}{\operatorname{Ker}}
\nc{\mspec}{\operatorname{mspec}}
\nc{\rker}{\operatorname{rKer}}
\nc{\im}{\operatorname{Im}}
\nc{\osp}{\mathfrak{osp}}
\nc{\sgn}{\operatorname{sgn}}
\nc{\F}{\operatorname{F}}

\nc{\DS}{\operatorname{DS}}
\nc{\Soc}{\operatorname{Soc}}
\nc{\Inj}{\operatorname{Inj}}
\nc{\Hom}{\operatorname{Hom}}
\nc{\End}{\operatorname{End}}
\nc{\supp}{\operatorname{supp}}
\nc{\Card}{\operatorname{Card}}
\nc{\Ann}{\operatorname{Ann}}
\nc{\Ind}{\operatorname{Ind}}
\nc{\Coind}{\operatorname{Coind}}
\nc{\wt}{\operatorname{hwt}}
\nc{\hwt}{\operatorname{wt}}
\nc{\ch}{\operatorname{ch}}
\nc{\sch}{\operatorname{sch}}
\nc{\mdim}{\operatorname{mdim}}
\nc{\Stab}{\operatorname{Stab}}
\nc{\Irr}{\operatorname{Irr}}
\nc{\Spec}{\operatorname{Spec}}
\nc{\Res}{\operatorname{Res}}
\nc{\Aut}{\operatorname{Aut}}
\nc{\Ext}{\operatorname{Ext}}
\nc{\Prec}{\operatorname{Prec}}
\nc{\Fract}{\operatorname{Fract}}
\nc{\gr}{\operatorname{gr}}
\nc{\df}{\operatorname{df}}
\nc{\core}{\operatorname{core}}
\nc{\HC}{\operatorname{HC}}
\nc{\dpth}{\operatorname{dpth}}
\nc{\sw}{\operatorname{sw}}
\nc{\red}{\operatorname{red}}

\nc{\cl}{\operatorname{cl}}
\nc{\pos}{\operatorname{pos}}

\nc{\wdchi}{\widetilde{\chi}}
\nc{\wdH}{\widetilde{H}}
\nc{\wdN}{\widetilde{N}}
\nc{\wdM}{\widetilde{M}}
\nc{\wdO}{\widetilde{O}}
\nc{\wdR}{\widetilde{R}}

\nc{\wdV}{\widetilde{V}}

\nc{\wdC}{\widetilde{C}}

\nc{\pari}{\operatorname{dex}}
\nc{\atyp}{\operatorname{atyp}}

\nc{\Core}{\operatorname{Core}}

\nc{\Obj}{\operatorname{Obj}}
\nc{\Dglie}{\operatorname{{\mathcal D}glie}}
\nc{\Fin}{\operatorname{{\mathcal F}in}}
\nc{\pr}{\operatorname{pr}}
\nc{\Adm}{\operatorname{\mathcal{A}dm}}

\nc{\Sg}{{\cS(\fg)}}
\nc{\Shg}{{\cS(\fhg)}}
\nc{\Ug}{{\cU(\fg)}}
\nc{\Uhg}{{\cU(\fhg)}}
\nc{\Sh}{{\cS(\fh)}}
\nc{\Uh}{{\cU(\fh)}}
\nc{\Uhh}{{\cU(\fhh)}}
\nc{\Zg}{{{\mathcal{Z}}(\fg)}}

\nc{\Vir}{{\mathcal{V}ir}}
\nc{\NS}{{\mathcal{N}S}}

\nc{\tZg}{{\widetilde{\mathcal Z}({\mathfrak g})}}
\nc{\Zk}{{\mathcal Z}({\mathfrak k})}

\nc{\Mod}{\mathcal{M}od}

\nc{\Up}{{\mathcal U}({\mathfrak p})}
\nc{\Ah}{{\mathcal A}({\mathfrak h})}
\nc{\Ag}{{\mathcal A}({\mathfrak g})}
\nc{\Ap}{{\mathcal A}({\mathfrak p})}
\nc{\Zp}{{\mathcal Z}({\mathfrak p})}
\nc{\cR}{\mathcal R}
\nc{\cS}{\mathcal S}
\nc{\cP}{\mathcal P}
\nc{\cT}{\mathcal{T}}
\nc{\CC}{\mathcal C}
\nc{\cA}{\mathcal A}
\nc{\cE}{\mathcal E}
\nc{\cU}{\mathcal U}
\nc{\cZ}{\mathcal Z}
\nc{\cN}{\mathcal N}
\nc{\cM}{\mathcal M}
\nc{\cL}{\mathcal L}
\nc{\cF}{\mathcal F}
\nc{\fg}{\mathfrak g}
\nc{\cB}{\mathcal{B}}

\nc{\fo}{\mathfrak o}

\nc{\CO}{\mathcal O}
\nc{\CR}{\mathcal R}

\nc{\cK}{\mathcal{K}}
\nc{\cW}{\mathcal{W}}
\nc{\bM}{\mathbf{M}}
\nc{\bL}{\mathbf{L}}
\nc{\bN}{\mathbf{N}}

\nc{\zq}{\mathpzc q}

\nc{\fl}{\mathfrak l}
\nc{\fn}{\mathfrak n}
\nc{\fm}{\mathfrak m}
\nc{\fp}{\mathfrak p}
\nc{\fh}{\mathfrak h}
\nc{\ft}{\mathfrak t}
\nc{\fk}{\mathfrak k}
\nc{\fb}{\mathfrak b}
\nc{\fs}{\mathfrak s}
\nc{\fr}{\mathfrak r}
\nc{\fB}{\mathfrak B}

\nc{\vareps}{\varepsilon}
\nc{\varesp}{\varepsilon}
\nc{\veps}{\varepsilon}
\nc{\fgl}{\mathfrak{gl}}
\nc{\fsl}{\mathfrak{sl}}
\nc{\fpgl}{\mathfrak{pgl}}
\nc{\fpsl}{\mathfrak{psl}}
\nc{\fso}{\mathfrak{so}}
\nc{\fosp}{\mathfrak{osp}}
\nc{\fsp}{\mathfrak{sp}}
\nc{\fq}{\mathfrak q}
\nc{\fsq}{\mathfrak{sq}}
\nc{\fpsq}{\mathfrak{psq}}
\nc{\fpq}{\mathfrak{pq}}

\nc{\fhg}{\hat{\fg}}
\nc{\fhn}{\hat{\fn}}
\nc{\fhh}{\hat{\fh}}
\nc{\fhb}{\hat{\fb}}
\nc{\hrho}{\hat{\rho}}

\nc{\hsl}{\hat{\fsl}}
\nc{\fpo}{\mathfrak{po}}
\nc{\dirlim}{\underset{\rightarrow}{\lim}\,}
\nc{\nen}{\newenvironment}
\nc{\ol}{\overline}
\nc{\ul}{\underline}
\nc{\ra}{\rightarrow}
\nc{\lra}{\longrightarrow}
\nc{\Lra}{\Longrightarrow}
\nc{\bo}{\bar{1}}
\nc{\Lla}{\Longleftarrow}

\nc{\Llra}{\Longleftrightarrow}

\nc{\thla}{\twoheadleftarrow}

\nc{\lang}{(}
\nc{\rang}{)}

\nc{\hra}{\hookrightarrow}

\nc{\iso}{\overset{\sim}{\lra}}

\nc{\ssubset}{\underset{\not=}{\subset}}

\nc{\vac}{|0\rangle}

\nc{\simka}{{\ \scriptscriptstyle _{{\sim}}^\text{\tiny{k}}\ }}

\nc{\Thm}[1]{Theorem~\ref{#1}}
\nc{\Prop}[1]{Proposition~\ref{#1}}
\nc{\Lem}[1]{Lemma~\ref{#1}}
\nc{\Cor}[1]{Corollary~\ref{#1}}
\nc{\Conj}[1]{Conjecture~\ref{#1}}
\nc{\Claim}[1]{Claim~\ref{#1}}
\nc{\Defn}[1]{Definition~\ref{#1}}
\nc{\Exa}[1]{Example~\ref{#1}}
\nc{\Rem}[1]{Remark~\ref{#1}}
\nc{\Note}[1]{Note~\ref{#1}}
\nc{\Quest}[1]{Question~\ref{#1}}
\nc{\Hyp}[1]{Hypoth\`ese~\ref{#1}}
\nen{thm}[1]{\label{#1}{\bf Theorem.\ } \em}{}
\nen{prop}[1]{\label{#1}{\bf Proposition.\ } \em}{}
\nen{lem}[1]{\label{#1}{\bf Lemma.\ } \em}{}
\nen{cor}[1]{\label{#1}{\bf Corollary.\ } \em}{}
\nen{conj}[1]{\label{#1}{\bf Conjecture.\ } \em}{}

\nen{claim}[1]{\label{#1}{\bf Claim.\ } \em}{}

\nen{defn}[1]{\label{#1}{\bf Definition.\ } }{}
\nen{exa}[1]{\label{#1}{\bf Example.\ } }{}
\nen{rem}[1]{\label{#1}{\em Remark.\ } }{}
\nen{note}[1]{\label{#1}{\em Note.\ } }{}
\nen{exer}[1]{\label{#1}{\em Exercise.\ } }{}
\nen{sket}[1]{\label{#1}{\em Sketch of proof.\ } }{}
\nen{quest}[1]{\label{#1}{\bf Question.\ } \em}{}

\nen{hyp}[1]{\label{#1}{\bf Hypoth\`ese.\ } \em}{}
\begin{document}
\setcounter{section}{-1}
\setcounter{tocdepth}{1}

\title[ ]
{Depths and cores in the light of DS-functors}

\author{ Maria Gorelik}

\address[]{Dept. of Mathematics, The Weizmann Institute of Science,Rehovot 76100, Israel}
\email{maria.gorelik@weizmann.ac.il}

\begin{abstract}
The Dulfo-Serganova functors $\DS$ are tensor  functors relating representations
of different Lie superalgebras. In
this paper  we study the behaviour of various invariants,
 such as the defect, the dual Coxeter number, the atypicality and the cores, under the $\DS$-functor.
We introduce
a notion of $\depth$ playing the role of defect for algebras
and atypicality for modules. We mainly concentrate on examples of symmetrizable Kac-Moody and $Q$-type superalgebras. 
\end{abstract}

\maketitle

\section{Introduction}\label{intro}
\subsection{}
Let $\fg$ be a Lie superalgebra over the field $\mathbb{C}$. 
We set
$$X(\fg):=\{x\in \fg_1|\ [x,x]=0\}.$$ 
For $x\in X(\fg)$ we consider the functor $\DS_x$ 
introduced by Dulfo and Serganova in~\cite{DS}: 
$$M \longmapsto \DS_x(M):=M^x/xM.$$
This is a  functor from the category of $\fg$-modules 
to the category of $\fg_x$-modules, where 
$\fg_x:=\fg^x/[x,\fg]$.

Let $\fg$ be  an indecomposable finite-dimensional Kac-Moody 
superalgebra with a fixed non-degenerate invariant bilinear form.
Consider the following notions:
\begin{enumerate}
\item
the defect of $\fg$;
\item
the dual Coxeter number;
\item
the atypicality of a central character;
\item
for a ``non-exceptional'' $\fg$: the type of the root system
($A,B,C$ or $D$) and the ``core'' of a central character (see~\ref{coreKMQ},   \ref{corechi}).
\end{enumerate}

By~\cite{DS},  $\fg_x$ is a  finite-dimensional Kac-Moody 
superalgebra.  The functor $\DS_x$ reduces the 
defect of $\fg$ and the atypicality of each central character
by  the same non-negative integer; $\DS_x$  preserves
the dual Coxeter number,
 the type of the root system and the core of a central character.
In this paper we establish the similar results  
for the queer superalgebra $\fq_n$.

Let $\fg$ be an  indecomposable affine Kac-Moody superalgebra
 with a fixed non-degenerate invariant bilinear form. For some ``bad'' values of $x$, $\DS_x(\fg)$ is
not a Kac-Moody superalgebra. We show that, for ``nice'' values of
$x$, $\DS_x$  preserves the dual Coxeter number and
 the type of the root system, 
reducing the defect of $\fg$ by a non-negative integer
which we denote by $\rank x$. In the absence of central characters, we define
the atypicality and the core for each block in the BGG-category
$\CO(\fg)$. We show that, for certain values of
$x$, $\DS_x$ reduces the atypicality 
by  $\rank x$ and  preserves the cores.

Let $\fg$ be an arbitrary Lie superalgebra. We suggest
a notion of $\depth$, which plays the role of defect for algebras
and atypicality for modules; this notion is defined in terms of 
$\DS$-functors. For a finite-dimensional Kac-Moody superalgebra $\fg$
the  formulae
\begin{equation}\label{deepdef}\begin{array}{lcl}
(i) & & \depth(\fg)=\defect \fg\\ 
(ii) & & \depth(\cB)=\atyp\cB\ \ \text{
for a 
 block  $\cB$  in $\CO(\fg)$}
 \end{array}\end{equation}
 can be easily obtained from the results of~\cite{DS}.
We will check (i) for $\fp_n,\fq_n$ and 
the symmetrizable affine Kac-Moody superalgebras. We will prove (ii)
 for $\fq_n$ and show that $\depth\cB=n$ for 
each  block in $\Fin(\fp_n)$.

\subsection{Content of the paper}
In Section~\ref{sectionDS} we recall some properties of $\DS$-functor  and
 a construction of the map $\theta_x:\cZ(\fg)\to\cZ(\fg_x)$.

In Section~\ref{sectisoset} we introduce the notion of iso-set, 
which generalizes the notion
of isotropic sets introduced in~\cite{KWnum}. We denote by
$\defect \fg$ the maximal cardinality of an iso-set for $\fg$
(for the finite-dimensional Kac-Moody superalgebras this agrees with the standard definition). We introduce the set $X_{iso}(\fg)$;
for the finite-dimensional Kac-Moody superalgebras $X_{iso}(\fg)=X(\fg)$.

In Section~\ref{sectdepthdef}
 we introduce the notion of depth  for a superalgebra and its
modules. This notion is defined recursively using the functors
$\DS_x$ for $x\in X_{iso}$.
We show that $\depth\fg\geq\defect \fg$.
We briefly consider several examples: the "relatives" of $\fgl(m|n)$
and of $\mathfrak{p}_n$. In~\ref{finpn} we prove that $\depth\cB=n$ for 
each  block in $\Fin(\fp_n)$ and give an example when $\DS_1(\DS_1(L))\not\cong \DS_2(L)$.

In Section~\ref{sectKMQ}  we consider the case
when $\fg$ is a
 symmetrizable Kac-Moody superalgebra or one of 
$Q$-type superalgebras. 
The classical results in~\cite{BGG}, \cite{J} and~\cite{KK} 
determine the list  of simple modules
for each block in the BGG-category $\CO(\fg)$ (for the $Q$-type this description is up to parity shift). In~\Thm{propblocks} we show that, up to the parity shift, the non-critical modules
$L(\lambda-\rho)$ and $L(\nu-\rho)$
lie in the same block if and only if $\nu\in W(\lambda) (\lambda+S_{\lambda})$,
where $S_{\lambda}$ is the maximal iso-set orthogonal to $\lambda$
and $W(\lambda)$ is a certain subgroup of the Weyl group $W$
(for some cases this description was established
 in~\cite{CW} and in~\cite{CCL}).  In~\Cor{blockcrit} 
we give a similar formula for the critical blocks of several affine 
superalgebras.
Using these results we introduce the following block invariants:
the atypicality and the $\Core$ (as in the finite-dimensional case, $\Core$
is defined when $\fg$ is  not exceptional).

In Section~\ref{sectQ} we obtain several results for $\DS$-functors in the case 
when $\fg$ is a $Q$-type superalgebra ($\fq_n$ or its relatives). 
 One has  
 $$\depth\fq_n=[\frac{n}{2}],\ \ \DS_x(\fq_n)\cong \fq_{n-2\rank x}.$$
We show that  the map $\theta_x$ is surjective and that the dual map
$\theta^*_x$ preserves the core of a central character and increases
the atypicality of a central character by $\rank x$.
In particular, $\DS_x$ commutes with  translation functors.

In Section~\ref{sectbasic} we recall some results of~\cite{DS} for
the  finite-dimensional Kac-Moody superalgebras and
 prove~\Thm{thmtheta} describing the image of $\theta_x$ in this case.

In Section~\ref{gxaff} we describe $\fg_x$ for the case when
$\fg$ is a
 symmetrizable Kac-Moody superalgebra and $x\in X_{iso}(\fg)$.
 As in the finite-dimensional case,
$\fg_x$ has the same type of the root system as $\fg$ (for instance, $\DS_x(A(2m|2n)^{(4)})=A(2m-2r|2n-2r)^{(4)}$).
 We check~(\ref{deepdef}) (i)  and explain why $\DS_x$
 preserves the dual Coxeter number.

 In Sections~\ref{sectO} and~\ref{Thm91} we consider 
 a symmetrizable Kac-Moody superalgebra $\fg$ and 
the case when  $x\in X(\fg)_{iso}$ has a special form. 
The main result of this section is~\Thm{thmcore} stating that $\DS_x$ 
reduces the atypicality by $\rank x$ and "preserves the core".
In addition, we obtain 
the formula~(\ref{compmult}) which can be useful for studying the $\DS$-functor
on the category $\CO(\fg)$ (note that it is not clear whether 
$\DS_x(\CO(\fg))\subset\CO(\fg_x)$).

 \subsection{Index of  frequently used notation} \label{sec:app-index}
Throughout the paper the ground field is $\mathbb{C}$; 
$\mathbb{N}$ stands 
 for the set of non-negative integers. We always assume that 
the dimension of $\fg$ is
 at most countable. We will frequently used the following notation.

\begin{center}
\begin{tabular}{ll}
$ \Mod, \cZ(\fg), \CC(\chi), \mspec_{\CC}, \theta_x, \theta_x^*$ & \ref{DScentre}\\

$\supp$, iso-set, $\defect$,\  $X_{iso}$ &\ref{defniso-set}\\

$\Omega(N)$, $\Irr(\cF)$, $\Fin(\fg)$ & \ref{notation}\\

$\depth, \rank, X(\fg)_r$, $\breve{X}(N)$ & \ref{depthdef} \\

$W, (-|-), \Delta_{red},\Delta_{nis},\Delta_{iso}, \alpha^{\vee}$,  $\mathring{h}^*,\delta$, 
$\Delta(\lambda), W(\lambda) \ \ $
  & \ref{notatKMQ}\\

$M(\lambda)$, $L(\lambda)$, $\CO^{inf}$, $\CO$, $\CO_{KK}$, 
 $\sim$  &\ref{cc}\\

 $\atyp$ & \ref{atypdef} \\

$\Core(\lambda)$ & \ref{coreKMQ}\\

$\chi_{\lambda}, \Core(\chi)$ & \ref{corechi}\\

$\Omega_h$, $\CO^{inf}_h$ &  \ref{Oinfh}\\

\end{tabular}
\end{center}

\subsection{Acknowledgments} The author is grateful to  I.~Entova-Aizenbud, V.~Hinich, T.~Heidersdorf, C.~Hoyt,  S.~Reif and V.~Serganova for numerous 
helpful discussions.

\section{$\DS$-functor}\label{sectionDS}
The $\DS$-functor was introduced in~\cite{DS}. We recall definitions
and some results of~\cite{DS} below. In~\ref{depthprop} we list some properties of $\depth$.
We retain notation of Section~\ref{intro}.

\subsection{Construction}\label{DSconst}
We set $X(\fg):=\{x\in\fg_1|\ [x,x]=0\}$.
For a $\fg$-module $M$ and $g\in\fg$ we set
$M^g:=\Ker_M g$. For $x\in X(\fg)$ we introduce
$$\DS_x(M):=M^x/xM.$$
Notice that $\fg^x$ and 
$\fg_x:=\DS_x(\fg)=\fg^{x}/[x,\fg]$
 are  Lie superalgebras. Since $M^x, xM$ are $\fg^{x}$-invariant and $[x,\fg] M^x\subset xM$,  $\DS_x(M)$ is a $\fg^{x}$-module and $\fg_x$-module. 
Thus 
$$\DS_x: M\mapsto \DS_x(M)$$ 
is a functor from the category of $\fg$-modules to
the category of $\DS_x(\fg)$-modules.
One has $\DS_x(\Pi(N))=\Pi(\DS_x(N))$ and
$$\DS_x(M)\otimes\DS_x(N)=\DS_x(M\otimes N),\ \ \ \ 
\sdim(N)=\sdim(\DS_x(N)).$$

%
%

The following result is called Hinich's Lemma:
each  exact sequence of $\fg$-modules
$$0\to M_1\to N\to M_2\to 0$$
induces a long exact  sequence of $\fg_x$-modules
$$0\to Y\to \DS_x(M_1)\to\DS_x(N)\to\DS_x(M_2)\to \Pi(Y)\to 0,$$
where $Y$ is some $\fg_x$-module.  Lemma 2.1 in~\cite{HW} gives a similar result.

\subsection{Action on the centre}\label{DScentre}
Denote by 
 $\Mod$ (resp., by $\Mod_x$)  the category
of $\fg$ (resp., $\fg_x$)-modules.  
We denote by $\cU(\fg)$  the universal enveloping algebra of $\fg$ and by $\cZ(\fg)$ its centre. We denote by $\mspec\cZ(\fg)$ the set of central characters
$\Hom(\cZ(\fg),\mathbb{C})$. 
For each central character $\chi$ we denote by $\CC(\chi)$ 
the full subcategory of $\CC$ consisting of the modules annihilated by $\Ker\chi$
and set 
$$\mspec_{\CC}\cZ(\fg):=\{\chi\in \mspec\cZ(\fg)|\ \CC(\chi)\not=0\}.$$

\subsubsection{}
Consider the following  algebra homomorphisms
$$\cZ(\fg)\hookrightarrow \cU(\fg)^{\ad x}\to \DS_x(\cU(\fg)).$$
By~\cite{DS},  $\DS_x(\cU(\fg))=\cU(\fg_x)$. Since $\cZ(\fg)=\cU(\fg)^{\ad\fg}$ we obtain the algebra homomorphism
$$\theta_x:\ \cZ(\fg)\rightarrow \cZ(\fg_x)$$
and the dual map $\Hom(\cZ(\fg_x),\mathbb{C})\to \Hom(\cZ(\fg),\mathbb{C})$;
the map $\theta_x^*$ is the restriction of the latter
to $\mspec_{\Mod_x}\cZ(\fg_x)$. One has $\DS_x(\Mod(\chi))=0$ if $\chi\not\in\, Im\theta_x^*$.
Note that a surjectivity
of $\theta_x$ implies the injectivity of $\theta_x^*$.

\subsubsection{}
\begin{prop}{propBZ}
Take  $N\in\Mod(\chi)$.
\begin{enumerate}
\item
If $\theta_x$ is surjective, then $\DS_x(N)$ lies in $\Mod_x((\theta_x^*)^{-1}\chi)$.

\item
For each simple subquotient $L'$ of $\DS_x(N)$ there exists
$\chi'\in (\theta_x^*)^{-1}(\chi)$ such that  $L'\in \Mod_x(\chi')$.
In particular, $\DS_x(N)=0$ if $\chi\not\in \, Im(\theta_x^*)$.
\end{enumerate}
\end{prop}
\begin{proof}
Set $\fm:=\Ker\chi=\Ann_{\cZ(\fg)}N$.
One has $\theta_x(\fm)\subset \Ann_{\cZ(\fg_x)}\DS_x(N)$.

If $\theta_x$ is surjective, then  $\theta_x(\fm)$
is a  maximal ideal in $\cZ(\fg_x)$; this gives (i).

For (ii) recall that the dimension of $\fg_x$ is at most countable.
By Dixmier generalization of Schur's Lemma (see~\cite{Dix}), each simple
$\fg_x$-module admits a central character, so  $(\Ker\chi') L'=0$ for some
$\chi'\in \mspec_{\Mod}\cZ(\fg_x)$. By above,
$\theta_x(\fm)\subset \Ker\chi'$ which implies
$\theta_x(\chi')=\chi$ as required.
\end{proof}

\subsubsection{Remark}\label{Ian}
It is not clear when  $\mspec_{\Mod}\cZ(\fg)=\mspec \cZ(\fg)$, 
see Conjecture 13.5.1 in~\cite{Musbook}. In~\cite{Mus}
 I.~M.~Musson  proved the following generalization of
Duflo's Theorem: 
for a basic classical Lie superalgebra $\fg$
any primitive ideal in $\cU(\fg)$ is the annihilator of a simple highest weight module.
The proof of~\cite{Mus} works for any quasi-reductive superalgebra\footnote{
a finite-dimensional Lie superalgebra 
is called {\em quasi-reductive} if $\fg_0$ is reductive and 
$\fg_1$ is a semisimple $\fg_0$-module, see~\cite{Sred}.}. As a result, for
a quasi-reductive superalgebra $\fg$ one has
$$\mspec_{\Mod}\cZ(\fg)=\mspec_{\CO(\fg)} \cZ(\fg),$$
where $\CO(\fg)$ is the BGG-category for $\fg$.

%

\section{Iso-sets}\label{sectisoset}
In this section we introduce the  iso-sets and 
prove  Lemmatta~\ref{lem42},\ \ref{lemdsxy}
which seem to be useful for computations.
For a finite-dimensional 
Kac-Moody superalgebra $\fg$ the iso-sets are the same as isotropic sets
introduced in~\cite{KWnum} (see~\ref{bilinear}).

\subsection{Definitions and  first properties}\label{defniso-set}
Let $\fh$ be a  commutative subalgebra  of $\fg_0$ which 
 acts diagonally in the adjoint representation of $\fg$.
We introduce the multisets of even and odd roots in the usual way
($\Delta_0,\Delta_1\subset\fh^*\setminus\{0\}$). 
We write each $a\in\fg_i$ (for $i=0,1$) in the form
$$a=\sum_{\alpha\in\supp(a)} a_{\alpha},\ \ \text{ where }
\ a_{\alpha}\in\fg_{\alpha}\setminus\{0\},\ \ \supp(a)\subset \Delta_i\cup\{0\}.$$
\subsubsection{Definition}
We say that $S\subset \Delta_1$ is an {\em iso-set} if the elements of $S$ are linearly independent and for each $\alpha,\beta\in \Delta_1\cap (S\cup (-S))$ one has $\alpha+\beta\not\in\Delta_0$.
We denote by $\defect\fg$ the maximal cardinality of an iso-set
for all possible choices of $\fh$.

\subsubsection{}\label{isoxyh}
\begin{lem}{}
Let $x\in\fg$ be such that $S:=\supp(x)$ is an iso-set.
Then 
 $x\in X(\fg)$. Moreover, for each $y\in\fg_1$ with 
$\supp(y)\subset (-S)$ one has  $[x,y]\in\fh$ and 
$x,y,[x,y]$ span a subalgebra which is a quotient of $\fsl(1|1)$. 
\end{lem}
\begin{proof}
Set $S+S=\{\alpha+\beta|\ \alpha,\beta\in S\}$. One has 
$$\supp([x,x])\subset (S+S)\cap (\Delta_0\cup\{0\}).$$
By definition, $(S+S)\cap \Delta_0=\emptyset$.
Since the elements of $S$ are linearly independent, $0\not\in S+S$. Hence
$\supp([x,x])=\emptyset$, so $[x,x]=0$. Since
$\supp(x)\not=0$, the algebra $\fh$ contains $h$ satisfying $[h,x]\not=0$.
Hence $x\in {X}$.

Similar arguments give $[y,y]=0$ and
$\supp([x,y])\subset \{0\}$, that is $[x,y]\in\fh$.  One has
$$2[x[xy]]=[[xx]y]=0,\ \ \ 2[y[xy]]=[[yy]x]=0,$$
so $x,y,[x,y]$ span a  quotient of $\fsl(1|1)$.
\end{proof}

\subsubsection{Definition}\label{Xiso}
We denote by $X_{iso}(\fg)$ the set of all $x\in\fg$ with the following property:
there exists a subalgebra $\fh$ as above such that 
$\supp(x)$ is an iso-set. By above, $X_{iso}(\fg)\subset {X}(\fg)$.

We say that $x,x'\in X_{iso}(\fg)$ are equivalent if
there exists $\fh$ as above and 
an inner automorphism $\phi\in \Aut(\fg)$
such that $S:=\supp(x), S':=\supp(\phi(x'))$ are iso-sets and
$$-S'\cup S'=-S\cup S.$$
 
 We will often use $X$  instead of $X(\fg)$ and $X_{iso}$ instead of $X_{iso}(\fg)$.

\subsubsection{Notation}\label{notation}
In this paper $\fh$ is always as in~\ref{defniso-set}.
For a semisimple $\fh$-module $N$ we denote  by $\Omega(N)$
the set of weights of $N$. 
For each category $\CC$ we denote by $\Irr(\CC)$ the set of simple modules
in $\CC$. If $\fg$ is finite-dimensional,
we denote by  $\Fin(\fg)$ the full category of finite-dimensional modules
which are semisimple over $\fg_0$.

\subsubsection{Remark}\label{bilinear}
Consider the case when $\fg^{\fh}\cap\fg_0=\fh$ and $\fg$ admits an even non-degenerate symmetric 
invariant bilinear form. Then this form induces a non-degenerate symmetric form on $\fh^*$
which we denote by $(-|-)$.

Fix $\alpha\in\Delta_1$ with $2\alpha\not\in\Delta_0$ and a pair
 $x_{\alpha}\in \fg_{\alpha}$,  $x_{-\alpha}\in\fg_{-\alpha}$  with $(x_{\alpha},x_{-\alpha})=1$. Arguying as in~\cite{Kbook}, Thm.2.2 we see
that $x_{\alpha}, x_{-\alpha},
[x_{\alpha}, x_{-\alpha}]$  form
an $\fsl(1|1)$-triple and 
$(\nu|\alpha)=\nu([x_{\alpha}, x_{-\alpha}])$ for each $\nu\in\fh^*$.
This implies $(\alpha|\alpha)=0$.
As a result, each iso-set $S\subset\Delta_1$ 
is a basis of an isotropic subspace of $\fh^*$.

\subsection{Iso-sets and $\DS$-functor}

Let $S=\{\beta_i\}_{i=1}^r$ be an iso-set
and $y_i\in X$ be such that $\supp(y_i)=\{\beta_i\}$.
Fix $s$ with $1\leq s<r$ and set
$$x:=\sum_{i=1}^s y_i,\ \ \ \fh_x:=\fh^x/([x,\fg]\cap \fh)=\fh^x/
\big(\sum_{i=1}^s [y_i,\fg_{-\beta_i}]\bigr)$$ 
We view $\fh_x$ as a subalgebra of $\fg_x$; clearly, 
$\fh_x$ is a finite-dimensional commutative subalgebra 
of $(\fg_x)_0$ which 
 acts diagonally in the adjoint representation of $\fg_x$.
We introduce the multisets $(\Delta_x)_0, (\Delta_x)_1\subset\fh^*_x\setminus\{0\}$
of even and odd roots in the usual way. 

The space $\sum_{i=s+1}^r\fg_{\beta_i}$ lies
in $\fg^x$ and has zero intersection with $[x,\fg]$.
 This gives an embedding of
$\sum_{i=s+1}^r \fg_{\beta_i}\to \fg_x$; for 
each $i=s+1,\ldots,r$  the image of $\fg_{\beta_i}$
lies in  $(\fg_x)_{\beta'_i}$ for some $\beta'_i\in (\Delta_x)_1\cup\{0\}$.

\subsubsection{}
\begin{lem}{lemisoset}
The set 
$\{\beta'_i\}_{i=s+1}^r$
is an iso-set in $\Delta_x$.
\end{lem}
\begin{proof}
Assume that $\sum_{i=s+1}^r c_i\beta'_i=0$ for some scalars $c_i\in\mathbb{C}$
which are not all equal to zero. Since $\{\beta_i\}_{i=1}^r$
are linearly independent there exists
$h\in \fh$ such that 
$$\sum_{i=s+1}^r c_i\beta_i(h)=1,\ \ \ \beta_1(h)=\beta_2(h)=\ldots=\beta_s(h)=0.$$
Then $h\in \fh^x$; denoting by $\ol{h}$ the image of $h$ in $\fh_x$ we obtain 
$\sum_{i=s+1}^r c_i\beta'_i(\ol{h})=1$, a contradiction.
\end{proof}

\subsection{}
\begin{lem}{lem42}
Let  $x=\displaystyle \sum_{\alpha\in\supp(x)} x_{\alpha}$ be such that $\supp(x)$ is an iso-set.
For each $\alpha\in \supp(x)$ fix any element $h_{\alpha}\in [\fg_{-\alpha},x_{\alpha}]$.

If $N$ is a $\fg$-module with a diagonal action of $\fh$, then
 $$M:=\displaystyle\bigcap_{\alpha\in \supp(x)}N^{h_{\alpha}}$$ 
 is a $\mathbb{C}x$-submodule of $N$ and the natural 
 map $\DS_x(M)\to \DS_x(N)$ is bijective.
\end{lem}
\begin{proof}
The proof follows the idea of the proof of Lemma 4.2 in~\cite{Skw}.
Set 
$$S:=\{\beta\in \supp(x)|\ h_{\beta}\not=0\}.$$ 
If $S$ is empty, then $M=N$ and the assertion is tautological. 
Assume that $S\not=\emptyset$ and that $N$ is indecomposable. 
Let $\mathbb{F}$ be the minimal subfield
of $\mathbb{C}$ which contains
$$\{\lambda(h_{\beta})\}_{\beta\in S,\lambda\in\Omega(N)}.$$
Since $\Delta$ is at most countable and $N$ is indecomposable,
the set $\Omega(N)$ is at most countable, so 
the degree $[\mathbb{C}:\mathbb{F}]$ is infinite.
We fix a  $\mathbb{F}$-linearly independent set
 $\{q_{\beta}\}_{\beta\in S}\subset\mathbb{C}$. For each $\beta\in S$ take $x_{-\beta}\in \fg_{-\beta}$ such that
$h_{\beta}=[x_{\beta},x_{-\beta}]$. We introduce
$$y:=\sum_{\beta\in S}
   q_{\beta}x_{-\beta},\ \
   h:=[x,y]=\sum_{\beta\in S}
   q_{\beta}h_{\beta}.$$

By~\ref{isoxyh},
 $x,y,h$ form an $\fsl(1|1)$-triple (since $h\not=0$).
As an $\fsl(1|1)$-module $N$ can be written as 
$N=N^h\oplus N^{\perp}$, where 
$$N^{h}=\sum_{\nu\in\Omega(N):\ \nu(h)=0}N_{\nu},\ \ \ \ 
N^{\perp}:=\sum_{\mu\in\Omega(N):\ \mu(h)\not=0}N_{\mu}.$$
Since $\{q_{\beta}\}_{\beta\in S}$ are linearly independent over $\mathbb{F}$
one has $N^h=M$. 
Notice that  $N^{\perp}$ is a sum of typical $\fsl(1|1)$-modules, so
$\DS_x(N^{\perp})=0$. The assertion follows.
\end{proof}

\subsubsection{Remark}
For $\fg$ as in~\ref{bilinear} one has  $M=\displaystyle\sum_{\nu\in\fh^*: (\nu|S)=0} N_{\nu}$.

\subsection{Case of disjoint isosets}\label{dsxy}
Let $x,y\in X$ be such that 
$$\supp(x)\cap\supp(y)=\emptyset,\ \ \ 
\supp(x)\cup\supp(y)\text{ is an iso-set.}$$
By~\ref{notation}, $y$ has a non-zero image in $\fg_x$; we denote this image by
$\ol{y}$.

\subsubsection{}
\begin{lem}{lemdsxy}
Let $x,y\in X$ be as in~\ref{dsxy}.
Let $N$ be a finite-dimensional 
$\fg$-module with a diagonal action of $\fh$.
If
$\dim \DS_{\ol{y}}\DS_x(N)=:(p|q)$, then
$$\dim \DS_{x+y}(N)=(p-j|q-j)$$
for some $j\geq 0$.
\end{lem}
\begin{proof}
 Take $h\in\fh$ such that $\alpha(h)=1$, $\beta(h)=-1$ for each $\alpha\in\supp(x)$
and $\beta\in\supp(y)$. Then
$$[x,y]=0,\ \ [h,x]=x,\ \ [h,y]=-y,$$
so the algebra spanned by
$x,y,h$ is isomorphic to $\mathfrak{pgl}(1|1)$. We identify
$y$ with $\ol{y}$.
 
It is enough to check the inequality on each finite-dimensional
indecomposable $\mathfrak{pgl}(1|1)$-module. These modules
were classified in~\cite{Ger}.
For each $c$ there are  one-dimensional modules $L(c), \Pi(L(c))$, where
$h$ acts by $cId$ and $x,y$ by zero. Up to the multiplication
on such modules, we have the following classes of non-isomorphic
indecomposable modules: a $4$-dimensional projective modules $M_4$ satisfying
$\DS_x(M_4)=\DS_{x+y}(M_4)=0$ and the ``zigzag'' modules $V_s^{\pm}$.
Each zigzag modules has a basis $\{v_i\}_{i=1}^s$ with $p(v_{i+1})=\ol{i}$ and
$hv_i=iv_{i}$; 
we depict each module by the diagram, where $xv_i=v_{i+1}$ is depicted as
$v_i{\longrightarrow}v_{i+1}$ and  $yv_i=v_{i-1}$ is depicted as
 $v_i{\longrightarrow}v_{i-1}$.
 We have
$$\begin{array}{ll}
V_{2n}^+ & v_1{\longrightarrow}v_2\longleftarrow v_3\longrightarrow v_4\longleftarrow\ldots \longrightarrow v_{2n}\\
V_{2n}^-& v_1{\longleftarrow}v_2\longrightarrow v_3\longleftarrow v_4\longrightarrow\ldots \longleftarrow v_{2n}
\\
V_{2n-1}^+& v_1{\longrightarrow}v_2\longleftarrow v_3\longrightarrow v_4\longleftarrow\ldots \longleftarrow v_{2n-1}\\
V_{2n-1}^-& v_1{\longleftarrow}v_2\longrightarrow v_3\longleftarrow v_4\longrightarrow\ldots \longrightarrow v_{2n-1}
\end{array}$$
(The modules $V_1^{\pm}$ are trivial). One sees that
$$\begin{array}{l}
\dim\DS_x(V^{\pm}_{2n-1})=\dim\DS_y(V^{\pm}_{2n+1})=\dim\DS_{x+y}(V^{\pm}_{2n+1})=(1|0),\\
\DS_x(V^+_{2n})=\DS_x(V^-_2)=0,\ \ \  \DS_{x+y}(V^{\pm}_{2n})=0\end{array}$$
and $\dim \DS_{\ol{y}}\DS_x(V^-_{2n}))=\dim\DS_x(V^-_{2n})=(1|1)$ for $n>1$.
 The assertion follows.
\end{proof}

\subsubsection{Remark}\label{remdsxy}
Note that
$\DS_{\ol{y}}\DS_x(V^{\pm}_2)=0$ and for $n>1$ one has
$$\begin{array}{l}\\
\DS_y(V_{2n}^-)=0,\ \ \ \ \
\dim \DS_x(V_{2n}^-)=\dim\DS_{\ol{y}}\DS_x(V_{2n}^-)=(1|1);\\
\DS_x(V_{2n}^+)=0,\ \ \ \ \ 
\dim \DS_y(V_{2n}^+)=\dim\DS_{\ol{y}}\DS_x(V_{2n}^+)=(1|1).
\end{array}$$

In many examples
$\dim\DS_{x+y}(N)=\dim \DS_{\ol{y}}\DS_{x}(N)$ for any $x,y$ as in~\ref{dsxy}.
By above, this property
is equivalent to the fact that 
 $N$ does not contain $\mathfrak{pgl}(1|1)$-modules
 $V_{2n}^{\pm}$ with $n>1$ as direct summands;
by~\cite{Quella}, the modules with this property
``form a tensor subalgebra'' in $\Fin(\mathfrak{pgl}(1|1))$. 
This property holds if $N$ is a simple finite-dimensional module over
 a Kac-Moody
superalgebra $\fg$ (see~\cite{HW}, \cite{GH});
this does not hold for $\fg=\fp_2$: see~\ref{examplep2} below.

\section{DS-depth}\label{sectdepthdef}
In this section we introduce a notion of depth  for a superalgebra and its
modules. In~\ref{relativesgl} and~\ref{isosetp} 
we briefly consider several examples: 
 we find $\depth\fg$ for the case when $\fg$ is a "relative" of $\fgl(n|n)$
or a "relative" of $\mathfrak{p}_n$.

\subsection{Definitions}\label{depthdef}
We define $\depth(\fg)\in\mathbb{N}\cup\{\infty\}$ by the formula
$$\depth(\fg)=\left\{\begin{array}{ll}
0 & \text{ if }{X}_{iso}=0\\
1+\displaystyle\max_{x\in {X}_{iso}\setminus\{0\}}\depth(\fg_x) & \text{ if }{X}_{iso}\not=0.
\end{array}\right.
$$ 

For $x\in X$ we define $\ \rank x:=\depth(\fg)-\depth(\fg_x)$.
We  set
$$X(\fg)_r:=\{x\in X|\ \rank x=r\}.$$

For a $\fg$-module $N$ we set 
$$\breve{X}(N):=\{x\in {X}_{iso}\setminus\{0\}|\ \DS_x(N)\not=0\}$$

and  introduce $\depth(N)\in\mathbb{N}\cup\{\infty\}$ recursively by
$$ \depth(N):=\left\{\begin{array}{ll}
0 & \text{ if }\breve{X}(N)=\emptyset\\
\displaystyle\max_{x\in \breve{X}(N)}\bigl(\depth(\DS_x(N))+\rank x\bigr) & \text{ if }\breve{X}(N)\not=\emptyset.
\end{array}\right.$$
For a full subcategory of $\fg$-modules $\CC$
we define  $\depth(\CC)$ as the maximum of $\depth(N)$ for $N\in\CC$.
For $\chi\in \mspec\cZ(\fg)$ we set 
$$\depth\chi:= \depth\mathcal{M}od(\chi).
$$

\subsubsection{Remark}
In~\ref{examplep2} we give an example of a simple module
$L$ satisfying $\max_{x\in \breve{X}(L)} \rank x\leq \depth(L)$.

\subsection{Properties of $\depth$}\label{depthprop}
Clearly,
$\depth(\fg'\times\fg'')=\depth(\fg')+\depth(\fg'')$.

One has $\depth(N)=0$ if and only if $\breve{X}(N)=\emptyset$.
Moreover, $\depth(N)\leq \depth(\fg)$ and 
$\depth(N)=\depth\fg$ if  $\sdim N\not=0$.
Note that
$\depth\fg$ may be greater than the $\depth$ of the adjoint module
(for $\fg=\fgl(1|1), \osp(2|2),\osp(3|2),\fpsq_3$ we have
$\depth(\fg)=1$ and $\depth(\, Ad)=0$). 

\subsubsection{}By~\ref{DSconst} for
 any $N',N''\in\Mod$ one has
$$\begin{array}{c}
\depth(N'\oplus N'')=\max(\depth(N')),\depth(N'')),\\
\depth(N'\otimes N'')=\min(\depth(N')),\depth(N'')).\end{array}$$
\subsubsection{}
In the light of~\Prop{propBZ} (ii) for $\chi\in \mspec\cZ(\fg)$ we have
\begin{equation}\label{depthleq}
\depth\chi\leq \max\{\rank x|\ \chi\in \, Im\theta_x^*\}.
\end{equation}

\subsubsection{}
Using~\Lem{lemisoset} and the induction on $r$ we obtain:
\begin{itemize}
\item
 if $\supp(x)$ is an iso-set of cardinality $r$, then $\rank x\geq r$;
\item
 $\depth \fg\geq \defect\fg$ (where $\defect$ as in~\ref{defniso-set}).

\end{itemize}

For an example when $\rank x$ is greater than the cardinality of
$\supp(x)$, see~\ref{rempn}.

\subsubsection{Example}
Let $\fg$ be a finite-dimensional  Kac-Moody superalgebra. 
From the results of~\cite{DS}, it follows that $X(\fg)=X(\fg)_{iso}$ and 
$\depth(\fg)$ is equal to the defect of $\fg$. 
For each block $\cB$ in $\Fin(\fg)$ (or in $\CO(\fg)$) $\depth(\cB)$
 is equal to the  atypicality
 of $\cB$. Moreover,
by~\cite{Skw}, the atypicality of a simple finite-dimensional module $L$ is equal to $\depth(L)$ and
\begin{equation}\label{depthDsxL}
\depth(L)=\depth\DS_x(L)+\rank x.
\end{equation}

\subsubsection{}\label{Serganovaproperty}
Let $\cB$ be a Serre subcategory of $\Mod$, i.e. 
 the full subcategory consisting of the modules of finite length
 whose all simple
subquotients lie in a given set $\Irr(\cB)$. Recall that
$\depth(\cB)$ is the maximal $\depth(N)$ for $N\in\cB$.
By Hinich's Lemma one has
$$\depth(\cB)=\max_{L\in\Irr(\cB)}\depth(L).$$
We say that a block $\cB$ satisfies {\em Serganova property} if
\begin{equation}\label{depthBL}
\depth(\cB)=\depth(L) \ \text{ for each } L\in\Irr(\cB).\end{equation}

By~\cite{Skw}, for a finite-dimensional Kac-Moody superalgebra $\fg$
this property holds for each  block in $\Fin(\fg)$. This property does not hold for strange superalgebras,
see~\ref{examplep2},\ref{mainresultq} below.

 \subsubsection{Remark}
One might expect that in "good cases" $\depth$ has the following properties:

  if $\fm$ be an ideal in $\fg$ and  $N$ is a $\fg/\fm$-module, then
 $\depth(N)=\depth\Res_{\fg/\fm}^{\fg}(N)$;

$\depth(N)=\depth\DS_x(N)+\rank x$.

\subsubsection{}\label{depthKMprop}
In~\ref{relativesgl} we will see that  "the relatives of $\fgl(m|n)$" 
 satisfy the following properties:
\begin{enumerate}
\item
$\depth (\fg)=\defect \fg$;
\item
all elements of a given rank  in $X_{iso}$ are equivalent with respect to
the equivalence relation introduced in~\ref{Xiso};
\item
if $\supp(x)$ is an iso-set, then $\rank(x)$ is equal to the cardinality of $\supp(x)$;
\item
the elements  $x\in (X\setminus X_{iso})$ have the maximal possible rank;

\item
for $x\in X_{iso}$ the rank of $x$
is equal to the rank of the corresponding matrix (i.e. the dimension of the image
of $x$ in the natural representation).

\end{enumerate}

In~\ref{isosetp} we will see that  $\fp_n$ satisfy (i),(ii), (iv)
and (v). Later we will show that the symmetrizable Kac-Moody superalgebras
and the $Q$-type superalgebras satisfy the properties (i)---(iv).

\subsection{Examples: $\fgl(n|n)$ and its relatives}\label{relativesgl}
Recall that $\fgl(m|n)$ consists of the block matrices
\begin{equation}\label{Tmatrix}
T_{A,B, C,D}:=\begin{pmatrix}
A & B\\
C & D
\end{pmatrix}\end{equation}

We consider Lie superalgebras $\fg=\fgl(n|n),\fsl(n|n),\mathfrak{pgl}(n|n),\mathfrak{psl}(n|n)$ and  identify $\fg_1$ for all these superalgebras.

\subsubsection{Case $n=1$}
The algebra $\fgl(1|1)$ is spanned by  even elements $h,z$ and 
odd elements $E,F$ with the relations
$$[h,E]=2E, \ \  [h,F]=-2F,\ \ [E,F]=z,\ \ [z,h]=[z,E]=[z,F]=0$$
and $\fpgl(1|1)=\fgl(1|1)/\mathbb{C}z$.
One has 
$\defect (\fgl(1|1))=\defect (\fpgl(1|1))=1$ and
$$X(\fgl(1|1))=X(\fgl(1|1))_{iso}=\mathbb{C}E\cup\mathbb{C}F=X(\fpgl(1|1))=X(\fpgl(1|1))_{iso}.$$
For $x\in X(\fgl(1|1))\setminus \{0\}$ we have
$$\DS_x(\fgl(1|1))=\fgl(0|0)=0,\ \ \ \DS_x(\fpgl(1|1))\cong \Pi(\mathbb{C}).$$
Hence $\depth\fgl(1|1)=\depth (\fpgl(1|1))=1$ and
$$ X((\fgl(1|1))_1=X((\fpgl(1|1))_1=
X(\fgl(1|1))\setminus\{0\}.$$

The algebra $\fsl(1|1)$ is spanned by $z,E,F$.
One has $\defect \fsl(1|1)=0$
$$X(\fsl(1|1))=X(\fgl(1|1))=\mathbb{C}E\cup\mathbb{C}F,\ \ X(\fsl(1|1))_{iso}=0,$$
so $\depth (\fsl(1|1))=0$ (in particular, 
all $x\in X(\fsl(1|1))$ have zero rank). Note that
for non-zero $x$ the algebra $\DS_x(\fsl(1|1))$ can be identified with 
$\mathbb{C}x$.

The algebra $\fpsl(1|1)$ is a commutative superalgebra of dimension
$(0|2)$. One has $X(\fpsl(1|1))=\fpsl(1|1)$ and
$X(\fpsl(1|1))_{iso}=0$, so 
$\depth (\fpsl(1|1))=\defect \fpsl(1|1)=0$.
Note that $\DS_x(\fpsl(1|1))=\fpsl(1|1)$ for all $x$.

\subsubsection{Case $\fgl(m|n)$}
One has $\defect\fgl(m|n)=\min(m,n)$.
By~\cite{DS} ,
$$X(\fgl(m|n))=X(\fgl(m|n))_{iso}=\{T_{0,B,C,0}|\ BC=0,\ \ CB=0\}$$
and $\depth(\fgl(m|n))=\min(m,n)$, 
$\DS_x(\mathfrak{gl}(m|n))=\mathfrak{gl}(m-r|n-r)$, where $r:=\rank x$.
Moreover, $\rank x$ is equal to the rank of the matrix of $x$ in~(\ref{Tmatrix}).
 The same formulae hold for $\fsl(m|n)$ with $m\not=n$. 

\subsubsection{Cases $\fsl(n|n)$}
The iso-sets for $\fsl(n|n)$
coincide with the iso-sets of cardinality $r<n$ for $\fgl(n|n)$
(the iso-sets of cardinality $n$ becomes linearly dependent for 
$\fsl(n|n)$). This gives $\defect(\fsl(n|n))=n-1$. One has
$$X(\fsl(n|n)) =X(\fgl(n|n))=X(\fsl(n|n))_{iso}\coprod X(\fgl(n|n))_n$$
and for $x\in X(\fgl(n|n))_r$ we have 
$$\DS_x(\mathfrak{sl}(n|n))=\mathfrak{sl}(n-r|n-r)\ \text{ if } r<n,\ \ \DS_x(\fsl(n|n))=\Pi(\mathbb{C})\ \text{ if }r=n.$$
 Thus $\depth(\fsl(n|n))=n-1$ and
$X(\fsl(n|n))_r=X(\fgl(n|n))_r \text{  for }r<n-1$ with
$$X(\fsl(n|n))_{n-1}=X(\fgl(n|n))_{n-1}\coprod X(\fgl(n|n))_n.$$

\subsubsection{Cases $\mathfrak{pgl}(n|n),\mathfrak{psl}(n|n)$}
One has
$$X(\mathfrak{pgl}(n|n))=X(\mathfrak{psl}(n|n))
=X(\fgl(n|n))\coprod X', \ \text{ where } X':=\{T_{0, B,C,0}|\ BC\in\mathbb{C}^* Id\}.$$

For $x\in X(\fgl(n|n))_r$ with $r<n$ one has 
$$\DS_x(\mathfrak{pgl}(n|n))=\mathfrak{pgl}(n-r|n-r),\ \ \ \DS_x(\mathfrak{psl}(n|n))=\mathfrak{psl}(n-r|n-r)$$
and 
$\DS_x(\mathfrak{pgl}(n|n))=\Pi(\mathbb{C})$, $\DS_x(\mathfrak{psl}(n|n))=\fpsl(1|1)$ for
$x\in X(\fgl(n|n))_n\cup X'$.

We have $X(\fpgl(n|n))_{iso}=X(\fgl(n|n))$ and
$$\depth (\fpgl(n|n))=\defect \fpgl(n|n)=n.$$
This gives $X(\fpgl(n|n))_r=X(\fgl(n|n))_r$ for $r<n$,
$$X(\fpgl(n|n))_n=X(\fgl(n|n))_n\coprod X',\ \ X'=X(\fpgl(n|n))\setminus X(\fpgl(n|n))_{iso}.$$

We have $X(\fpsl(n|n))_{iso}=X(\fsl(n|n))_{iso}$ and
$$\depth (\fpsl(n|n))=\defect \fpsl(n|n)=n-1.$$
This gives $X(\fpsl(n|n))_r=X(\fsl(n|n))_r$ for $r<n-1$ and 
$$X(\fpsl(n|n))_n=X(\fsl(n|n))_{n-1}\coprod X'.$$

\subsection{$P$-type}\label{isosetp}
A strange Lie superalgebra $\fp_n$ is a subalgebra of $\fgl(n|n)$ consisting of the matrices with the block form
$$T_{A,B, C}:=\begin{pmatrix}
A & B\\
C & -A^t
\end{pmatrix}$$
where $B^t=B$ and $C^t=-C$. The commutant
  $\fp_n':=[\fp_n,\fp_n]$  is simple for $n\geq 3$;
one has  $\fp'_n=\{T_{A,B,C}|\ A\in\fsl_n\}$.

\subsubsection{}
One has $\fp_0=0$, $\fsp_1=\Pi(\mathbb{C})$ and
$$\begin{array}{l}
X(\fp_n)=X(\fp_n)_{iso}=X(\fgl(n|n))\cap \fp_n,\ \ 
X(\fp_n)_r=X(\fgl(n|n))_r\cap \fp_n,\\
X(\fp'_n)=X(\fsp_n)_{iso}\coprod X(\fp_n)_n,\ \ 
X(\fp'_n)_r=X(\fp_n)_r\ \text{ for }r<n-1
\end{array}
$$
and $X(\fp'_n)_{n-1}=X(\fp_n)_{n-1}\coprod X(\fp_n)_n$.
One has
$$\begin{array}{ll}\depth(\mathfrak{p}_n)=\defect \fp_n=n,\  & \DS_x(\mathfrak{p}_n)=\mathfrak{p}_{n-\rank x},\\
\depth(\mathfrak{p}'_n)=\defect \fp'_n=n-1,\ \ \  &
\DS_x(\mathfrak{p}'_n)=\mathfrak{p}'_{n-\rank x}.\end{array}$$

Notices that $GL_n$ acts transitively on $X(\fp_n)_1$;
we denote by $\DS_1$ the functor $\DS_x$ for some $x\in X(\fp_n)_1$
and set $\DS_1^{i+1}:=\DS_1\circ\DS_1^i$ (these functors 
are not uniquely defined).

\subsubsection{Remark}\label{rempn}
Take $x$ with  $\supp(x)=\{\vareps_1+\vareps_2\}$ (using the standard
notation of~\cite{KLie}). Notice that $\supp(x)$ is an iso-set of cardinality $1$ and
$\rank x=2$. It is not hard to see that $\fp_n$ contains  a Cartan subalgebra $\fh'$ such that 
the support of $x$ computed with respect to $\fh'$ is an iso-set
of cardinality two (i.e., $\{2\vareps_1,2\vareps_2\}$).

\subsection{Example: $\Fin(\fp_n)$}\label{finpn}
Consider $\fp_n$ with  the simple roots
$$\vareps_2-\vareps_1,\ldots,\vareps_n-\vareps_{n-1},-\vareps_n-\vareps_{n-1}.$$
Denote
by $\rho$ the Weyl vector:
$\rho:=\frac{1}{2}\sum_{\alpha\in\Delta^+} (-1)^{p(\alpha)}\alpha$.

We denote by $L_{\fp_n}(\mu)$  the simple  $\fp_n$-module of the highest
weight $\mu$ with the {\em even} highest weight space. The module 
$L_{\fp_n}(\sum_{i=1}^n a_i\vareps_i-\rho)$ is finite-dimensional if and only if
$a_{i+1}-a_i\in\mathbb{N}_{>0}$ for $i=1,\ldots,n-1$.

\subsubsection{}
\begin{lem}{lemperi}
Take $n\geq 2$ and a weight $\lambda=\sum_{i=1}^n a_i\vareps_i$ with
$a_{i+1}-a_i\in\mathbb{N}_{>2n+2}$ 
for $i=1,\ldots,n-1$. Then
$\dim\DS_1^n(L_{\fp_n}(\lambda-\rho))=(\frac{n!}{2}|\frac{n!}{2})$.
\end{lem}
\begin{proof}
The composition factors of $\DS_1(L)$ for $L\in\Irr(\Fin(\fp_n))$
are described in~\cite{ES}. From this description one
sees that  the composition factors of $\DS_1(L_{\fp_n}(\lambda-\rho_n))$ are
$$\{\Pi^{s}L_{\fp_{n-1}}(\nu_s-\rho')\}_{s=0}^{n-1},$$ where 
$\rho'$ is the Weyl vector for $\fp_{n-1}$ and
$\nu_s=\sum_{i=1}^{n-1}b_i\vareps_i$ with
$$b_1\leq b_2\leq b_{n-1},\ \ \ \{b_i\}_{i=1}^{n-1}=\{a_i\}_{i=1}^n\setminus\{a_{n-s}\}.$$
In particular, $|(\nu_s-\nu_j|\vareps_{n-s-1})|>2n+2$ for all $0<s<j\leq n-1$.
Using~\cite{perigirls},  Prop. 3.7.1 we conclude that 
$\DS_1(L(\lambda))$ is completely reducible. For $n=2$ this gives
$$\DS_1(L_{\fp_2}(\lambda-\rho))=\Pi(L_{\fp_1}(a_1\vareps_1-\rho'))\oplus L_{\fp_1}(a_2\vareps_1-\rho')$$
which implies  the required formula (since $\dim L_{\fp_1}(a\vareps_1)=(1|0)$
for each $a\in\mathbb{C}$).
The general case follows by induction on $n$, since, by above,
$\DS_1(L(\lambda))$ is a completely reducible module of length $n$ and
the highest weight of each composition factor 
 satisfies the condition of the lemma (i.e., $b_{i+1}-b_i>2n$).
\end{proof}

\subsubsection{}
 By~\cite{perigirls}, each block in 
the category $\Fin(\fp_n)$  contains a simple module $L_{\fp_n}(\lambda-\rho)$
with $\lambda$ satisfying the conditions of~\Lem{lemperi}.
Hence
$$\depth\cB=n\ \text{ for each block $\cB$  in $\Fin(\fp_n)$.}$$

\subsubsection{Example: $\Fin(\fp_2)$}\label{examplep2}
Fix the root vectors 
$$x_i\in \fg_{2\vareps_i}, i=1,2,\ \ \ 
x_{\pm}\in\fg_{\pm(\vareps_1+\vareps_2)}.$$
Note that $\rank x_i=1$ and  $\rank x_{\pm}=2$.

By~\cite{perigirls}, 
the category $\Fin(\fp_2)$ contains the following  blocks: 
$\cB_0$, containing the trivial module,  and
$\cB_1$ containing the natural $(2|2)$-dimensional module;
any block in $\Fin(\fp_2)$ is equivalent to one of these blocks via an  equivalence  given by the tensoring on a one-dimensional module. 

One has $\Irr(\cB_0)=\{{L}^{(2j+1)}\}_{j=0}^{\infty}$, where
$\dim {L}^{(1)}=1$ and
$\dim {L}^{(j)}=(2j+1|2j+1)$ for $j>1$. For example,
${L}^{(3)}$ is a simple submodule of the adjoint representation and
$$ \DS_{x_-}(L^{(3)})=0,\ \ \dim\DS_{x_+}({L}^{(3)})=(1|1).$$

One has $\Irr(\cB_1)=\{L^{(2j)}\}_{j=1}^{\infty}$, where
$\dim L^{(2j)}=(2j|2j)$.
The modules $L^{(2j)}$ provide interesting examples for~\Rem{remdsxy}: consider
$\fpgl(1|1)$ spanned by $x_1,x_2$; as a $\fpgl(1|1)$-module,
$L^{(2j)}=V^+_{2j}\oplus V^-_{2j}$  (see~\Lem{lemdsxy} for notation) and so
$$\DS_{x_1+x_2}(L^{(2j)})=0,\ \  \DS_{x_1}(\DS_{x_2}(L^{(2j)})=\mathbb{C}\oplus \Pi(\mathbb{C})\ \text{ for }j>1.$$
The $\fp_1$-module
 $\DS_1(L^{(2j)})$ is 
 an indecomposable (resp., semisimple)  for $j=1$ 
(resp., $j>1$) and $\dim\DS_1(L^{(2j)})=(1|1)$.
This gives
$$\depth L^{(2)}=1,\ \ \ \depth L^{(2j)}=2\ \text{ for }j>1.$$
Note that $X(\fp_2)_2=GL_2 x_-\coprod GL_2 (x_1+x_2)$.
It is easy to check that $\DS_{x_-}(L^{(2j)})=0$; using 
$\DS_{x_1+x_2}(L^{(2j)})=0$  we obtain 
$\DS_x(L^{(2j)})=0$  for each $x\in X(\fp_2)_2$. 
This gives
\begin{equation}\label{Ljj}
\max_{x\in \breve{X}(L^{(2j)})} \rank x=1\not=\depth(L^{(2j)}) \ \text{ for } j>1.
\end{equation}

\section{Iso-sets and blocks in the category $\CO$}\label{sectKMQ}
In this section $\fg$ is either an indecomposable
 symmetrizable Kac-Moody superalgebra with an isotropic root or one of the
$Q$-type superalgebras $\fpsq_m,\fpq_m$ for $m\geq 3$, $\fsq_m,\fq_m$
for $m\geq 2$. We will refer to the former  case as the "KM-case" and to the latter case as the "$Q$-type case".
By~\cite{Hoyt}, in the KM-case $\fg$ is either finite-dimensional (classified 
in~\cite{KLie}) or affine (classified in~\cite{vdL}).

\subsection{Notation}\label{notatKMQ}
We  denote  by $W$ the Weyl group
of $\fg_0$.
The algebra $\fg_0$ admits a non-degenerate invariant bilinear form and we denote by $(-|-)$ the corresponding form on $\fh^*$. 
We say that two triangular decompositions of $\fg$ are {\em compatible}
if they induce the same triangular decomposition of $\fg_0$; similarly,
we say that two bases (the sets of simple roots) are compatible
if the corresponding triangular decompositions are compatible.

We denote by $\Delta_{re}$ the set of real roots.
If $\fg$ is finite-dimensional, then 
$\Delta_{re}=\Delta$ and we set $\delta:=0$.
If $\fg$ is affine 
we denote by $\delta$ the minimal imaginary root. In both cases
$$\Delta\setminus\Delta_{re}=\mathbb{Z}\delta\setminus\{0\},\ \ \ 
(\delta|\Delta)=0.$$
 We introduce
$$\mathring{\fh}^*:=\left\{\begin{array}{ll}
\fh^* & \text{ if } \dim\fg<\infty\\
\{\lambda\in\fh^*|\ (\lambda|\delta)\not=0\}  & \text{ if } \dim\fg=\infty.
\end{array}\right.$$

We  denote by $\Sigma$  
a  base (the set of simple roots). 
For the Kac-Moody case  we 
fix a Weyl vector $\rho\in\fh^*$ satisfying $2(\rho|\alpha)=(\alpha|\alpha)$
for each $\alpha\in\Sigma$; for the $Q$-type superalgebras we set $\rho:=0$.
We introduce
$$\begin{array}{ll}
{\Delta}_{red}:=&\{\alpha\in\Delta_{re}\cap\Delta_0|\ \frac{\alpha}{2}\not\in\Delta\},\\
\Delta_{iso}:=&\{\alpha\in\Delta_{re}\cap \Delta_{1}|\ 2\alpha\not\in\Delta\},\\
\Delta_{nis}:=&\{\alpha\in\Delta_{re}\cap \Delta_{1}|\ 2\alpha\in\Delta\},
\end{array}$$
and $\Delta^+_{red}:=\Delta^+\cap\Delta_{red}$ and so on. 
For $\alpha\in\Delta_{iso}$  the subalgebra generated by $\fg_{\alpha}$
 and $\fg_{-\alpha}$ is isomorphic to $\fsl(1|1)$. 
One has
$$\Delta_{iso}=\left\{\begin{array}{ll}
\Delta_1\ & \text{  for the $Q$-type case}\\
\{\alpha\in\Delta_{re}\cap \Delta_{1}|\ (\alpha|\alpha)=0\} & \text{  for the KM-case.}\end{array}\right.$$

\subsubsection{}
For $\alpha\in\Delta_{re}\setminus\Delta_{iso}$ one has
 $(\alpha|\alpha)\not=0$
and we set 
$\alpha^{\vee}:=\frac{2\alpha}{(\alpha|\alpha)}$.
For the KM-case we set ${\alpha}^{\vee}:=\alpha$ for each $\alpha\in\Delta_{iso}$. In the $Q$-type case   
 $\Delta_{iso}=\{\vareps_i-\vareps_j\}_{i\not=j}$ (we
use the standard $\fgl$-notation) and
 we set ${\alpha}^{\vee}:=\vareps_i+\vareps_j$
for $\alpha:=\vareps_i-\vareps_j\in\Delta_1$.

In this way  $\alpha^{\vee}$ is defined for each $\alpha\in\Delta_{re}$
(if we identify $\fh^*$ with $\fh$ via the form $(-|-)$, then
the image of $\alpha^{\vee}$ in the usual coroot
lying in $[\fg_{\alpha},\fg_{-\alpha}]$).  Notice that in all cases $w{\alpha}^{\vee}=({w\alpha})^{\vee}$ for every $w\in W$.

\subsubsection{Affine root systems}\label{affro}
Let $\fg$ be affine.
 We call $\dot{\Sigma}\subset{\Sigma}$
a {\em finite part} of ${\Sigma}$ if  ${\Sigma}\setminus\dot{\Sigma}$
contains exactly one root  and
$\Sigma$ has a connected Dynkin diagram.  The finite parts of ${\Sigma}$ are described in~\cite{GKadm}, 13.2. We fix $d\in {\fh}$ with
$\delta(d)=1$ and $\alpha(d)=0$ for $\alpha\in\dot{\Sigma}$. 
Then 
$$\dot{\Delta}:=\{\alpha\in\Delta|\ \alpha(d)=0\}$$
is a finite root system  with a base $\dot{\Sigma}$. One has $\Delta_{iso}=\dot{\Delta}_{iso}+\mathbb{Z}\delta$ except for the cases $A(2m|2n)^{(4)}$, $D(m|n)^{(2)}$, where 
$\Delta_{iso}=\dot{\Delta}_{iso}+2\mathbb{Z}\delta$.


\subsubsection{The group $W(\lambda)$}\label{Deltalambda}
 We set
$$\Delta(\lambda):=\{\alpha\in\Delta_{re}\cap \Delta_0|\ 
(\lambda|\alpha^{\vee})\in\mathbb{Z}\ \text{ if }\frac{\alpha}{2}\not\in\Delta,\ \ (\lambda|\alpha^{\vee})\in\mathbb{Z}+\frac{1}{2}\ \text{ if }\frac{\alpha}{2}\in\Delta\}$$
and  denote by $W(\lambda)$ the subgroup of $W$
generated by the reflections $\{r_{\alpha}\}_{\alpha\in\Delta(\lambda)}$.
Note that for $w\in W(\lambda)$ one has
 $w\lambda\in \lambda+\mathbb{Z}\Delta$.
Since   $\ad\fg_{\alpha}$ acts locally nilpotently, one has $(\beta|\alpha^{\vee})\in\mathbb{Z}$ 
  for each $\beta\in\Delta$ and $\alpha\in\Delta_{re}\cap \Delta_0$. This gives
 $\Delta(\lambda)=\Delta(\lambda+\mu)$ for  $\mu\in\mathbb{Z}\Delta$ and, in particular,
$\Delta(w\lambda)=\Delta(\lambda)$ for $w\in W(\lambda)$.

%
%
%

\subsection{Category $\CO$ and the equivalence relation $\sim$}\label{cc}
We denote by $M_{\fg}(\lambda)$ (resp., $L_{\fg}(\lambda)$)
a Verma (resp., a simple) module of the highest weight $\lambda$ (usually
we won't distinguish
between the modules which differ by $\Pi$);
if $\fg$ is fixed, we will denote these modules
 by $M(\lambda),L(\lambda)$ respectively.

We denote by ${\CO}^{inf}(\fg)$
 the full category of $\fg$-modules $N$ with the
following properties:

(C1) $N$ is a semisimple $\fh$-module;

(C2) $\fn^+_0$ acts locally nilpotently on $N$.

The BGG-category $\CO(\fg)$  is  the full subcategory of
$\CO^{inf} (\fg)$ consisting of finitely generated modules.
The Verma modules lie in $\CO(\fg)$.
In~\cite{KK} the authors consider another version of
$\CO$-category, which we denote 
by $\CO_{KK}(\fg)$. One has
$$\CO(\fg)\subset \CO_{KK}(\fg)\subset {\CO}^{inf}(\fg).$$

Note that $\CO(\fg),{\CO}^{inf}(\fg)$ do not depend on the choice of 
triangular decomposition
in the following sense:  two compatible triangular decompositions define the same categories (the category $\CO_{KK}(\fg)$
 does not possess this property).

\subsubsection{}
Let $\CO$ be one of the categories $\CO(\fg),\CO_{KK}(\fg), {\CO}^{inf}(\fg)$.
One has 
$$\Irr(\CO)=\{L(\lambda)\}_{\lambda\in\fh^*}.$$ 
Consider
the graph, where the set of vertices is  $\fh^*$  and the number of edges $\lambda_1\to\lambda_2$
 is equal to $\dim \Ext^1(L(\lambda_1-\rho),L(\lambda_2-\rho))$. It is easy to see that this graph is the same for all three
categories. We  write $\lambda_1\sim\lambda_2$ 
if $\lambda_1,\lambda_2$ lie in the same connected component
of this graph.

By~\cite{DGK}, the modules in $\CO_{KK}(\fg)$ admit "local composition series",
which play a role of Jordan-H\"older series for the modules of infinite length. As a result,  the following conditions are equivalent:
\begin{itemize}
\item
$\lambda_1\sim\lambda_2$;
\item
$L(\lambda_1-\rho), L(\lambda_2-\rho)$ lie in the same block in $\CO_{KK}(\fg)$;
\item
$L(\lambda_1-\rho), L(\lambda_2-\rho)$ lie in the same block in $\CO(\fg)$.
\end{itemize}

\subsection{Definition of atypicality}
The following definition extends the usual notion of atypicality to
affine case.

\subsubsection{Definition}\label{atypdef}
We say that an iso-set $S\subset \Delta_1$ is {\em orthogonal to }
$\lambda$ if $(\lambda|\alpha^{\vee})=0$ for each $\alpha\in S$.
For $\lambda\in {\fh}^*$ we denote by
$\atyp(\lambda)$ the maximal cardinality of an iso-set orthogonal to $\lambda$. Clearly, 
$$0\leq \atyp\lambda\leq \ \defect\fg,\ \ \atyp 0=\defect \fg.$$
 We say that $\lambda$ is {\em typical}
if $\atyp(\lambda)=0$.

\subsubsection{}
 Using the formulae for $\Delta_{iso}$
given in~\ref{coreKMQ} it is easy to see that 
all maximal iso-sets orthogonal to $\lambda$ have the same cardinality.
Moreover,
$\nu\sim\lambda$ implies $\atyp\nu=\atyp\lambda$ (see~\Cor{coratyp}).
In the light of~\ref{cc} this allows to introduce the atypicality of 
 of an indecomposable module $N$ in $\CO_{KK}(\fg)$ via the formula
 $\atyp(N):=\atyp(\lambda)$ where
$[N:L(\lambda-\rho)]\not=0$. 

By~\ref{oddrefl}, the atypicality of $N\in\CO(\fg)$
does not depend on the triangular decomposition. Using this fact
it is easy to show that the atypicality of a one-dimensional module
is  equal to the defect of $\fg$.

\subsection{Main  results and content of the section}
The classical results in~\cite{BGG}, \cite{J} and~\cite{KK} 
determine the list of irreducible modules for each block in the category $\CO$
of a  symmetrizable Kac-Moody superalgebra (see~\cite{GSsnow}, 1.13.2 for details).
 For the $Q$-type superalgebras the same arguments 
determine these modules  up to a parity
shift. This give a description of the relation $\sim$, 
which we will call the "Kac-Kazhdan description"
(the reasoning in~\cite{KK} works equally well for supercase, while
some of the arguments~\cite{BGG} and \cite{J} should be modified  in the infinite-dimensional setting). We recall this description in~\ref{KKdescr} below. 
In~\Thm{propblocks} we show that  for $\lambda\in\mathring{\fh}^*$
one has 
\begin{equation}\label{nuWla}
\nu\sim\lambda\ \Longleftrightarrow \ \nu\in W(\lambda)(\lambda+\mathbb{Z}S_{\lambda}),\end{equation}
where $S_{\lambda}$ is a maximal iso-set orthogonal to $\lambda$.
(For $\dim\fg<\infty$, this formula was known in many cases, see, for example,~\cite{CCL}.)
Remark that, by~\cite{DS}, Lemma 6.1 and~\cite{CW}, Chapter II, 
for a finite-dimensional $\fg$ one has
\begin{equation}\label{nuZla}
\Ann_{\cZ(\fg)} L(\lambda-\rho)=\Ann_{\cZ(\fg)} L(\nu-\rho)\ \ \Longleftrightarrow 
\ \ \nu \in W(\lambda+\mathbb{C}S_{\lambda}).\end{equation}

In~\ref{simcrit} we  study the equivalence classes of
$\sim$ in  the set $\fh^*\setminus\mathring{\fh}^*$. 
In~\Cor{coratyp} we show that  $\nu\sim\lambda$ implies $\atyp\nu=\atyp\lambda$.
 In~\ref{coreKMQ}
we introduce $\Core(\lambda)$; in~\ref{centchar} we discuss the 
connection between $\Core(\lambda)$ and $\Ann_{\cZ(\fg)} L(\lambda-\rho)$.
By~\cite{DS},
 the functor $\DS_x$ preserves $\Core(\lambda)$ for $\fg=\fgl(m|n),\osp(m|2n)$
(in~\Thm{thmcore} we will check this statement for certain values of $x$
in the affine case).

\subsubsection{Conjecture}\label{conjKlapl}
The formula~(\ref{nuZla}) can be generalized to
the infinite-dimensional case for $\cZ(\fg)$  substituted 
by the algebra of Laplace operators introduced in~\cite{Kcentre}.

These operators might be useful for generalization of~\Thm{thmcore}
to other $x\in X_{iso}$.

\subsection{Kac-Kazhdan description of the relation $\sim$}\label{KKdescr} 
We introduce
$$\begin{array}{lll}
\cK_{red}:=&\{(\lambda,\lambda-m\alpha)\in \fh^*\times\fh^*|\ \alpha\in {\Delta}^+_{red},\ m\in \mathbb{N}_{>0}\ \text{ s.t. } (\lambda|\alpha^{\vee})=
m\}\\
\cK_{nis}:=&\{(\lambda,\lambda-m\alpha)\in \fh^*\times\fh^*|\ \alpha\in {\Delta}^+_{nis},\ m\in 2\mathbb{N}+1\ \text{ s.t. } (\lambda|\alpha^{\vee})=
m\}\\
\cK_{iso}:=&\{\ \ (\lambda,\lambda-\alpha)\in \fh^*\times\fh^*|\ \alpha\in {\Delta}^+_{iso},\ (\lambda|\alpha^{\vee})=0\},\\
\cK_{im}:=&\{\ \ (\lambda,\lambda-\delta)\in \fh^*\times\fh^*|\  (\lambda|\delta)=0\}.
\end{array}$$
We set 
$$\mathring{\cK}:=\cK_{red}\coprod\cK_{iso}\coprod\cK_{nis},\ \ \ \ 
\cK:=\mathring{\cK}\coprod \cK_{im}.$$

Take  $\lambda\in {\fh}^*$. Let
$M'(\lambda)$ be the maximal submodule of $M(\lambda)$ which satisfies the property $M'(\lambda)_{\lambda}=0$.  The factorization of Shapovalov determinants, obtained in~\cite{KK},\cite{Kconf} (\cite{Gq} for the $Q$-type)
gives: 
\begin{itemize}
\item
$M'(\lambda)_{\nu}\not=0$ if  $(\lambda,\nu)\in \cK$, 
\item
if $M'(\lambda)_{\mu}\not=0$, then  $(\lambda,\nu)\in \cK$ 
for some
$\nu\in\mu+\mathbb{N}\Delta^+$.
\end{itemize}
By the density arguments of~\cite{BGG}, Lemma 10 (see also~\cite{KK} and~\cite{Gq}, 11.4.4 for the $Q$-type) we have
$
\Hom(M(\nu), M(\lambda))\not=0\ \text{ if } (\lambda,\nu)\in \cK$.
Finally, the classical arguments of~\cite{J} (see also~\cite{KK}, Thm. 2)  give
\begin{equation}\label{KKdes}
\text{ \em{the  equivalence relation $\sim$  is generated by the set $\ \cK$.}}\end{equation}
Note that the restriction 
of the equivalence relation $\sim$ on $\mathring{\fh}^*$ is generated by the set 
$\mathring{\cK}$.

\subsection{Change of the base}\label{oddrefl}
For an odd simple root $\beta\in\Delta^+$ with $(\beta|\beta)=0$, we can construct a new subset of positive roots $r_{\beta}(\Delta^+)$ using so-called  {\em odd reflection}:
\begin{equation}\label{Drbeta}
r_{\beta}(\Delta^+)=(\Delta^+\setminus\{\beta\})\cup\{-\beta\};
\end{equation}
taking $\rho':=\rho+\beta$ we obtain a Weyl vector
for  $r_{\beta}(\Delta^+)$. 
Note that there are no odd reflections in the $Q$-type case. By~\cite{Sint}
any two compatible triangular 
decompositions  are connected by a chain of
odd reflections. Each module $L(\nu)$ is a  highest weight module with respect
to $r_{\beta}(\Delta^+)$ and the highest weight 
$\nu'$ is given by the formula
$$\nu'+\rho'=\left\{\begin{array}{ll}
\nu+\rho & \text{ if }(\nu|\beta)\not=0,\\
\nu+\rho+\beta  & \text{ if }(\nu|\beta)=0.
\end{array}\right.$$
Notice that  $(\nu'+\rho')\sim (\nu+\rho)$.

\subsubsection{}
\begin{cor}{corbasesigma}
 If  $\lambda'-\rho'$ is the highest weight of $L(\lambda-\rho)$ with respect to
a base $\Sigma'$ which is compatible with $\Sigma$, then 
$\lambda\sim \lambda'$.
\end{cor}


\subsection{Iso-sets}\label{isosethere}
Recall that for each $\beta\in\Delta_{iso}$ the root spaces $\fg_{\beta},\fg_{-\beta}$
generate a subalgebra isomorphic to $\fsl(1|1)$ and
$[\fg_{\beta},\fg_{-\beta}]=h_{\beta}$, where $\mu(h_{\beta})=(\mu|{\beta}^{\vee})$ for each $\mu\in\fh^*$.

\subsubsection{} 
\begin{lem}{lemrank2}
For  $\beta_1,\beta_2\in\Delta_{iso}$ one has
$$(\beta_1|{\beta}_2^{\vee})\not=0\ \ \Longleftrightarrow\ \ \beta_1+\beta_2\in\Delta_{re}
\text{ or }
\beta_1-\beta_2\in\Delta_{re}.$$
\end{lem}
\begin{proof}
The implication $\Longrightarrow$ follows from the facts that
$\fg_{\beta},\fg_{-\beta}, h_{\beta}$ span $\fsl(1|1)$. 
Assume that $\gamma:=\beta_1-\beta_2\in\Delta_{re}$.
For the KM-case $\beta\in\Delta_{iso}$ implies $(\beta|\beta)=0$, so 
$$0\not=(\gamma|\gamma)=-2(\beta_1|\beta_2)=-2(\beta_1|{\beta}_2^{\vee}).$$
For the $Q$-type case we have 
$\beta_1=\vareps_i-\vareps_j$ and $\beta_2=\vareps_i-\vareps_k$
or $\beta_2=\vareps_k-\vareps_j$ for some  $i\not=j\not=k$, so
$(\beta_1|{\beta}_2^{\vee})\not=0$.
\end{proof}

\subsubsection{}
\begin{cor}{isosetnow}
\begin{enumerate}
\item
A set
 $S\subset \Delta_1$ is an iso-set if and only if  for each $\alpha,\beta\in S$
one has  $(\alpha|{\beta}^{\vee})=0$  and
$\alpha\pm\beta\not\in \mathbb{Z}\delta$ for $\alpha\not=\beta$.

\item
For two iso-sets $S',S''$ of the same cardinality
there exists $w\in W$ such that $w(S'\cup (-S'))=S''\cup (-S'')$.

\item Let $S$ be an iso-set of the maximal cardinality.
For each $x\in X_{iso}$ (see~\ref{Xiso} for notation)
there exists an  automorphism $\phi\in \Aut(\fg)$
such that $\supp \phi(x)$ lies in  $-S\cup S$.
\end{enumerate}
      \end{cor}
\begin{proof}
\Lem{lemrank2} implies ``only if''  in (i). Fix $S\subset \Delta_1$  such that  for each $\alpha,\beta\in S$ one has $(\alpha|\beta^{\vee})=0$ and $\alpha\pm\beta\not\in\mathbb{Z}\delta$
for  $\alpha\not=\beta$.

For the $Q$-type case these conditions give
$S=w\{\vareps_{2i}-\vareps_{2i-1}\}_{i=1}^r$
for some $r<n$ and $w\in W$; thus $S$ is an iso-set; this also establishes (ii)
for this case.

In the remaining KM-case one has  $\beta^{\vee}=\beta$.
Let $\ol{S}\subset S$ be the maximal linearly independent subset of $S$. By~\Lem{lemrank2}, $\ol{S}$ is an iso-set. Using the description in~\cite{vdL} 
it is easy to check that  $\mathbb{C}\ol{S}\cap\Delta=-\ol{S}\cup \ol{S}$; this gives $S=\ol{S}$, so $S$ is an iso-set. This establishes (i).  
 For $\dim\fg<\infty$   (ii) was verified in~\cite{DS}; using~\cite{Reif}, Table I
one can easily check (ii)  case-by-case for affine $\fg$.

For (iii) let $\fh'$ be as in Section~\ref{sectisoset}. By~\cite{KP}, $\phi(\fh')\subset\fh$
for some (inner) automorphism $\phi\in \Aut(\fg)$. Now (iii) follows from
(ii).
\end{proof}

\subsubsection{Remark}
If $\fg$ is finite-dimensional, then $X=X_{iso}$, see~\cite{DS} and~\ref{G0orbitQ}. This does not hold for the affine case:
taking  
$x$ with $\supp(x)=\{\alpha,\alpha+\delta\}$ for $\alpha\in\Delta_{iso}$ 
we obtain $x\in {X}\setminus X_{iso}$.

\subsection{}
\begin{thm}{propblocks}
Fix $\lambda\in\mathring{\fh}^*$.
Let $S\subset\Delta^+$ be a maximal iso-set orthogonal to $\lambda$. 
One has
$$\nu\sim\lambda\ \Longleftrightarrow\ \ 
\nu\in W(\lambda)(\lambda+\mathbb{Z}S).$$
\end{thm}
\begin{proof}
The implication ``$\Longleftarrow$'' easily follows from
the formula $W(\lambda+\mathbb{Z}S)=W(\lambda)$.
In order to verify the implication ``$\Longrightarrow$'' it is enough to show that 
$$\nu\in
 W(\lambda)(\lambda+\mathbb{Z}S),\ \  (\nu,\nu-s\alpha)\in\cK\ 
\ \Longrightarrow \ \ (\nu-s\alpha)\in W(\lambda)(\lambda+\mathbb{Z}S).$$
Take $\nu,\alpha$ as above, that is
$$\nu=w\lambda'\ \text{ for }w\in W(\lambda),\  \lambda'\in \lambda+\mathbb{Z}S
\ \text{ and }\ (\nu,\nu-s\alpha)\in\cK.$$
By~\ref{Deltalambda}, one has $W(\lambda)=W(\nu)$.
 If $\alpha\in\Delta^+_{red}$, then 
$\alpha\in \Delta(\nu)$; if $\alpha\in\Delta^+_{nis}$, then
$2\alpha\in \Delta(\nu)$. In both cases one has  $r_{\alpha}\in W(\lambda)=W(\nu)$ and 
so $\nu-s\alpha=r_{\alpha}\nu$ lies in 
$ W(\lambda)(\lambda+\mathbb{Z}S)$
as required. Consider the remaining case $\alpha\in\Delta^+_{iso}$. Then
$s=1$, $(\nu|\alpha^{\vee})=0$, so 
$$\nu-s\alpha=w(\lambda'-\beta),\ \ \text{ where } \beta:=w^{-1}\alpha\in \Delta_{iso}.$$
Since $w\in W(\lambda)$,  it sufficies to verify that 
\begin{equation}\label{lambdaZS}
\lambda'-\beta\in W(\lambda)(\lambda+\mathbb{Z}S).\end{equation}

If $\beta\in (S\cup(-S))$,
 then $\lambda'-\beta$ lies in $\lambda+\mathbb{Z}S$ as required.  Assume that
$\beta\not\in (-S\cup S)$. Observe that
\begin{equation}\label{lambdaS}
0=(\nu|\alpha^{\vee})=(w\lambda'|({w\beta})^{\vee})= (\lambda'|\beta^{\vee}).
\end{equation}

If $S\cup\{\beta\}$ is an iso-set, then  $(S|\beta^{\vee})=0$ and~(\ref{lambdaS})  implies
$(\lambda|\beta^{\vee})=0$, so  $S\cup\{\beta\}$ is an iso-set
 orthogonal to $\lambda$, which
contradicts to the maximality of $S$. Hence  $S\cup\{\beta\}$ 
is not an iso-set. 
In the light of~\ref{isosetnow}, $S$ contains $\beta_1$ such that
either $\beta_1\in 
\mathbb{Z}\delta\pm \beta$ or $(\beta|\beta^{\vee}_1)\not=0$.
 Note that $\beta_1\in S$ implies 
\begin{equation}\label{lambdaS1}
(\lambda'|\beta^{\vee}_1)=0.\end{equation}
If $\beta_1\in 
\mathbb{Z}\delta\pm\beta$, then $\fg$ is affine and  
$(\lambda'|{\beta})=(\lambda'|\beta_1)=0$; this
 gives $(\lambda'|\delta)=0$ in contradiction to
$\lambda\in\mathring{\fh}^*$. Hence $(\beta|\beta^{\vee}_1)\not=0$.
 We claim that there exists $\gamma\in\Delta_{red}$ satisfying 
\begin{equation}\label{gammalambda}
 r_{\gamma}\lambda'=\lambda',\ \ \beta_1\in \{r_{\gamma}\beta,-r_{\gamma}\beta\}.
\end{equation}

Indeed, if $\fg$ is a $Q$-type, then $(\beta|\beta^{\vee}_1)\not=0$ gives
$\beta=\pm(\vareps_i-\vareps_j)$, 
$\beta_1=\pm(\vareps_i-\vareps_k)$
 for some  $i\not=j\not=k$. Combining~(\ref{lambdaS}) with (\ref{lambdaS1})
we obtain
$$(\lambda'|\vareps_i+\vareps_j)=(\lambda'|\vareps_i+\vareps_k)=0,$$
so~(\ref{gammalambda}) holds  for  $\gamma:=\vareps_j-\vareps_k$ 
(since $(\lambda'|\gamma)=0$).

Consider the KM-case. Notice that 
$(\mathbb{C}\beta+\mathbb{C}\beta_1)\cap \Delta$
is a generalized root system in the sense of~\cite{VGRS} which
is spanned by two non-orthogonal isotropic roots ($\beta$ and $\beta_1$).
 By~\cite{VGRS} this root system is  of the type $A(1|0), C(2)$ or $B(1|1)$; in each case
 it is easy to check the existence of $\gamma \in\Delta_{red}$ satisfying 
 $\beta_1=\pm r_{\gamma}\beta$.  Since $r_{\gamma}\lambda'=\lambda'-c\gamma$ for some $c\in\mathbb{C}$, combining~(\ref{lambdaS}) and (\ref{lambdaS1})
we obtain
$$0=(\lambda'|\beta_1)=(\lambda'|r_{\gamma}\beta)=(\lambda'-c\gamma|\beta),$$
so $c(\gamma|\beta)=0$.
By above,
 $r_{\gamma}\beta\not=\beta$ so $(\gamma|\beta)\not=0$
and thus $c=0$. This  establishes~(\ref{gammalambda}).

By~(\ref{gammalambda}),
$r_{\gamma}\in W(\lambda')=W(\lambda)$. Using $\lambda'\in \lambda+\mathbb{Z}S$ and $\beta_1\in S$ we get
$$\lambda'-\beta=r_{\gamma}(\lambda'\pm \beta_1)\in W(\lambda)(\lambda+\mathbb{Z}S)$$
as required.
\end{proof}

\subsection{The relation $\sim$ on $\fh^*\setminus\mathring{\fh}^*$}
\label{simcrit}
Consider the case when $\fg$ is affine. 
One has
$$\fh^*\setminus\mathring{\fh}^*=\{\lambda\in\fh^*|\ (\lambda|\delta)=0\}=
\dot{\fh}^*+\mathbb{C}\delta$$
(see~\ref{affro} for the notation).
We will write $\lambda\in \fh^*\setminus\mathring{\fh}^*$ in the form
$\dot{\lambda}+a_{\lambda}\delta$ for $\dot{\lambda}\in \dot{\fh}^*$
and $a_{\lambda}\in\mathbb{C}$. 
We denote by $\dot{W}$ the Weyl group of $\dot{\fg}$
and define $\dot{W}(\dot{\lambda})$ as in~\ref{Deltalambda}. 

\subsubsection{}
Take $\lambda\in \fh^*\setminus\mathring{\fh}^*$. 
By~\ref{KKdescr},  one has $\lambda\sim \lambda-\delta$.
Using~\Thm{propblocks} for $\dot{\lambda}$ we get
\begin{equation}\label{KKsimcrit}
\dot{\nu}\in \dot{W}(\dot{\lambda})(\dot{\lambda}+S_{\dot{\lambda}}),\ \ a_{\nu}\in a_{\lambda}+\mathbb{Z}\ \ \Longrightarrow\ \ \nu\sim\lambda\end{equation}
if $S_{\dot{\lambda}}\subset\dot{\Delta}^+$ is a maximal iso-set orthogonal to $\dot{\lambda}$.

\subsubsection{}
\begin{cor}{blockcrit} 
For the cases $\fg=\dot{\fg}^{(1)}$ with 
$\dot{\fg}=\fgl(m|n)$, $\osp(2m|2n)$, $D(2,1|a)$ or $F(4)$ one  has
$$\dot{\nu}\in \dot{W}(\dot{\lambda})(\dot{\lambda}+S_{\dot{\lambda}}),\ \ a_{\nu}\in a_{\lambda}+\mathbb{Z}\ \ \Longleftrightarrow\ \ \nu\sim\lambda.$$
\end{cor}
\begin{proof} 
For the implication ``$\Longrightarrow$'' note that
$$\Delta_{red}=\dot{\Delta}_{red}+\mathbb{Z}\delta,\ \ \ \ \Delta_{iso}=\dot{\Delta}_{iso}+\mathbb{Z}\delta,\ \ \ \Delta_{nis}=\emptyset.$$
Take $\dot{\nu}\in
 \dot{W}(\dot{\lambda})(\dot{\lambda}+\mathbb{Z}S_{\dot{\lambda}})$ and 
$(\nu,\mu)\in\cK$. Using the notation of~\ref{KKdescr}, we have $\mu=\nu-s\alpha$ for $\alpha\in\Delta$. Then
$(\nu|\alpha)=(\dot{\nu}|\dot{\alpha})$ and $\dot{\mu}=\dot{\nu}-s\dot{\alpha}$.
Using~\Thm{propblocks} for $\dot{\lambda}$ we obtain 
$\dot{\mu}\in \dot{W}(\dot{\lambda})(\dot{\lambda}+\mathbb{Z}S_{\dot{\lambda}})$
as required.
\end{proof}

\subsubsection{}
Take $\tilde{W}\subset GL(\dot{\fh}^*)$ generated by $r_{\dot{\alpha}}$
for all even real roots $\alpha$ (one has $\tilde{W}=\dot{W}$ if
$\fg=\dot{\fg}^{(1)}$). From~\cite{vdL} we see that 
$\dot{\Delta}_{iso}$ is $\tilde{W}$-invariant.
Arguing as in the proof above, it is easy to show that
\begin{equation}\label{abdelta}
 \ \nu\sim\lambda  \ \Longrightarrow\ \ 
\dot{\nu}\in \tilde{W}(\dot{\lambda}+\mathbb{Z}S_{\dot{\lambda}}),\ \ a_{\nu}\in a_{\lambda}+\mathbb{Z}.
\end{equation}

%
%
%

\subsection{}
\begin{cor}{coratyp}
  $\nu\sim\lambda \ \Longrightarrow\ \ \atyp\nu=\atyp\lambda$.
\end{cor}
\begin{proof}
Since $\sim$ is an equivalence relation, it is enough to show that
for $\nu\sim\lambda$ one has  $\atyp\nu\geq \atyp\lambda$. 
Let $S$ be a maximal iso-set orthogonal to $\lambda$. 
If $\lambda\in \mathring{\fh}^*$, then, by~\Thm{propblocks}, 
$\nu=w(\lambda+\mu)$, where
$\mu\in\mathbb{Z}S$ and $w\in W(\lambda)$. Since
$wS$ is an iso-set orthogonal to $\nu$, we obtain $\atyp\nu\geq \atyp\lambda$
as required. Similarly, 
for $\lambda\not\in \mathring{\fh}^*$, the required inequality
 follows from~(\ref{abdelta}) and the fact that $\dot{\Delta}_{iso}$ is $\tilde{W}$-invariant.
\end{proof}

\subsection{Cores}\label{coreKMQ}
We introduce $\Core(\lambda)$ in the cases when $\fg$ is not exceptional. For $\fg=\fgl(m|n),\osp(M|2n)$, our definition 
is similar to one used in~\cite{GS}.  

We will consider the  root systems of the types 
$$\begin{array}{ll}
\text{ finite}: & \fgl(m|n),\  B(m|n)=\osp(2m+1|2n), D(m|n)=\osp(2m|2n),\  \fq_m;\\
\text{ affine}: & \fgl(m|n)^{(1)},\ \osp(M|2n)^{(1)},\  A(M|2n)^{(2)}, \ 
A(2m|2n)^{(4)},\  D(m|n)^{(2)}.\end{array}$$ 
For the finite root systems
   we will use the standard notation of~\cite{KLie}.
For affine case we retain notation of~\ref{affro}; note that
  the finite root system $\dot{\Delta}$ is of the $\fgl(m|n)$-type for
$\fgl(m|n)^{(1)}$ and of  the $\osp$-type
 for the rest of the cases. In all cases we set
$$a_i:=(\lambda|\vareps_i),\ \ \ \ b_j:=(\lambda|\delta_j)\ \ \text{ for }
i=1,\ldots,m,\ \ \ j=1,\ldots,n.$$

\subsubsection{Case $\fgl(m|n)$}\label{Amn}
In this case
$\Delta_{iso}=\{\pm (\vareps_i-\delta_j)\}_{i=1,\ldots,m}^{j=1,\ldots,n}$.

\begin{defn}{}
For $\lambda\in\fh^*$ let $\Core(\lambda)$ be the multiset
obtained from $\{a_i\}_{i=1}^m\coprod \{b_j\}_{j=1}^n$ by deleting the maximal number of pairs satisfying $a_i=b_j$.
\end{defn}

For instance, for $\lambda=\vareps_1+\vareps_2+\vareps_3-\delta_1-2\delta_2$ we 
have $a_1=a_2=a_3=1$, $b_1=1$, $b_2=2$ and 
$\Core(\lambda)=\{1,1\}\coprod\{2\}$.

\subsubsection{Cases $B(m|n),D(m|n)$}
In this case
$\Delta_{iso}=\{\pm \vareps_i\pm\delta_j\}_{i=1,\ldots,m}^{j=1,\ldots,n}$.

\begin{defn}{}
For $\lambda\in\fh^*$ let $\Core(\lambda)$ be the multiset
obtained from $\{a_i^2\}_{i=1}^m\coprod \{b_j^2\}_{j=1}^n$ by deleting the maximal number of pairs satisfying $a_i^2=b_j^2$.
\end{defn}

\subsubsection{Case $\fq_m$}\label{Qcore}
In this case $\Delta_{iso}=\{\vareps_i-\vareps_j\}_{1\leq i\not=j\leq m}$.

\begin{defn}{}
For $\lambda\in\fh^*$ we  
let $\Core(\lambda)$ be the multiset
obtained from $\{a_i\}_{i=1}^n$ by deleting the maximal number of pairs satisfying $a_i+a_j=0$ for $i\not=j$.
\end{defn}
 
For example, for  $\lambda=\vareps_1+\vareps_2-\vareps_m$
we have $\Core(\lambda)=\{0,1\}$ if $m$ is even, $\Core(\lambda)=\{1\}$ if $m$ is odd.

\subsubsection{Case $\fgl(m|n)^{(1)}$}
In this case
$\Delta_{iso}=\{\mathbb{Z}\delta\pm (\vareps_i-\delta_j)\}_{i=1,\ldots,m}^{j=1,\ldots,n}$.

\begin{defn}{}
Take $\lambda\in{\fh}^*$ and set $k:=(\lambda|\delta)$.
We let $\Core(\lambda)$ be the multiset
obtained from $\{a_i\}_{i=1}^m\coprod \{b_j\}_{j=1}^n$ by deleting the maximal number of pairs satisfying $a_i-b_j\in\mathbb{Z}k$.
We view the elements of the multiset $\Core(\lambda)$ as 
elements in $\mathbb{C}/\mathbb{Z}k$.
\end{defn}

\subsubsection{Cases $B(m|n)^{(1)}, D(m|n)^{(1)}, A(2m|2n-1)^{(2)}, A(2m-1|2n-1)^{(2)}$}\label{coreaff}
In this case
$$\Delta_{iso}=\{\mathbb{Z}\delta\pm \vareps_i\pm\delta_j\}_{i=1,\ldots,m}^{j=1,\ldots,n}.$$

\begin{defn}{}
Take $\lambda\in{\fh}^*$ and set $k:=(\lambda|\delta)$.
Let $\Core(\lambda)$ be the multiset
obtained from $\{a_i\}_{i=1}^m\coprod \{b_j\}_{j=1}^n$ by deleting the maximal number of pairs satisfying $a_i\pm b_j\in\mathbb{Z}k$.

We view the elements of the multiset $\Core(\lambda)$ as 
elements in the set $\mathbb{C}$ modulo 
the action of the group $\mathbb{Z}\rtimes\mathbb{Z}_2$ 
generated by $z\mapsto z+k$ and  $z\mapsto -z$.
\end{defn}

\subsubsection{Case
 $A(2m|2n)^{(4)}, D(m|n)^{(2)}$}\label{Amn4}
In this case 
$\Delta_{iso}=\{2\mathbb{Z}\delta\pm \vareps_i\pm\delta_j\}_{i=1,\ldots,m}^{j=1,\ldots,n}$
and  the  definition can be obtained from~\ref{coreaff} above by substituting  $k$ by $2k$.

\subsubsection{Remark: degenerate cases}
In the case when $mn=0$ the set $\Delta_{iso}$ is empty, but we
use the same definition for $\Core(\lambda)$. For instance,
$\osp(2|0)=\mathbb{C}$; for this algebra 
${\fh}^*=\mathbb{C}\vareps_1$ and we set $\Core(a\vareps_1)=a^2$. Note that
$$\begin{array}{cl}
& \fgl(0|0)=\osp(0|0)=\osp(0|0)=\fq_0=0,\\
& \fgl(0|0)^{(1)}=\osp(0|0)^{(1)}=\osp(0|0)^{(1)}=\mathbb{C}K\times\mathbb{C}d; 
\end{array}$$
in all these cases we set $\Core(\lambda)=\emptyset$ for each weight $\lambda\in\fh^*$.

\subsubsection{}\label{corefin}
Take $\lambda\in {\fh}^*$. Let $S_{\lambda}$ be a maximal iso-set orthogonal to $\lambda$. It is easy to see that
\begin{equation}\label{coreABQ}\begin{array}{ll}
W(\lambda+\mathbb{C}S_{\lambda})\subset \{\nu|\ \Core(\nu)=\Core(\lambda)\},\\
W(\lambda+\mathbb{C}S_{\lambda})= \{\nu|\ \Core(\nu)=\Core(\lambda)\} & \ \text{ for }
A(m|n), B(m|n), \fq_m.\end{array}
\end{equation}
For $D(m|n)$ the equality holds if $\atyp\lambda\not=0$.

 \subsubsection{}
  \begin{cor}{corcore}
\begin{enumerate}
    \item
The cardinality of $\Core(\lambda)$ is equal to $m+n-2\atyp\lambda$
(resp., $m-2\atyp\lambda$)
for the KM-case (resp., 
 for the $Q$-type case).

\item For $\lambda,\nu\in {\fh}^*$ 
one has $\Core(\lambda)=\Core(\nu)$ if $\nu\sim\lambda$.

\item Let $\fg$ be finite-dimensional. If $\lambda'-\rho'$
is the highest weight of the module $L(\lambda-\rho)$
with respect to a base $\Sigma'$, which is compatible to $\Sigma$, then
  $\Core(\lambda)=\Core(\lambda')$. For affine case this holds
if $\dot{\Delta}=\dot{\Delta'}$.
 \end{enumerate}
\end{cor}
 \begin{proof}
 For (i) note that in all cases~\ref{Amn}--\ref{Amn4}   the number 
 of deleted pairs
is equal to  $\atyp\lambda$. Combining~\Thm{propblocks} with~(\ref{coreABQ})
 we get (ii); (iii) follows from~\ref{oddrefl}. 
\end{proof}

\subsubsection{Remark}
For affine case the construction of 
$\Core(\lambda)$  depends on the choice of $\dot{\Sigma}$. Using the tables 
in~\cite{GKadm}, 13.2, it is not hard to check that, apart from the case
$\fgl(m|n)^{(1)}$ 
$\Core(\lambda)$  does not depend on the choice of $\dot{\Sigma}$.

\subsection{Central characters}\label{centchar}
Consider the case when $\dim\fg<\infty$. 

For $z\in\cZ(\fg)$
let $\mathcal{HC}(z)\in\cS(\fh)$ be such that $z$ acts on $M(\lambda-\rho)$ by
$\mathcal{HC}(z)Id$. Then 
$$\mathcal{HC}:\cZ(\fg)\to \cS(\fh)$$
is an algebra monomorphism ($\mathcal{HC}$ is the composition of the Harish-Chandra projection
$\HC$ with $\rho$-twisting; by contrast with $\HC$, the map $\mathcal{HC}$
does not depend on the triangular decomposition of $\Delta$). 
The arguments of~\cite{Kcentre} (see also~\cite{Gq})  imply
that the image of $\mathcal{HC}$ is given by
$$
Z(\fg):=\{f\in\cS(\fh)|\ f(\lambda)=f(\nu)\ \text{ if }
(\nu,\lambda)\in \cK\}.$$

Take any $\beta\in\Delta_{iso}$. It is easy to see that $\Delta_{iso}=W\beta\cup W(-\beta)$; this allows to rewrite the above formula as
\begin{equation}\label{Zn} \begin{array}{rl}
Z(\fg)=&\{f\in\cS(\fh)^W|\ \forall \lambda\in\fh^* \text{ with } (\lambda|\beta^{\vee})=0\\
& \text{ one has }
\  f(\lambda)=f(\lambda-c\beta)\ \text{ for }c\in\mathbb{C}\}.
\end{array}\end{equation}

(The above formulae were established  earlier by A.~Sergeev in~\cite{Sinv},\cite{Sq} by other methods; see also~\cite{GN} for another approach in the $Q$-type case).

\subsubsection{}\label{corechi}
Let $\chi_{\lambda}:\cZ(\fg)\to\mathbb{C}$ be
the  central character of $L(\lambda-\rho)$, i.e. 
$$\chi_{\lambda}(z):=\mathcal{HC}(z)(\lambda).$$
By~\ref{Ian}, each central character $\chi\in\mspec_{\Mod}\cZ(\fg)$ is of the form $\chi=\chi_{\lambda}$ for some $\lambda\in\fh^*$. 
We set $\atyp\chi_{\lambda}:=\atyp\chi$. If $\fg$ is not exceptional we set
$\Core(\chi_{\lambda}):=\Core(\lambda)$.
These notions are well-defined;
 moreover, using~\ref{corefin},
it is not hard to see that  the assignment $\chi\mapsto\Core(\chi)$
 is bijective for $\fg\not=\osp(2m|2n)$. 
We give a proof for the $Q$-type cases in~\Prop{propchi} below
 (the proof is similar to one for $\fq_n$ in~\cite{CW}, Ch. II).
 For $\fg=\osp(2m|2n)$
the fiber is of  the form $\{\chi,\sigma(\chi)\}$, where $\sigma$ is as in~\ref{isosetbasic1}; one has 
$\sigma(\chi)=\chi$ if $\atyp\chi\not=0$
(see~\cite{CW}, Ch. II for the details).

\subsubsection{}
\begin{prop}{propchi}
Let $\fg$ be a $Q$-type algebra.
For $\lambda,\nu\in\fh^*$ one has $\chi_{\lambda}=\chi_{\nu}$ if and only if
$\Core(\lambda)=\Core(\nu)$.
\end{prop}
\begin{proof}
Our goal is to show that
$$\Core(\lambda)=\Core(\nu)\ \ \Longleftrightarrow\ \ \forall f\in Z(\fg)\ \ 
f(\lambda)=f(\nu).$$
The implication $\Longrightarrow$ follows from~(\ref{coreABQ}).
For the implication $\Longleftarrow$ we introduce the following notation. For  $\fq_n$ and $\fsq_n$ we 
denote by $\{h_i\}_{i=1}^n$ the basis of $\fh$ which is dual to the basis of $\{\vareps_i\}_{i=1}^n$ in $\fh^*$;
for $\fpq_n$ and $\fpsq_n$ we denote by the same symbol
the image of $h_i$ in $\fh$. Consider the symmetric polynomials 
$$p_{2r+1}:=\sum_{i=1}^n h_i^{2r+1}$$
(one has $p_1=0$ for $\fpq_n,\fpsq_n$).
From~(\ref{Zn}) one has $p_{2r+1}\in Z(\fg)$. We set
$$\Phi:=\displaystyle\sum_{r=1}^{\infty} p_{2r+1}z^r=\sum_{i=1}^n \frac{h_i}{1-h_i^2z}.$$
For each $\mu=\sum_{i=1}^n a_i\vareps_i\in\fh^*$  we introduce the function
$$\phi(z):=\Phi(\mu)=\sum_{i=1}^n \frac{a_i}{1-a_i^2z}.$$
Clearly, $\phi(z)$ 
is a meromorphic function in $z$ with the set of poles equal to
$$\{a^{-2}|\ a\in\Core(\mu)\setminus\{0\}\}.$$ 
Moreover, the residue of $\phi(z)$ at the point $a^{-2}$
is equal to $\mult(a)a^3$, where $\mult(a)$ stands for the multiplicity of $a$
in the multiset $\Core(\mu)$.

Let $\nu$ be such that  $p_{2r+1}(\lambda)=p_{2r+1}(\nu)$ for each $r$.
Then $\Phi(\lambda)=\Phi(\nu)$. By above, this gives
$\Core(\lambda)\setminus\{0\}=\Core(\nu)\setminus\{0\}$.
Since the multiplicity of $0$ in the multisets
$\Core(\lambda),\Core(\nu)$ is at most $1$
and 
$$\#\Core(\lambda)\equiv n\equiv \#\Core(\nu)\ \ \mod 2$$
we obtain $\Core(\lambda)=\Core(\nu)$ as required.
\end{proof}

 \subsection{Useful fact}\label{usefulq}
The following fact will be used later. Let $\fg$ be a Kac-Moody superalgebra
or $\fq_n$. Fix $\Sigma'\subset\Sigma$ and choose
$h\in\fh$ such that $\alpha(h)=0$ for $\alpha\in\Sigma'$
and $\alpha(h)=1$ for $\alpha\in \Sigma\setminus\Sigma'$. 
The algebra
$\fg^h$ has the triangular decomposition:
$$\fg^h=\fg^{\fh}\oplus (\fn\cap \fg^h)\oplus (\fn^-\cap\fg^h)$$
( the corresponding set of simple roots is $\Sigma'$).

\begin{lem}{lemusefulq}
Take $\lambda\in\fh^*$ and denote by $L$ the module $L(\lambda)$ viewed as a $\fg^h$-module. Then
$$L':=\sum_{\nu\in\mathbb{N}\Sigma'} L_{\lambda-\nu}$$
is a simple $\fg^h$-module of the highest weight $\lambda$ and
$L'$  is a direct summand
of $L$.
\end{lem}
\begin{proof}
Each $\fg^h$-module $N$ with a diagonal action of $\fh$ can be decomposed as 
$$N=\displaystyle\bigoplus_{\mu\in\fh^*/\mathbb{Z}\Sigma'} N_{\mu},$$
where each $N_{\mu}$ is a $\fg^h$-submodule.
This shows that $L'$ is a direct summand of $L$.
Let $v\in L'_{\lambda-\nu}$ be a $\fg^h$-primitive vector, i.e.
 $(\fg^h\cap \fn)v=0$. Take $\alpha\in\Delta^+\setminus \Delta(\fg_h)$.
 Since $\nu\in \mathbb{N}\Sigma'$
 one has 
 $$\lambda-\nu+\alpha\not\in \lambda-\mathbb{N}\Sigma',$$
 so $\fg_{\alpha} L(\lambda)_{\lambda-\nu}=0$. Hence $v$ is a $\fg$-primitive, that is 
 $v\in L'_{\lambda}$. However, $L'_{\lambda}=L(\lambda)_{\lambda}$ is a simple
 $\fg^{\fh}$-module and 
$\fg^{\fh}\subset \fg^h$.
 Hence $L'$ is simple.
 \end{proof}

\section{$\DS$-functor for $Q$-type algebras}\label{sectQ}
Queer ($Q$-type) Lie superalgebras are the superalgebras
$\fq_n,\fsq_n,\fpq_n$ and $\fpsq_n$. Recall that
 $\fq_n$ is a subalgebra of $\fgl(n|n)$ consisting of the matrices with the block form
$$T_{A,B}:=\begin{pmatrix}
A & B\\
B & A
\end{pmatrix}$$
The centre of $\fq_n$ is spanned by the identity matrix, which we denote by $z$. The commutant
  $\fsq_n:=[\fq_n,\fq_n]$ 
 is a subalgebra of  $\fq_n$ consisting of the matrices $T_{A,B}$ with
 $Tr B=0$. One has  $\fpq_n:=\fq_n/\mathbb{C}z$ and
 $\fpsq_n$ is the image of $\fsq_n$ in $\fpq_n$. The algebra
 $\fpsq_n$ is simple for $n\geq 3$.
The group $GL_n$ acts on these algebras by the inner
automorphisms $g.T_{A,B}:=T_{gAg^{-1},gBg^{-1}}$.

We retain notation of~\ref{notation} and Section~\ref{sectKMQ}.
We denote by $\cB_0(\fg)$ the principal block in $\Fin(\fg)$, which is the block
 containing the trivial module.

\subsection{Main results}\label{mainresultq}
For $\fg=\fq_n,\fsq_n,\fpq_n,\fpsq_n$
we have 
$$\defect \fg=\depth(\fg)=[\frac{n}{2}]$$  
For $x\in X(\fq_n)_r$  we have
 $\DS_x(\fq_n)\cong \fq_{n-2r}$. 
The similar formulae hold for $\fsq_n,\fpq_n,\fpsq_{n}$ 
if $r\not=[\frac{n}{2}]$; for  $r=[\frac{n}{2}]$ we have 
$$\begin{array}{l}
\DS_x(\fsq_n)\cong \fsq_1=\mathbb{C},\ \ \DS_x(\fpq_n)\cong \fpq_1=\Pi(\mathbb{C})\\
\DS_x(\fpsq_n)= \fpsq_1=0\ \ \text{ if $n$ is odd},\ \ 
\DS_x(\fpsq_n)\cong \mathbb{C}\times\Pi(\mathbb{C}) \ \ \text{ if $n$ is even}. \end{array}$$

One has  $X(\fq_n)_{iso}=X(\fsq_n)_{iso}$
and $X(\fpq_n)_{iso}=X(\fpsq_n)_{iso}$ is the image of $X(\fq_n)_{iso}$.
Moreover, $X(\fg)_{iso}=X(\fg)$ for $\fg=\fq_n,\fsq_n$ and
$X(\fg)\setminus X(\fg)_{iso}\subset X(\fg)_{[\frac{n}{2}]}$
for $\fg=\fpq_n,\fpsq_n$.  

 In~\Cor{DStheta} we prove that 
 the map $\theta_x^*$
preserves the core of a central character,
 increases the
atypicality  by $r$ and that the  image of $\theta_x^*$ consists of
the central characters of atypicality $\geq r$. This implies that
$\DS$ commutes with translation functors (see~\ref{DStransQ}).
We also show  that $\depth\chi$
coincides with the atypicality of $\chi$. In~\Prop{propthetaq} we prove that
$\theta_x$ is surjective.

View $\fpsq_3$ as a quotient of the adjoint representation
for $\fq_3$. This is a simple module 
which lie in the principal block
$\cB_0(\fq_3)$. Since $\DS_x(\fpsq_3)=0$ for all
non-zero $x\in X(\fq_3)$, this module has zero depth; this gives
 an example of a   
of zero depth module which is not projective (by~\cite{DS}, for finite-dimensional
Kac-Moody superalgebras, $N\in\Fin(\fg)$ has zero depth if and only if $N$ is projective).
 Note that $\depth(\cB_0(\fq_3))=1$, 
since $\depth$ of
the trivial module is equal to $\depth(\fg)$.

\subsection{Questions and content of the section}
Let $L$ be a simple finite-dimensional $\fg$-module.

\subsubsection{Question}
Is it true that $\DS_x(L)\cong \DS_y(L)$ if $\rank x=\rank y$? More general, is it true
that $\DS_x(L)\cong \DS_{y'}(\DS_y(L))$ if $\rank x=\rank y+\rank y'$?

If $\dim L=\infty$, this does not hold, see~(\ref{CpiC}).

\subsubsection{Question}
Is it true that $\DS_x(L)$ is semisimple for $\fg=\fq_n$?

This property holds for finite-dimensional Kac-Moody superalgebras
(see~\cite{HW} and~\cite{GH}) and does not hold for $\fsq_2$ and for $\fpq_2$.

\subsubsection{}
In~\ref{notq} we introduce additional notation.
In~\ref{q1},\ \ref{q12} we consider the cases $n=1,2$.
In~\ref{DSq}--\ref{thetaqn} we study the case $\fq_n$. Finally, 
in~\ref{qrelatives} we describe $X(\fg)$ and $\DS_x(\fg)$ for
other $Q$-type algebras.

\subsection{Notation}\label{notq}
Let $\fg$ be one of the $Q$-type algebras. 
We denote by $\fh$ the standard Cartan subalgebra of $\fg_0$ (consisting of $T_{A,0}$
with diagonal $A$).
 The sets $\Delta_0=\Delta_1\subset\fh^*$ coincide with the root system of $\fgl_n$.
All triangular decompositions are $GL_n$-conjugated.
We fix a triangular decomposition:
$\fg=\fg^{\fh}\oplus \fn\oplus\fn^-$,
where $\fg^{\fh}$ is the Cartan subalgebra (one has $\fg^{\fh}\cap \fg_0=\fh$).
We use the standard $\fgl_n$-notation for the root system $\Delta_0$
and take
$\Sigma:=\{\vareps_i-\vareps_{i+1}\}_{i=1}^{n-1}$.
We  identify the Weyl group
$W$ with the permutations of $\{\vareps_i\}_{i=1}^n$.
 For $\fg=\fq_n,\fsq_n$ we set
$$h_i:=\vareps_i^*\in\fh;$$
 for $\fpq_n,\fpsq_n$ we denote by
$h_i$ the image of the corresponding element in $\fh$. 
One has
$$\displaystyle\sum_{i=1}^n h_i= \left\{
\begin{array}{lll}z\ & \text{ for }& \fq_n,\ \fsq_n,\\
0 & \text{ for }  & \fpq_n, \fpsq_n.
\end{array}\right.$$

 We denote by $\iota$ the canonical map
$\fq_n\twoheadrightarrow\fpq_n$ and its restriction $\fsq_n\twoheadrightarrow\fpsq_n$.

\subsubsection{Remark}\label{Fintyp}
Let $L(\lambda)$ be a typical finite-dimensional module and 
$\Fin(\chi_{\lambda})$ be the full subcategory of $\Fin(\fg)$ 
which corresponds to the central character $\chi_{\lambda}$
(one has $L(\lambda)\in \Fin(\chi_{\lambda})$). By~\cite{P}, the category $\Fin(\chi_{\lambda})$
is equivalent to
 the finite-dimensional representation of a Clifford algebra
 $\mathcal{C}l(\lambda):=\cS(\fg^{\fh})/\cS(\fg^{\fh})\Ker\lambda$.
For  instance, for $\fq_n$-case,  the category $\Fin(\chi_{\lambda})$ 
is equivalent to the category of the finite-dimensional modules over
 the algebra $\mathbb{C}[\xi]/(\xi^2)$
  if $0\in\Core(\lambda)$ and  $\Fin(\chi_{\lambda})$  is semisimple if
$0\not\in\Core(\lambda)$.

\subsection{Case $n=1$}\label{q1}
We have $\fq_1=\mathbb{C}z+\mathbb{C}H$, where $H$ is odd and
 $[H,H]=z$. One has
 $\fsq_1=\mathbb{C}z\cong \mathbb{C}$ and $\fpsq_1=0$; the algebra
$\fpq_1\cong \Pi(\mathbb{C})$ is spanned  $\iota(H)$.
In all these cases ${X}_{iso}=0$, so $\depth \fg=0$. One has
 $X(\fq_1)=X(\fsq_1)=0$ and $X(\fpq_1)=\fpq_1$.

The action of $z$ decomposes 
the category $\Fin(\fq_1)$ into blocks $\cB_c$ with $c\in\mathbb{C}$. 
All these blocks  are typical.
For $c\not=0$ the block $\cB_c$ is semisimple with
 $\Irr(\cB_c)=\{L_{\fq_1}(c)\}$, where $L_{\fq_1}(c)\cong \Pi(L_{\fq_1}(c))$ 
and $\dim L_{\fq_1}(c)=(1|1)$.
For the principal block $\cB_0$ 
we have  
$$\Irr(\cB_0)=\{L_{\fq_1}(0),\Pi(L_{\fq_1}(0))\},\ \ \Ext^1(L_{\fq_1}(0),\Pi(L_{\fq_1}(0))=\mathbb{C}.$$

\subsection{Case $n=2$}\label{q12}
The modules $L(\lambda)$ are explicitly described in~\cite{P}).
The module $L(\lambda)$   is
finite-dimensional if and only if
$\lambda(h_1-h_2)\in\mathbb{Z}_{>0}$
or $\lambda=0$. The module $L(\lambda)$ is atypical if and only if $\lambda(z)=0$ (for $\fpq_2,\fpsq_2$ all modules are atypical).
In~\ref{DSq}, \ref{DSsq} we will see that for $x\not=0$ one has
$\DS_x(N)=0$ for a typical module $N$.

Set $\fg:=\fq_2$.
Recall that
$\fg_0=\mathbb{C}z\times \fsl_2$; we fix the standard basis
$f,h,e$ in $\fsl_2$. Note that
$\fg_1$ has a basis 
$H_1:=T_{0,\Id},H_2, E,F$, where $H_2\in\fsq_2^{\fh}$ is such that
$[H_1,H_2]=2h$ and
$E=\frac{1}{2}[H_2,e]$, $F=\frac{1}{2}[f,H_2]$.
 One has $[\fg_0,H_1]=0$ and
$$
[H_1,H_1]=[H_2,H_2]=2z,\ \ [h,E]=2E,\ \ [h,F]=-2F,\ \ [E,F]=z, [E,E]=[F,F]=0.$$
  We will use the same notation for the images of these elements 
  in $\fpq_2$ and in $\fpsq_2$.
  
\subsubsection{Typical case}
Let $\lambda$ be such that $\lambda(z)\not=0$. Then
 $L_{\fq_2}(\lambda), L_{\fsq_2}(\lambda)$ are typical and
$L_{\fq_2}(\lambda)=L_{\fsq_2}(\lambda)$ as $\fsq_2$-module. One has
$\Pi(L_{\fsq_2}(\lambda))\cong L_{\fsq_2}(\lambda)$.
The typical blocks in $\Fin(\fsq_2)$ are semisimple.
By~\ref{DSq} below, for each typical $\fq_2$-module we have
$\DS_x(N)=0$ for $x\not=0$. This gives
$\DS_x(L_{\fsq_2}(\lambda))=0$ and thus $\DS_x(N)=0$
if $N$ is a typical module of a finite length.

\subsubsection{Cases $\fq_2,\fsq_2$}\label{sq2}
One has 
$$X(\fsq_2)=X(\fsq_2)_{iso}=X(\fq_2)=X(\fq_2)_{iso}=GL_2 E\cup \{0\}.$$
One has $\DS_E(\fq_2)=0$;
the algebra $\DS_E(\fsq_2)$ can be identified with $\mathbb{C}e$.

Let  $L:=L_{\fsq_2}(\lambda)$ be atypical. Then
 $L$  is a simple $\fsl_2$-module (so $\fg_1 L=0$). In particular, 
 $EL=0$, so 
 $\DS_E(L)=L$ as a  $\DS_E(\fsq_2)$-module (the action of $e\in \DS_E(\fsq_2)$  on $\DS_E(L)$ 
coincides with the action of $e$ on $L=L_{\fsl_2}(\lambda)$).

Let $L:=L_{\fq_2}(\lambda)$ 
be an atypical module for
$\lambda\not=0$.  As a $\fsq_2$-module, $L$ is a non-splitting extension of
 $\Pi(L_{\fsq_2}(\lambda))$ by $L_{\fsq_2}(\lambda)$. If $L$ is finite-dimensional, then
$$
\DS_x(L)\cong \mathbb{C}\oplus\Pi(\mathbb{C})\ \text{ for }x\in X_{iso}\setminus\{0\}.$$
If $L$ is  infinite-dimensional, then 
\begin{equation}\label{CpiC}
\DS_E(L_{\fq_2}(\lambda))\cong \mathbb{C},\ \ \ \ \ \DS_F(L_{\fq_2}(\lambda))\cong \Pi(\mathbb{C}).\end{equation}


\subsubsection{Case $\fpq_2$}
One has 
$$X(\fpq_2)_{iso}=X(\fq_2),\ \ \ X(\fpq_2)\setminus X(\fpq_2)_{iso}= \mathbb{C}^* \iota(H_1)\coprod \mathbb{C}^* \iota(H_2);$$
for each non-zero $x\in X(\fpq_2)$ one has $\rank x=1$ and 
$\DS_x(\fpq_2)\cong \Pi(\mathbb{C})\cong \fpq_1$.

The $\fpq_2$-modules are $\fq_2$-modules annihilated by $z$ (in particular,
 $\CO(\fpq_2)$ corresponds to the subcategory of atypical modules
in $\CO(\fq_2)$). Take $L:=L(\lambda)$ for $\lambda\not=0$.

The algebra  $\DS_E(\fpq_2)$ can be identified with
$\mathbb{C}F$. If $L$ is finite-dimensional and $\lambda\not=0$,
 then $\dim\DS_x(L)=(1|1)$ for each $x\in X(\fpq_2)_{iso}$; it is easy to check that
 $\DS_E(L)$ is an indecomposable $\fpq_1$-module if and only if $\dim L=(2|2)$.
If $L$ is infinite-dimensional, then $\DS_E(L),\DS_F(L)$ are given by~(\ref{CpiC})
(with $\fpq_1\cong \Pi(\mathbb{C})$ acting by zero).

For $x\in \mathbb{C}^* \iota(H_1)\coprod \mathbb{C}^* \iota(H_2)$
the algebra  $\DS_{x}(\fpq_2)$ can be identified with 
$\mathbb{C}x$. It is easy to see
that $\DS_{x}(L)=0$ for $\lambda\not=0$.

\subsubsection{Case $\fpsq_2$}
One has 
$$X(\fpsq_2)_{iso}=\iota(X(\fq_2)),\ \ 
X(\fpsq_2)\setminus X(\fpsq_2)_{iso}=\mathbb{C}^* \iota(H_2);$$
for each non-zero $x\in X(\fpsq_2)$ one has $\rank x=1$ and 
$\DS_x(\fpq_2)\cong \mathbb{C}\times \Pi(\mathbb{C})$.

The $\fpsq_2$-modules are $\fsq_2$-modules annihilated by $z$. In particular,
 $L_{\fpsq_2}(\lambda)$ is a simple $\fsl_2$-module. 

The algebra  $\DS_E(\fpsq_2)$ can be identified with
$\mathbb{C}e\oplus \mathbb{C}F)$.
By above,  $\DS_{E}(L)=L$ as a $\mathbb{C}e$-module and $F(\DS_E(L))=0$.

For $x\in \mathbb{C}^* H_2$
the algebra  $\DS_{x}(\fpsq_2)$ can be identified with
$\mathbb{C}h\oplus \mathbb{C}x$.
This gives $\DS_{x}(L)=L$ as a $\mathbb{C}h$-module and $x(\DS_{x}L)=0$.

\subsubsection{}
Note that for $\fg=\fq_2,\fsq_2,\fpq_2,\fpsq_2$ we have 
$\depth \fg=\defect \fg=1$.

\subsection{The set $X(\fq_{n})$}\label{DSq}
In~\ref{DSq}--\ref{thetaqn} we consider the case  $\fg:=\fq_n$ with $n\geq 2$.

For $i=1,\ldots, n-1$ we set $\alpha_i:=\vareps_i-\vareps_{i+1}$.
For each $s\leq n$ we identify $\fq_s$ with the subalgebra of $\fq_n$ with the set of simple roots
$\{\alpha_1,\ldots,\alpha_s\}$. 
For $r=0,1,\ldots, [\frac{n}{2}]$ we set
$$S_r:=\{\alpha_{n-1-2i}\}_{i=0}^{r-1}$$
and fix $x_r\in X$ such that $\supp(x_r)=S_r$ ($x_0:=0$).

\subsubsection{}
\begin{lem}{G0orbitQ}
The elements $x_0,\ldots, x_{[\frac{n}{2}]}$ 
 form a set of representatives of  $GL_n$-orbits in $X$.
\end{lem}
\begin{proof}
One has $X=\{T_{0,B}|\ B^2=0\}$.
The elements of $GL_n$
act  by conjugation on $B$, so
the $GL_n$-orbits correspond to the Jordan forms. 
\end{proof}

\subsubsection{}\label{findimq}
Recall that $g\in GL_n$ acts on $\fg$ by an inner automorphism.
If $N$ is a finite-dimensional $\fg$-module, then $g$ acts on
$N$; as a result, $g$ induces a bijection between $\DS_x(N)$ and $\DS_{gx}(N)$, which is compatible with the algebra isomorphism $\fg_x\iso \fg_{gx}$. 

\subsubsection{}
\begin{prop}{propqn}
The algebra $\DS_{x_r}(\fq_n)$  can be identified
with $\fq_{n-2r}$.
\end{prop}
\begin{proof}
Set $x:=x_r=\sum_{\alpha\in S_r} x_{\alpha}$, where $x_{\alpha}\in\fg_{\alpha}$. 
By~\Lem{lem42} one has 
$$\DS_x(\fg)=\DS_x(\displaystyle\sum_{\alpha\in Y}\fg_{\alpha}+\fg^{\fh}),
\text{ where }\ (\alpha|\beta^{\vee})=0\ \text{ for }\beta\in\supp(x)\}.$$

Denote by  $\fq_2(\alpha)$  the copy of
$\fq_2$ corresponding to the root $\alpha$. One has
$$\sum_{\alpha\in Y}\fg_{\alpha}+\fh=\displaystyle\bigoplus_{\alpha\in S_r} \fq_2(\alpha)\bigoplus \fq_{n-2r}$$
as a $\mathbb{C}x$-module. One has
 $x \fq_{n-2r}=0$. For each $\alpha\in S_r$, the
  action of $x$  on  $\fq_2(\alpha)$ coincides with the adjoint action
  of $x_{\alpha}\in \fq_2(\alpha)$, so  $\DS_x( \fq_2(\alpha))=0$
  (by~\ref{sq2}). This implies the statement.
\end{proof}

\subsubsection{}
\begin{cor}{depthq}
$\depth(\fq_n)=[\frac{n}{2}]$ and 
$X_r=GL_n x_r$.

\end{cor}
%

\subsubsection{Remark}
One has $X(\fq_n)_r=X(\fgl(n|n)_{2r})\cap\fq_n$.

\subsection{Case when $\supp(x)$ is an iso-set}\label{primit}
The maximal cardinality of an iso-set is $[\frac{n}{2}]$ and
each iso-sets of the cardinality $r$ is $W$-conjugated to $S_r$.

Take 
$x:=x_r$. Using~\Prop{propqn}    we identify $\fg_x$ with $\fq_{n-2r}$.
The algebra $\fh_x:=\fg_x\cap\fh$ is spanned by $h_1,\ldots,h_{n-2r}$.
We identify $\fh_x^*$ with the span of $\vareps_1,\ldots,\vareps_{n-2r}$.


\subsubsection{}\label{wxr}
Now take an arbitrary $x\in X_r$  such that $\supp(x)$ is an iso-set, i.e.
$\supp(x)=wS_r$ for some $w\in W$.
Take $g\in GL_n$ such that
$g\fh=\fh$  and that $g\fg_{\alpha}=\fg_{w\alpha}$.
We identify 
$\DS_x(\fg)$ with $g\DS_{x_r}(\fg)$.
Then $\fh_x:=\DS_x(\fg)\cap \fh$ is a Cartan subalgebra of $(\DS_x(\fg))_0\cong
\fgl_{n-2k}$. 
We fix the following triangular decomposition of $\fg_x$:
$$\fg_x=\fn^-_x\oplus \fg_x^{\fh_x}\oplus \fn_x,$$
where $\fn^-_x:=\fn^-\cap \fg_x$ and $\fn_x:=\fn\cap\fg_x$.
Notice that $\fh_x$ is spanned by $\{h_{w(i)}\}_{i=1}^{n-2r}$;
we identify $\fh_x^*$ with the span of $\{\vareps_{w(i)}\}_{i=1}^{n-2r}$.
Observe that 
\begin{equation}\label{CS}
\{\nu\in\fh^*|\   \nu|_{\fh_x}=0,\   \nu\ \text{orthogonal to }\ S\}=\mathbb{C}S
\end{equation}
(these spaces have the same dimension and the inclusion $\supset$ is straightforward).

Let $N$ be a  $\fg$-module with a diagonal action of $\fh$.
  For each $\mu\in\fh^*_x$ we set
$$N_{\mu}:=\sum_{\gamma\in\mathbb{C}S} N_{\mu+\gamma}.$$
Note that $N_{\mu}$ is $(\fh+\mathbb{C}x)$-submodule of $N$.
Combining~(\ref{CS}) and~\Lem{lem42} we  get
\begin{equation}\label{eq42q}
\DS(N)=\sum_{\mu\in\fh^*_x}\DS_x(N)_{\mu},\ \ \ 
\DS_x(N)_{\mu}=\DS_x(N_{\mu}).\end{equation}

\subsubsection{}
Recall that $\fg_x^{\fh_x}$ is the Cartan subalgebra of $\fg_x$
and $\fh_x=(\fg_x^{\fh_x})_0$.
The following statement demonstrates a peculiarity
of $\fq_n$.

\begin{prop}{primitq}
Let $x\in X$ be such that 
$S:=\supp(x)$ is an iso-set  and $S\subset\Pi$. Assume 
$\lambda\in\fh^*$ is orthogonal to $S$.
We set $\lambda':=\lambda|_{\fh_x}$ 
and view $M':=\DS_x(L(\lambda))_{\lambda'}$
as a $\fg^{\fh_x}$-module.
\begin{enumerate}
\item
One has $\fn_x M'=0$.
\item 
For $\lambda\in\fh^*_x\ $  $M'$ can be identified with
$L(\lambda)_{\lambda}$.
\item
One has 
$M'\cong L_{\fg_x^{\fh_x}}(\lambda')\otimes (\mathbb{C}\oplus\Pi\mathbb{C}) ^{\otimes s}$, where $s$ is the cardinality of the set
$$\{\beta\in S|
 \ (\lambda|{\beta})\in\mathbb{N}_{>0}\}.$$
\end{enumerate}
\end{prop}
\begin{proof}
Set $N:=L(\lambda)$ and $M:=\DS_x(L)$. Identify $\fh_x^*$ with a subspace in $\fh^*$ as in~\ref{wxr}.
The formula~(\ref{CS}) gives $\lambda-\lambda'\in\mathbb{C}S$. Using~(\ref{eq42q}) we get
$$M'=M_{\lambda'}=\DS_x(N_{\lambda'}),\ \text{ where }
N_{\lambda'}:=
\sum_{\nu\in \mathbb{C}S}N_{\lambda-\nu}.$$

For (i) assume that $(\fn_x)_{\alpha} M'\not=0$ for some
 $\alpha\in\Delta^+(\fg_x)$. Since
$(\fn_x)_{\alpha} M_{\lambda'}\subset M_{\lambda'+\alpha}$, the formula
(\ref{eq42q})  implies the existence of  $\gamma\in\mathbb{N}\Delta^+$ satisfying
$$\lambda-\gamma\in (\lambda'+\alpha)+\mathbb{C}S.$$
Then $\gamma+\alpha\in \mathbb{C}S$, which is impossible since
$\alpha\in\Delta^+$ and $\alpha\not\in \mathbb{C}S$. This establishes (i).

For (ii)
recall that $(\fg_x^{\fh_x})_0=\fh_x$ is spanned by $\{h_j\}_{j\in J}$ 
where
$J\subset \{1,\ldots,n\}$ has cardinality $n-2r$. Therefore
 $$\fg_x^{\fh_x}\cong \underbrace{\fq_1\times\ldots\times\fq_1}_{n-2r\ \text{times}}.$$
 Retain notation of~\ref{usefulq} and set $\Sigma':=S=:\{\beta_1,\ldots,\beta_r\}$.
We have
 \begin{equation}\label{qqqq}
 \fg^h=\fg_x^{\fh_x}\times\fl,\ \text{ where }
 \fl:=\fq_2(\beta_1)\times\ldots\times
\fq_2(\beta_r).\end{equation}

By~\ref{usefulq}, $N_{\lambda'}$ is a simple $\fg^h$-module of the highest weight $\lambda$
and is a direct summand of $N$ viewed as a $\fg^h$-module. 
Since $x\in \fg^h$, the latter property allows to
identify $M'$ with
$\DS_x(N_{\lambda'})$. By above, $N_{\lambda'}\cong L_{\fg^h}(\lambda)$.
In the light of~(\ref{qqqq}),
 $L_{\fg^h}(\lambda)$ can be decomposed as
$$L_{\fg^h}(\lambda)\cong L_{\fg_x^{\fh_x}}(\lambda')\otimes L_{\fl}(\nu),$$
where $\nu$ is the restriction of $\lambda$ to $\fh\cap\fl$
and $\fl$ (resp., $\fg_x^{\fh_x}$) 
 acts by zero on the first (resp., on the second) factor in the above decomposition. Since $x\in\fl$ one has
 $xL_{\fg_x^{\fh_x}}(\lambda')=0$ and
$$\DS_x(L_{\fg^h}(\lambda))\cong L_{\fg_x^{\fh_x}}(\lambda')\otimes 
 \DS_x(L_{\fl}(\nu)).$$
 For $\lambda\in\fh_x^*$ we have $\nu=0$, so 
$L_{\fg^h}(\lambda)=L_{\fg_x^{\fh_x}}(\lambda')$; this establishes (ii).

For (iii)  we denote by $\nu_i$ the restriction of $\lambda$ to $\fh\cap \fq_2(\beta_i)$ and write
 $$L_{\fl}(\nu) \cong  \displaystyle\bigotimes_{i=1}^r L_{\fq_2}(\nu_i).
 $$

The action of $x$ on each 
$L_{\fq_2}(\nu_i)$ coincides  with the action of
a non-zero $x'\in X(\fq_2)$ and using~\ref{sq2} we obtain
 (iii).
\end{proof}


\subsection{The map $\theta_x$ for $\fq_n$}\label{thetaqn}
By~\ref{DScentre}, the functor $\DS_x$ induces the algebra homomorphism
$\theta_x:\cZ(\fg)\to\cZ(\fg_x)$.

By~\ref{G0orbitQ},  each $y\in X(\fq_n)_r$ 
is of the form $y=\phi(x)$, where $\phi$ is an inner automorphism. 
Note that $\phi$ induces an isomorphism $\iota:\fg_x\iso \fg_{y}$ 
and that $\theta_{y}=\iota\circ \theta_x$. Hence we can (and will) take
$x:=x_r$ and use the identification of~\ref{primit}.

For $\nu\in \fh_x^*$ we
denote by $\chi'_{\nu}$ the corresponding central character 
$\chi':\cZ(\fg_x)\to \mathbb{C}$.

\subsubsection{}
\begin{cor}{DStheta}
Take  $x:=x_r$.

\begin{enumerate}
\item For $\lambda\in  \fh_x^*$ the $\fg_x$-module
$\DS_x(L(\lambda))$ has a non-zero primitive vector of weight $\lambda$.
\item
$(\HC_x\circ \theta_x)(z)=\HC(z)|_{\fh_x^*}$,
where 
$\HC_x:\cZ(\fg_x)\to \cS(\fh_x)$ is the Harish-Chandra
homomorphism  for $\fg_x\cong \fq_{n-2r}$. 

\item
 One has $\theta_x^*(\chi'_{\lambda})=\chi_{\lambda}$
for each $\lambda\in \fh_x^*$.

\item
$\Core(\theta^*_x(\chi))=\Core(\chi)$ and $\theta_x^*$ increases
atypicality  by $\rank x$; 
\item
 $\theta_x^*$ is injective and $Im\theta_x^*$ consists of 
the central characters of atypicality at least $\rank x$.
\end{enumerate}
\end{cor}
\begin{proof}
The assertion (i) follows from~\Prop{primitq} and implies (ii).
In its turn, (ii) implies (iii) and (iv) follows from (iii).
Finally,  (v)
follows from~\Prop{propchi} and (iv).
\end{proof}

\subsubsection{}\begin{cor}{DScentq}
For each $\chi\in\mspec_{\Mod}\cZ(\fq_n)$ one has 
$$\atyp\chi=\depth\chi=\depth \CO(\chi).$$
\end{cor}
\begin{proof}
Set $k:=\atyp\chi$ and take $x:=x_k$.
Using~(\ref{depthleq}) and~\Cor{DStheta}
(v) we obtain $\depth \chi\leq k$ and thus $\depth \CO(\chi)\leq k$.
 Let $\Core(\chi)=\{a_i\}_{i=1}^{n-2k}$. Note that
$\chi=\chi_{\nu}$ for $\nu:=\sum_{i=1}^{n-2k} a_i\vareps_i\in\fh^*_x$. By~\Lem{primitq},  $\DS_x(L(\nu))\not=0$, so
$\depth(L(\nu))\geq  k$. This gives
 $\depth \chi,\depth \CO(\chi)\geq k$ and completes the proof.
\end{proof}

\subsubsection{}
\begin{prop}{propthetaq}
Take $x\in X_r$.

\begin{enumerate}
\item
The map $\theta_x$ is surjective;
\item
$\DS_x(\Mod(\chi))\subset \Mod_x((\theta_x^*)^{-1}\chi)$, where 
$\Mod_x$ stands for the  category of $\fg_x$-modules.
\end{enumerate}
\end{prop}
\begin{proof}
By above, it is enough to consider $x:=x_r$. Using notation of~\ref{DStheta} (ii)
we set
$$Z_n:=\HC(\cZ(\fq_n))\subset \cS(\fh),\ \ \ \ Z_{n-2r}:=\HC_x(\cZ(\fq_{n-2r}))\subset \cS(\fh_x)\subset \cS(\fh)$$
and identify $\cZ(\fq_n)$ with $Z_n$ and
 $\cZ(\fq_{n-2r})$ with $Z_{2n-r}$.  
 
Consider the projection $\psi:\cS(\fh)\twoheadrightarrow\cS(\fh_x)$ given by
 $$h_i\mapsto h_i,\ \ \ h_j\mapsto 0\ \text{ for } 1\leq i\leq n-2r<j\leq n.$$
 Using the above identification the formula in~\ref{DStheta} (ii) takes the form
 $\theta_x(z)=\theta(z)|_{\fh_x^*}$, that is
$\theta_x$ coincides with the restriction of $\psi$ to $Z_n$.

By~\cite{Sq}, the algebra $Z_n$ is generated by the set $\{p_{2k+1}^{(n)}\}_{k=0}^{\infty}$, where
$$p_{2k+1}^{(n)}:=\displaystyle\sum_{i=1}^n h_i^{2k+1}.$$ 
Since $\psi(p_{2k+1}^{(n)})=p_{2k+1}^{(n-2r)}$, 
 the map $\theta_x$ is surjective. 
Now (ii) follows from~\Prop{propBZ}.
\end{proof}

\subsubsection{Remark}
Note that, as in~\ref{thmtheta} below,  for injectivity of $\theta_x^*$
it is enough to see that
$p_{2k+1}^{(n)}$ lies in $Z_n$, whereas  the proof of surjectivity is based
on the fact that $Z_n$ is generated by these elements.

\subsection{$\DS$ and translation functors}\label{DStransQ} 
Let $\chi_1,\chi_2$  be  central characters of $\fg$. Consider 
the translation functor 
 $T_{\chi_1,\chi_2}:\CO(\chi_1)\to\CO(\chi_2)$ given by
$$T_{\chi_1,\chi_2}(M):=(M\otimes V_{st})^{\chi_2},$$
where $V_{st}=L(\vareps_1)$ is the standard $\fq_n$-module and
$M\mapsto M^{\chi_2}$ stands for the projection $\CO\to \CO(\chi_2)$.
Similarly to Corollary 4.4 in~\cite{Skw},
combining~\Prop{DStheta} (iv) and the formula $\DS_x(L_{\fg}(\vareps_1))=
L_{\fg_x}(\vareps_1)$, we obtain the following corollary.

\subsubsection{}
\begin{cor}{}
Let $\chi_1,\chi_2$ (resp., $\chi'_1,\chi_2'$) be the central characters of $\fg$ (resp., $\fg_x$) satisfying
$\Core(\chi_i)=\Core(\chi'_i)$ for $i=1,2$. Then
$$\DS_x\circ T_{\chi_1,\chi_2}=T_{\chi'_1,\chi'_2}\circ \DS_x.$$
\end{cor}

\subsection{KW-conditions}
Take $\fg:=\fq_n$. We denote by $P^+(\fg)$ the set of dominant weights, i.e.
$$P^+(\fg):=\{\lambda\in\fh^*|\ \dim L(\lambda)<\infty\}.$$
Character formulae for finite-dimensional simple modules
were obtained in~\cite{PS1},\cite{PS2},\cite{Br},\cite{ChK}. 

By~\cite{P}, $\lambda\in P^+(\fg)$ if and only if
$\lambda=\sum a_i\vareps_i$ with $a_i-a_{i+1}\in\mathbb{Z}_{\geq 0}$ and $a_i=a_{i+1}$
implies $a_i=0$.
Recall that $L(\lambda)$ is typical if $a_i+a_j\not=0$ for
any $1\leq i<j\leq n$. For a typical central character $\chi$ there exists at most 
one weight $\lambda\in P^+(\fg)$ with $\chi_{\lambda}=\chi$:
in this case $a_i-a_{i+1}\in\mathbb{Z}_{>0}$ for each $i$. 

Arguing as in~\Cor{DScentq} we obtain $\depth\Fin(\fg)(\chi)=\depth \chi$ if $\Fin(\fg)(\chi)\not=0$.

\subsubsection{}
\begin{defn}{KWqdef}
We say that $\lambda\in P^+(\fg)$ satisfies the {\em KW-conditions} if there exists  an iso-set $S\subset\Pi$ which is a maximal iso-set orthogonal to $\lambda$.
\end{defn}

\subsubsection{}
\begin{prop}{tameq}
Let $\lambda=\sum a_i\vareps_i\in P^+(\fg)$ has atypicality $k$ and
 satisfies the KW-condition. 
 Take $x$ of atypicality $k$.
 
\begin{enumerate}
\item There exists a unique $\lambda'\in P^+(\fg_x)$ with $\Core\lambda'=\Core\lambda$.
\item
If for some $s$ one has
\begin{equation}\label{ai}
a_s=a_{s+1}=\ldots=a_{s+2k-1}=0,
\end{equation}
 then
 $\DS_x(L(\lambda))=L_{\fg_x}(\lambda')$.
 
 \item If~(\ref{ai}) does not hold, then  $k=1$ and
$$\DS_x(L(\lambda))=L_{\fg_x}(\lambda')\oplus \Pi(L_{\fg_x}(\lambda')).$$
\end{enumerate}
\end{prop}
\begin{proof}

Take $S$ as in~\ref{KWqdef}. By~\ref{findimq},
 we can (and will) assume that $\supp(x)=S$.
Define $\fg_x$ and $\fh_x$ as in~\ref{primit}.
Since $\rank x=k$, $\Core(\lambda)$ has the cardinality $n-2k=\dim \fh_x$;
this gives (i). By~\Cor{DScentq}, 
$$(\theta_x^*)^{-1}(\chi_{\lambda})=\chi'_{\lambda'},$$
where $\chi'_{\lambda'} \in \mspec\cZ(\fg_x)$.
Since $M:=\DS_x(L(\lambda))$ is finite-dimensional, 
all simple subquotients of $M$ are of the form
$L_{\fg_x}(\lambda')$ or $\Pi(L_{\fg_x}(\lambda'))$.
Using~\Prop{primitq} we obtain $\lambda'=\lambda|_{\fh_x}$
and deduce (ii). 

It is easy to see that if $\lambda\in P^+(\fg)$
satisfies the KW-conditions and~(\ref{ai}) does not hold, then
 for some index $i$ one has $a_{i+1}=-a_i\not=0$ and $S=\{\vareps_i-\varesp_{i+1}\}$.
Since $\lambda$ is dominant, $a_i-a_{i+1}\in\mathbb{N}_{>0}$ and
(iii) follows from from~\Prop{primitq} (iii).
\end{proof}

\subsection{$\DS_x$ for $\fsq_n$, $\fpq_n$ and $\fpsq_n$}\label{qrelatives}
We retain notation of~\ref{DSq}. Recall that $z:=T_{Id,0}$.
Recall that 
$\iota$ stands for the canonical maps
$\fq_n\twoheadrightarrow\fpq_n$ and $\fsq_n\twoheadrightarrow\fpsq_n$.

\subsubsection{}
\begin{prop}{propsqn}
Take $x\in \fg$ with $\supp(x)=S_r$.

\begin{enumerate}

\item For $r=\frac{n}{2}$ one has  $\DS_x(\fsq_n)\cong\mathbb{C}$,  $\DS_x(\fpq_n) \cong\Pi(\mathbb{C})$ and  $\DS_x(\fpsq_{2r}) \cong\mathbb{C}\times \Pi(\mathbb{C})$.

\item For $r<\frac{n}{2}$ the algebra $\DS_x(\fsq_n)$ (resp., $\DS_x(\fpq_n),\DS_x(\fpsq_n)$)
 can be identified
with $\fsq_{n-2r}$ (resp., with $\fpq_{n-2r}, \fpsq_{n-2r})$).\end{enumerate}
\end{prop}
\begin{proof}
For (i) recall that $\DS_x(\fq_n)=0$ if $r=\frac{n}{2}$.
Since  $\fq_n/\fsq_n\cong \Pi(\mathbb{C})$ as $\fsq_n$-modules,
Hinich's Lemma gives  $\DS_x(\fsq_n)\cong\mathbb{C}$;
the case $\fpq_n$ is similar. Take $\fg:=\fpsq_{2r}$.
Since $\fg=\fsq_{2r}/\mathbb{C}$,
Hinich's Lemma implies that  $\DS_x(\fg)$ is either $0$
or $\mathbb{C}\times \Pi(\mathbb{C})$. It is easy to see that $\fq_n$ contains $y$ with $\supp(y)=-\supp(x)$ and $[y,x]=z$.
Therefore $[\iota(x),\iota(y)]=0$.
 One readily sees that $\iota(y)\not\in [x,\fg]$, so
$\iota(y)$ has a non-zero image in $\DS_x(\fg)$. This gives (i).

For (ii) take $r<\frac{n}{2}$ and
denote by $\tilde{\fq}_{2r}$ the copy of 
$\fq_{2r}$ corresponding to the root system
$\{\vareps_i-\varesp_j\}_{n-2r<i,j\leq n}$. 

Take $\fg:=\fsq_n$ and set
 $\tilde{\fsq}_{2r}:=\fsq_n\cap \tilde{\fq}_{2r}$.
It is easy to see that
$$\fp:=\fsq_{n-2r}+ \tilde{\fsq}_{2r}+\fg^{\fh}$$
is  a subalgebra of $\fsq_n$. By~\Lem{lem42} 
$$\DS_x(\fp)=\DS_x(\displaystyle\sum_{\alpha\in Y}\fp_{\alpha}+\fp^{\fh}),\ \ \ \DS_x(\fg)=\DS_x(\displaystyle\sum_{\alpha\in Y}\fg_{\alpha}+\fg^{\fh}),$$
where
$Y:=\{\alpha\in\Delta|\ (\alpha|\beta^{\vee})=0\ \text{ for }\beta\in\supp(x)\}$.
Since $\fg^{\fh}=\fp^{\fh}$ and  $\fg_{\alpha}=\fp_{\alpha}$ for $\alpha\in Y$,
we obtain
$\DS_x(\fp)=\DS_x(\fg)$.

As $\tilde{\fsq}_{2r}$-module $\fp$ can be decomposed as
\begin{equation}\label{peq}
\fp=\fsq_{n-2r}\oplus (\tilde{\fsq}_{2r} +\mathbb{C} H),
\end{equation}
where $H$ spans  $\fp_1^{\fp_0}$.
 Note that $x\in\tilde{\fsq}_{2r}$ and that  $[\tilde{\fsq}_{2r},\fsq_{n-2r}]=0$.
Therefore
$$\DS_x(\fp)=\fsq_{n-2r}\oplus \DS_x(\tilde{\fsq}_{2r} +\mathbb{C} H).$$

Define  a linear bijection $\tilde{\fsq}_{2r} +\mathbb{C} H\iso \fsq_{2r}+\mathbb{C}T_{0,Id}=\fq_{2r}$
via the identification $\ft=\fsq_{2r}$ and  $H\mapsto  T_{0,Id}$.
One readily sees that this map is a $\tilde{\fsq}_{2r}$-isomorphism. 
One has 
$\DS_x(\tilde{\fsq}_{2r}+\mathbb{C} H)=\DS_x(\fq_{2r})=0$ and so
 $\DS_x(\fsq_n)=\DS_x(\fp)=\fsq_{n-2r}$ as required.
 
For the case $\fpq_n$  we  identify  
$\tilde{\fq}_{2r}$ with 
$\iota(\tilde{\fq}_{2r})\subset\fpq_n$.  Then  $x\in \tilde{\fq}_{2r}$
and 
$\fpq_{n-2r}\oplus \tilde{\fq}_{2r}$ is a subalgebra of $\fpq_n$.
Arguing as above, we obtain
$$\DS_x(\fpq_n)=\DS_x(\fpq_{n-2r}\oplus \tilde{\fq}_{2r})=\fpq_{n-2r}.$$

 For   the remaining case $\fg:=\fpsq_n$ we  identify  
$\tilde{\fsq}_{2r}$ with 
$\iota(\tilde{\fsq}_{2r})\subset\fg$ and substitute $\fp$ by 
$$\fpsq_{n-2r}+ \tilde{\fsq}_{2r}+\fg^{\fh}=
\fpsq_{n-2r}\oplus (\tilde{\fsq}_{2r} +\mathbb{C} H).$$
By above, $\DS_x(\tilde{\fsq}_{2r}+\mathbb{C} H)=0$, so
 $\DS_x(\fpsq_n)=\fpsq_{n-2r}$ as required.
\end{proof}

\subsubsection{Case $\fsq_n$}\label{DSsq}
 One has $X(\fsq_n)=X(\fsq_n)_{iso}=X(\fq_n)$. By~\Prop{propsqn}
we have $\depth(\fsq_n)=[\frac{n}{2}]$ and $X(\fsq_n)_r=X(\fq_n)_r$.
For $\fsq_n$ 
the analogue of
\Prop{propthetaq} holds for  $r\not=\frac{n}{2}$ (the proof is the same).

\subsubsection{Case $\mathfrak{pq}_n$}
One has $X(\mathfrak{pq}_n)_{iso}=\iota(X(\fq_n))$ and 
$$X(\mathfrak{pq}_n)=\iota(X(\fq_n))\cup X',\ \text{ where } X'=\{\iota(T_{0,B}) |\ B^2\in \mathbb{C}^*\Id\}.$$

Each element in $X'$ is $GL_n$-conjugated to $x':=T_{0,B}$, where
$B$ is a diagonal matrix with $B^2\in \mathbb{C}^*\Id$.
 It is easy to see that the algebra 
$\DS_{x'}(\fpq_n)$ can identified with $\mathbb{C}x'\cong \Pi(\mathbb{C})$.
We obtain
$$\depth(\fpq_n)=[\frac{n}{2}],\ \ \ 
\ \iota(X(\fq_n)_r)\subset X(\fpq_n)_r,\ \ X'\subset  X(\fpq_n)_{[\frac{n}{2}]}.$$ 
\Prop{propthetaq} holds for $r<\frac{n}{2}-1$ (the proof is the same).

\subsubsection{Case $\mathfrak{psq}_n$}
One has $X(\mathfrak{psq}_{n})=X(\mathfrak{pq}_{n})\cap \fpsq_n$ and
 $X(\mathfrak{psq}_{n})_{iso}=X(\mathfrak{pq}_{n})_{iso}$.
This gives  $X(\mathfrak{psq}_{2n+1})=\iota(X(\fsq_{2n+1}))$ and
$$X(\mathfrak{psq}_{2n})=\iota(X(\fsq_{2n}))\cup X',\ \text{ where } X'=\{\iota(T_{0,B}) |\ B^2\in \mathbb{C}^*\Id,\ Tr B=0\}.$$
Note that $X'$ is $GL_n$-conjugated to $x':=T_{0,B}$, where
$B\in\fsl_{2n}$ is a  diagonal matrix satisfying $B^2\in \mathbb{C}^*\Id$. It is easy to see that the algebra 
$\DS_{x'}(\fpsq_{2n})$ can identified with 
$\mathbb{C}T_{B,0}\oplus \mathbb{C}x'\cong \mathbb{C}\times \Pi(\mathbb{C})$.
Summarizing we have
$$\depth(\fpsq_{n})=[\frac{n}{2}],\ \ \
\ \iota(X(\fsq_n)_r)\subset X(\fpsq_n)_r,\ \ X'\subset  X(\fpsq_n)_{\frac{n}{2}}.$$

\Prop{propthetaq} holds for $r<\frac{n}{2}-1$ (the proof is the same).

\section{The case of finite-dimensional Kac-Moody superalgebras}\label{sectbasic}
\label{orbit}
Let $\fg$ be a finite-dimensional Kac-Moody superalgebra.
We retain notation of Section~\ref{sectKMQ}.
From~\cite{DS} it follows that  $X(\fg)=X(\fg)_{iso}\ $, 
$\depth(\fg)=\defect\fg$ and
\begin{equation}\label{isosetrank}
\text{ if $\supp(x)$ is an iso-set, then $\rank(x)$ is equal to the cardinality of $\supp(x)$.}
\end{equation}

We recall some other results of~\cite{DS} in~\ref{gx}, \ref{endrecollections}
below.

%
%
%

\subsection{Choice of base}\label{isosetbasic1}
The algebra $D(m|n)$ with $m\not=0$ admits an involutive automorphism
 $\sigma$ which acts on $\fg_0=\mathfrak{o}_{2m}\times \mathfrak{o}_{2m}$ 
 as follows:
$\sigma|_{\mathfrak{sp}_{2n}}=\Id$ and  $\sigma|_{\mathfrak{o}_{2m}}$
is a Dynkin diagram involution if $m>1$ and $-\Id$ for $m=1$
(one has $\mathfrak{o}_2=\mathbb{C}$). 
Note that $\sigma(\fh)=\fh$.
We denote by $\sigma$ also the induced map on $\fh^*$ (then
 $\sigma=r_{\vareps_m}$).

We fix a base $\Sigma\subset \Delta$ with the following properties:
\begin{itemize}
\item
$\Sigma$ contains a maximal
possible number of odd roots;

\item $\sigma(\Sigma)=\Sigma$ for $D(m|n)$ if $m>1$.
\end{itemize}

 Then $\Sigma$ contains
an iso-set of the maximal possible cardinality (equal to $\defect \fg$).

For instance, for $D(2|2)=\osp(4|4)$ we may take $\Sigma=\{\delta_1-\vareps_1,\vareps_1-\delta_2,\delta_2\pm\vareps_2\}$.

\subsection{The algebra $\fg_x$}\label{gx}
Let $S$ be an iso-set of cardinality $r>0$ satisfying 
$$S\subset \Sigma\cup(-\Sigma).$$
We fix $x\in X$ such that $\supp(x)=S$.

\subsubsection{}\label{alggx}
One has 
${\fh}^x=\{h\in{\fh}|\ S(h)=0\}$. We introduce
$${\Delta}_x:=(S^{\perp}\cap{\Delta})\setminus(-S\cup S).$$
By~\cite{DS},  ${\fg}_x$ can be identified
with a subalgebra of ${\fg}$ generated by the root spaces
$\fg_{\alpha}$ with $\alpha\in{\Delta}_x$ and a subalgebra
${\fh}_x\subset\fh^x$ with the following properties
$${\fh}_x\oplus (\sum_{\beta\in
  S}\mathbb{C}h_{\beta})={\fh}^x;\ \ \ \forall \alpha\in {\Delta}_x\ \  [\fg_{\alpha},\fg_{\alpha}]\subset
{\fh}_x.$$
Moreover, ${\fh}_x$
is a Cartan subalgebra of ${\fg}_x$.

If ${\Delta}_x$ is not empty, then
${\Delta}_x$ is the root system of the Lie superalgebra ${\fg}_x$ and one can choose a base
${\Sigma}_x$ in $\Delta_x$ such that 
\begin{equation}\label{triangx}
\Delta^+({\Sigma}_x)=\Delta^+\cap {\Delta}_x.
\end{equation}
(If $\Delta_x=\emptyset$ we take $\Sigma_x=\emptyset$.)
If ${\fg}=\fgl(m|n)$ (resp., $\osp(m|n)$), then ${\fg}_x=\fgl(m-r|n-r)$
(resp., $\osp(m-2r|n-2r)$).
For  $\fg=D(2|1;a),G_3,F_4$ with $x\not=0$, one has  $r=1$ and ${\fg}_x=\mathbb{C},\fsl_2,\fsl_3$ respectively.

\subsubsection{Dual Coxeter number}\label{dualCoxeter}
The restriction of the non-degenerate invariant bilinear form on $\fg$ gives a 
non-degenerate invariant bilinear form on $\fg_x$.

Recall that the dual Coxeter number $h^{\vee}(\fg)$ is the eigenvalue of
the Casimir operator on the adjoint representation.
By~\cite{HW},
$\theta$ maps the Casimir operator of $\fg$ to the Casimir operator
of $\fg_x$. Therefore  
$$h^{\vee}(\fg)=h^{\vee}(\fg_x).$$

\subsubsection{}\label{W''}
We identify $\fh^*_x$
 with a subspace in $\fh^*$: we take $\fh^*_x$  spanned
 by $\Delta_x$ if $\fg_x\not=\mathbb{C}$ and $\fg\not=\fgl(m|n)$;
 for $\fg=\fgl(m|n)$ we take the minimal span of
 $\vareps_i$s and $\delta_j$s which contains  $\Delta_x$;
 for $\fg_x=\mathbb{C}$ we choose an arbitrary $\fh^*_x$ with 
the property $S^{\perp}=\mathbb{C}S\oplus \fh_x^*$.
We set
 $$W'':=\{w\in W|\ w(-S\cup S)=(-S\cup S)\}.$$
Then $W''\fh^*_x=\fh^*_x$ and $W''\Delta_x=\Delta_x$.
Note that $W''$ contains 
 $W(\fg_x)$ (the Weyl group of $\fg_x$)  which is generated by $r_{\alpha}$ with $\alpha\in (\Delta_x\cap \Delta_0)$. Viewing $W''$ as a subgroup of $GL(\fh^*_x)$ we obtain
$W''=W(\fg_x)\cup W(\fg_x)\sigma_x$, where  $\sigma_x$
is as follows:
 \begin{itemize}
 \item
 for $\fg=D(m|n)$ with $m>r$ we have
 $\fg_x=D(m-r|n-r)$ and $\sigma_x$ is as above;
 \item
 for $\fg=F(4)$ $\sigma_x$ is 
the involution of the Dynkin diagram of $\fg_x=\fsl_3$;
\item
  for $\fg=D(2|1;a)$ we have $\fg_x=\mathbb{C}$ and  $\sigma_x:=-Id$;
\item
 $\sigma_x:=\Id$ for all other cases. 
  \end{itemize}

In all cases $\sigma_x(\Sigma_x)=\Sigma_x$.

\subsection{The map $\theta_x^*$}\label{endrecollections}
Consider the usual $\rho$-twisted action of $W$ on $\fh^*$: $$w\lambda:=w(\lambda+\rho)-\rho.$$
The restriction of the Harish-Chandra map gives an  algebra 
 monomorphism 
 $$\HC:\cZ(\fg)\to\cS(\fh)^{W.}.$$

Using~\Lem{lem42} and~\ref{usefulq} it is easy to see that 
\begin{equation}\label{HC}
\HC_x(\theta_x(z))=\HC(z)|_{\fh_x^*}. 
\end{equation}

\subsubsection{}\label{preservescore}
Take $\lambda\in\fh_x^*\subset\fh^*$. Let $\chi'_{\lambda}$ be the corresponding
central character of $\fg_x$. By above, $\theta_x^*(\chi'_{\lambda})=\chi_{\lambda}$.
  As in~\Cor{DStheta} this implies that $\theta_x^*$ preserves
 the cores for non-exceptional $\fg$ and that
$Im\theta_x^*$    consists of the central characters of atypicality at
least $\rank x$.  Recall that $\DS_x(\Mod(\chi))=0$ if 
 $\chi\not\in Im\ \theta_x^*$, see~\Prop{propBZ}. 
Arguing as in~\Cor{DScentq} we obtain
$$\atyp\chi=\depth \CO(\chi)=\depth\chi$$
and  $\depth\Fin(\fg)(\chi)=\depth \chi$ if
 $\Fin(\fg)(\chi)\not=0$.

\subsubsection{}
Since $\HC(\cZ(\fg))\subset \cS(\fh)^{W.}$ and $W''\fh_x^*=\fh_x^*$ we have
$$\HC_x(\theta_x(\cZ(\fg))\subset \cS(\fh_x)^{W''.}.$$
Take $\sigma_x$ as in~\ref{W''}.
Since  $\sigma_x(\Sigma_x)=\Sigma_x$,  the projection $\HC_x$ commutes with $\sigma_x$.
It is not hard to see that
the Weyl vector of $\fg_x$ is equal to $\rho|_{\fh_x}$,
so the dot action of $W(\fg_x)$ on $\fh_x^*$ is the restriction
of the dot action of $W$. Moreover,
$\sigma_x\rho_x=\rho_x$, so
$\sigma_x.\lambda=\sigma_x\lambda$ for $\lambda\in\fh_x^*$.
This gives
\begin{equation}\label{HCsigma}
 \HC_x(\theta_x(\cZ(\fg))\subset\cS(\fh_x)^{\sigma_x},\ \ \ \theta_x(\cZ(\fg))\subset \cZ(\fg_x)^{\sigma_x}.
\end{equation}

\subsection{}
\begin{thm}{thmtheta}
Take $x\in X(\fg)_r$ such that $\supp(x)\subset\Sigma\cup (-\Sigma)$ is an iso-set. 
\begin{enumerate}
\item
For each $y\in X(\fg)_r$ 
there exists  $\iota_{x,y}:\fg_x\iso \fg_y$ such that
$\theta_y=\theta_x\circ \iota_{x,y}$.
\item
For $r\not=0$ we have $\theta_x(\cZ(\fg))=\cZ(\fg_x)^{\sigma_x}$.
\end{enumerate}
\end{thm}
\begin{proof}
For (i) note that each  $\phi\in\Aut(\fg)$ induces
the required isomorphism for $y:=\phi(x)$.  By~\Cor{isosetnow}
we can  assume  that 
$\supp(y)\subset (-S\cup S)$. Then, by~\ref{gx}, we can identify
$\fg_x$ with $\fg_y$ and $\fh^*_x$ with $\fh^*_y$. Now the assertion (i)
follows from~(\ref{HC}).

For (ii) observe that, by~(\ref{HCsigma}), 
$\theta_x(\cZ(\fg))\subset\cZ(\fg_x)^{\sigma_x}$.
For the opposite inclusion $\supset$ we retain notation of~\ref{centchar} and identify $\cZ(\fg)$ with  $Z(\fg)\subset \cS(\fh)^W$.
We denote by $\theta'$ the corresponding map
$Z(\fg)\to Z(\fg_x)$ and denote by $\psi:\cS(\fh)\to\cS(\fh_x)$  the map 
$$\psi(f):=f|_{\fh_x^*}.$$
Since the Weyl vector of $\fg_x$ is equal to $\rho|_{\fh_x}$, the formula~(\ref{HC}) implies that $\theta'$ coincides with  the restriction of $\psi$ to 
$Z(\fg)$. Thus $\supset$ can be rewritten as 
\begin{equation}\label{Zpsi}
\psi(Z(\fg))\supset Z(\fg_x)^{\sigma_x}.
\end{equation}

For the cases $\fgl(m|n),\osp(2m|2n),\osp(2m+1|2n)$
let $\{\vareps_i\}_{i=1}^m\cup\{\delta_i\}_{i=1}^n$
be the standard basis of $\fh^*$ and  let
$B_{m|n}:=\{e_i\}_{i=1}^m\cup\{d_i\}_{i=1}^n$
be the dual basis of $\fh$.  By (i) we can assume that 
 $\supp(x)=\{\delta_{n+1-i}-\vareps_{m+1-i}\}_{i=1}^r$.   Then
 $B_{m-r|n-r}$ is a basis of $\fh_x$ and
 $\psi$ is the projection  $\fh\to \fh_x$ given by
$$\psi(a)=a \text{ if }a\in  B_{m'|n'},\ \ \ \psi(a)=0\ \text{ if } 
a\in (B_{m|n}\setminus B_{m'|n'})$$
where $m':=m-r$ and $n':=n-r$. Consider the polynomials
$$p_{k}^{(m|n)}:=\sum_{i=1}^m e_i^{k}-\sum_{i=1}^n d_i^{k},\ \  k\in\mathbb{N}.$$

By~\cite{Sinv}, the algebra $Z(\fgl(m|n))$ is generated by $\{p_{k}^{(m|n)}\}_{k=1}^{\infty}$
and  $Z(\osp(2m+1|2n))$ is generated by $\{p_{2k}^{(m|n)}\}_{k=1}^{\infty}$.
Since $\psi(p_{k}^{(m|n)})=p_{k}^{(m'|n')}$, this gives
$\psi(Z(\fg))=Z(\fg_x)$  for the cases
$\fgl(m|n),\osp(2m+1|2n)$ and establishes (ii) for these cases.
For the case $\osp(2m|2n)$
one has 
$$W(\osp(2m+1|2n))=W(\osp(2m|2n))\coprod W(\osp(2m|2))\sigma.$$
 Taking $\beta=\vareps_1-\delta_1$ in the formula~(\ref{Zn}), we obtain 
$$Z(\osp(2m|2n))^{\sigma}=Z(\osp(2m+1|2n)).$$
In particular, $\psi(Z(\osp(2m|2n))$ contains 
$$\psi(Z(\osp(2m+1|2n))=Z(\osp(2m'+1'|2n'))=
Z(\osp(2m'|2n'))^{\sigma_x},
$$
that is $\psi(Z(\osp(2m|2n))\supset Z(\osp(2m'|2n'))^{\sigma_x}$. This establishes~(\ref{Zpsi}) for $\osp(2m|2n)$.

The remaining cases are exceptional Lie superalgebras.  In this case $\rank x=1$.
We view $\cS(\fh)$ as a polynomial algebra.  The algebra $Z(\fg)$ contains
a homogeneous polynomial $L_2$ of degree $2$ (the Casimir element) and
 $\psi(L_2)$ is a polynomial  of degree $2$ in $\cS(\fh_x)$.
 
For 
$\fg=G(3)$  one has $\fg_x=\fsl_2, \sigma_x=\Id$ and  for $\fg=D(2,1|a)$ one has $\fg_x=\mathbb{C}z$, $\sigma_x=-\Id$. In both cases 
 the algebra $Z(\fg_x)^{\sigma_x}$ is
generated by a non-zero polynomial of degree $2$ (the Casimir element and $z^2$ respectively); this gives~(\ref{Zpsi}).

Consider the remaining case $\fg=F(4)$. In this case $\{\vareps_i\}_{i=1}^3\cup\{\delta_1\}$ form a basis of $\fh^*$ and
we denote by $\{e_i\}_{i=1}^3\cup \{d_1\}$ the dual basis of $\fh$.
We choose
$$\Sigma:=\{\beta, \vareps_1-\vareps_2,\vareps_2-\vareps_3,\vareps_3\},\ \ 
\beta:=\frac{1}{2}(\delta_1-\sum_{i=1}^3\vareps_i).$$
Take $x$ such that $\supp(x)=\{\beta\}$. Then $\fg_x\cong \fsl_3$
with $\fh_x$ spanned by $e_1-e_2,e_2-e_3$ and  $\Sigma_x=\{\vareps_1-\vareps_2,\vareps_2-\vareps_3\}$.
The map $\psi$ is given by
$$\psi(d_1)=0,\ \psi(e_1+e_2+e_3)=0,\ \ \psi(e_1-e_2)=e_1-e_2,\ \ \psi(e_2-e_3)=
e_2-e_3.$$

 The involution $\sigma_x$
permutes $e_1-e_2$ with $e_2-e_3$ and can be extended to the
span of $e_1,e_2,e_3$ by $\sigma_x(e_1)=-e_3$, $\sigma_x(e_2)=-e_2$.
Recall that $Z(\fgl_3)$ is generated by the symmetric polynomials
$p_s:=\sum_{i=1}^3 e_i^s$.  One has $\sigma_x(p_s)=(-1)^s p_s$.
Identifying $Z(\fsl_3)$ with 
$\mathbb{C}[p_2,p_3]$ we obtain
$$Z(\fsl_3)^{\sigma_x}=\mathbb{C}[p_2,p_3^2]=\mathbb{C}[p_2,p_6].$$
By~\cite{Sinv}, $Z(\fg)$ contains  $L_2$ (as above)  and
$L_6$
satisfying $\psi(L_2)=p_2$, $\psi(L_6)=p_6$. This establishes~(\ref{Zpsi}) and completes the proof.
\end{proof}

\subsection{}
\begin{cor}{corthetabasic}
Fix $x\in X(\fg)_r$ and denote by $\Mod_x$ the full category of $\fg_x$-modules.
Take $\chi\in \, Im\theta_x^*$ and $\chi' \in (\theta_x^*)^{-1}(\chi)$.

\begin{enumerate}
\item
If $\sigma_x=\Id$, then
$\DS_x(\Mod(\chi))\subset \Mod_x(\chi')$.
\item
If  $\sigma_x(\chi')\not=\chi'$, then
 $\DS_x(\Mod(\chi))\subset\Mod_x(\chi')+ \Mod_x(\sigma_x(\chi'))$.
\item
If $\sigma_x\not=\Id$ and $\sigma_x(\chi')=\chi'$, then
$(\Ker \chi')^2\DS_x(N)=0$ for each $N\in \Mod(\chi)$.
\end{enumerate}
\end{cor}
\begin{proof}
Take $N\in \Mod(\chi)$ and set $N':=\DS_x(N)$.
Set $\fm:=\Ker\chi$ and $\fm':=\Ker\chi'$.
 Then $\theta_x(\fm) \subset  \Ann_{\cZ(\fg_x)} N'$.

If $\theta_x$ is surjective, then $\theta_x(\fm)=\fm'$. This gives (i).
Consider the remaining case when $\theta_x$ is not surjective. 
By~\Thm{thmtheta}, $\theta_x(\fm)$ is a maximal ideal
in $\cZ(\fg_x)^{\sigma_x}$. Set $A:=\cZ(\fg_x)/\theta_x(\fm)$.
By above, $N'$ is an $A$-module and $A\not=\mathbb{C}$.
 
The algebra $A$ inherits the action of $\sigma_x$
and $A^{\sigma_x}=\cZ(\fg_x)^{\sigma_x}/\theta_x(\fm)\cong\mathbb{C}$, so
$$A=\mathbb{C}\oplus A_-,\ \text{ where }\ 
A_-:=\{a\in A|\ \sigma_x(a)=-a\}\not=0.$$

If $A_-=\mathbb{C}a$ with  $a^2=1$, then $N'=N'_+\oplus N'_-$,
where $N'_{\pm}=\{v\in N'|\ av=\pm v\}$ and
$\fm',\sigma(\fm')$ are the preimages of
the ideals $\mathbb{C}(a-1), \mathbb{C}(a+1)$ in
$\cZ(\fg_x)$. This gives (ii).

Assume that  $A_-\not=\mathbb{C}a$ with  $a^2=1$. For any $a_1,a_2\in A_-$ one has $a_1a_2\in A^{\sigma_x}=\mathbb{C}$, so 
$a_1a_2=0$. Then $A_-$ is the unique maximal ideal in $A$
and thus  $\fm'=\sigma_x(\fm')$ is the preimage of $A_-$. 
This gives (iii).
\end{proof}

\section{The algebra $\fg_x$ in the affine case}\label{gxaff}
Let $\fg$ be an indecomposable symmetrizable affine superalgebra or 
$\fgl(m|n)^{(1)}$.
We retain notation of~\ref{affro}. The dual Coxeter number for $\fg$ is given by 
$$h^{\vee}(\fg):=(\rho|\delta).$$
Note that $h^{\vee}(\fg)$ depends on the choice of $(-|-)$. It is easy to see
that $h^{\vee}(\fg)$  does not
depend on the choice of $\rho$; in the light of~\ref{oddrefl}, $h^{\vee}(\fg)$  does not
depend on the triangular decomposition of $\fg$.

Recall that 
$\dot{\Sigma}$ is a finite part of $\Sigma$ and $d\in {\fh}$ satisfies
$\delta(d)=1$ and $\alpha(d)=0$ for $\alpha\in\dot{\Sigma}$. One has
$${\fg}^d=\dot{\fg}\times\mathbb{C}K\times\mathbb{C}d,\ \ \ {\fh}=\dot{\fh}\oplus \mathbb{C}K\oplus \mathbb{C}d$$
where $K$ is a central element of ${\fg}$,
 $\dot{\fg}$ is a finite-dimensional Kac-Moody superalgebra
or $\fgl(m|n)$,
  and $\dot{\fh}$ is the Cartan subalgebra of $\dot{\fg}$.

We consider the triangular decomposition of $\dot{\fg}$ which is induced
by the triangular decomposition of ${\fg}$ (then
$\dot{\Sigma}$ is the base
of $\dot{\Delta}^+$). We set 
$$\dot{\rho}=\frac{1}{2}\displaystyle\sum_{\alpha\in\dot{\Delta}^+} (-1)^{p(\alpha)}\alpha$$
and fix ${\rho}\in\fh^*$ such that 
${\rho}_{\dot{\fh}}=\dot{\rho}|_{\dot{\fh}}$
(this condition holds for any choice of $\rho$ if $\fg\not=\fgl(m|n)$).

\subsection{The algebra $\fg_x$ for  $x\in X_{iso}$}\label{Hmn}
By contrast with the finite-dimensional case, $\fg_x$ is not always Kac-Moody, see the example~\ref{affbad} below. The following proposition shows
that $\fg_x$ is Kac-Moody for $x\in X_{iso}$.

\subsubsection{}
\begin{prop}{}
Take $x\in X_{iso}$ (see~\ref{Xiso} for the notation) and let $r$ be the cardinality of $\supp(x)$. 
\begin{enumerate}
\item
One has $\DS_x(\dot{\fg}^{(1)})\cong \dot{\fg}_x^{(1)}$ and
$\DS_x(H(M|N)^{(i)})=H(M-2r|N-2r)^{(i)}$ 
for $H(M|N)^{(i)}=A(M|2n-1)^{(2)}$, $A(2m|2n)^{(4)}, D(M+1|N)^{(2)}$.
\item
One has $\depth\fg=\defect\fg$ and $\rank x=r$.
\end{enumerate}
\end{prop}
\begin{proof}
By~\Cor{isosetnow}, we can assume that 
$x\in X(\dot{\fg})_r$. Now (i) follows from
the fact that $\DS_x$ maps the adjoint (resp., the standard)  
 $\dot{\fg}$-module
to the adjoint (resp., standard) $\dot{\fg}_x$-module; (ii) follows from (i).
\end{proof}

\subsubsection{}
Take $x\in X(\dot{\fg})_r$. Using~\ref{gx} we view ${\fg}_x$
 as a subalgebra of ${\fg}$ which 
contains $\dot{\fg}_x$,
$K$ and $d$; one has  
$\fg_x^d=\dot{\fg}_x\times\mathbb{C}K\times\mathbb{C}d$. The restriction of
the non-degenerate invariant bilinear form $(-|-)$ on $\fg$ gives a 
non-degenerate invariant bilinear form on $\fg_x$. 

\subsubsection{}
A $\fg$-module $N$ is called {\em restricted} if for each $v\in N$
$\fg_{\alpha}v=0$ for almost all positive roots $\alpha$. By~\cite{Kbook}, Ch. 2 
the bilinear form $(-|-)$ gives rise to the Casimir operator which acts on restricted $\fg$-modules by $\fg$-endomorphisms; in particular, this operator acts
on $M(\lambda)$ by $(\lambda+2\rho|\lambda)\Id$. 

The above definition of a restricted module can be reformulated as follows:
for each $v\in N$ there exists $j\in\mathbb{N}$ such that $\fg_{\alpha}v=0$ if
$\alpha(d)>j$. From this reformulation it follows that $\DS_x(N)$ is
a restricted $\fg_x$-module
if $N$ is a restricted $\fg$-module.

\subsubsection{}
\begin{prop}{propdualCox}
Take $x\in X(\dot{\fg})$.
\begin{enumerate}
\item
If the Casimir element of $\fg$ acts on a restricted $\fg$-module $N$
by $cId$ (where $c\in\mathbb{C}$), then the Casimir element of $\fg_x$ acts on
$\DS_x(N)$ by $cId$.

\item If $\dot{\fg}_x\not=0,\mathbb{C}$, then $h^{\vee}(\fg)=h^{\vee}(\fg_x)$.
\end{enumerate}
\end{prop}
\begin{proof}
In the light of~~\Cor{isosetnow} it is enough to consider the case
$\supp(x)\subset (-\dot{\Sigma}\cup\dot{\Sigma})$. For such $x$ (i)
is established in~\cite{GSaff}. For (ii)  consider the module
$L(\delta)$. This module is one-dimensional, so
as the vector space $\DS_x(L(\delta))$ is equal to $L(\delta)$.
Since $d$ acts on $L(\delta)$ by $\Id$, we obtain 
$\DS_x(L(\delta))=L_{\fg_x}(\delta_x)$, where $\delta_x$ is the minimal
imaginary root of $\fg_x$. Using (i) we get
$(\delta+2\rho|\delta)=(\delta_x+2\rho_x|\delta_x)$
which gives $(\rho|\delta)=(\rho_x|\delta_x)$ as required.
\end{proof}

\subsubsection{Remark}
Using the notation of~\ref{coreKMQ}
and normalizing the form by the condition $(\vareps_1|\vareps_1)=1$ we have

\begin{tabular}{|l|l|l|l|l|l|l|}
\hline
$\osp(M|N)^{(1)}$   &  $A(2m|2n)^{(4)}$ & $D(2|1;a)^{(1)}$ & 
$G(3)^{(1)}$& $F(4)^{(1)}$\\
\hline
 $M-N-2 $ & $m-n$ &  0 & 2 & 3 \\
\hline
\end{tabular}

and $h^{\vee}(\fg)=M-N$ for $\fg=A(M|N)^{(1)}$,
$A(M|2n-1)^{(2)}$, $D(M+1|N)^{(2)}$.

\subsection{Examples}\label{affbad}
Consider the case when  $\fg=\dot{\fsl}(2|1)^{(1)}$.
For each real root $\alpha$ fix a root vector $e_{\alpha}\in \fg_{\alpha}$.
Set $y:=e_{\vareps_1-\delta_1}$ and $h:=[e_{\delta_1-\vareps_1},e_{\vareps_1-\delta_1}]$.

One has $\DS_y(\fg)\cong \mathbb{C}K\times\mathbb{C}d=\fsl_1^{(1)}$.

One has  $\DS_{y+yt}(\fg)\cong\mathbb{C}K\times \ft$, where
$\ft$ is $(2|1)$-dimensional superalgebra
with a basis $h, e:=e_{\vareps_1-\vareps_2}, F:=e_{\vareps_2-\delta_1}$.
Thus $\ft$ has the relations
$$[e,F]=0,\ \ [h,e]=e,\ \ [h,F]=-F.$$

The algebra $\DS_{y+yt+yt^2}(\fg)$ is spanned by the images
of $K, y, h, ht^{-1}, e, et, F, Ft$.

\subsection{Induced modules}\label{indgd}
We consider  the functor 
$$\Ind:\ \fg^d-\Mod\to {\fg}-\Mod$$
which is given by the following construction: we
set
$${\fg}_{>0}:=\{g\in {\fg}|\ [d,g]=ig,\ i>0\},\ \ \ \ 
\Ind (V):=\Ind^{{\fg}}_{\fg^d+{\fg}_{>0}} V,$$
where a
$\fg^d$-module $V$ is viewed as a $(\fg^d+{\fg}_{>0})$-module
with the zero action of ${\fg}_{>0}$.

For $x\in X(\dot{\fg})$ 
we introduce similarly the functor
$\Ind_x:\ \fg^d_x-\Mod\to {\fg}_x-\Mod$.
We will use the following result.

\subsubsection{}
\begin{lem}{lemIndDSx}
Take $x\in X(\dot{\fg})$.
Let $V$ be a $\fg^d$-module, where $d$  acts diagonally. Then
\begin{equation}\label{IndDS}
\DS_x(\Ind (V))=\Ind_x (\DS_x(V)).\end{equation}
\end{lem}
\begin{proof}
Set $\fm:=\{g\in {\fg}|\ [d,g]=ig,\ i<0\}$ and note that
$\fm$ is $\ad\fg^d$-invariant. One has
$$\fm_x=\{g\in {\fg}_x|\ [d,g]=ig,\ i<0\}=\DS_x(\fm).$$
The embedding $\fm\to \cU(\fm)$ induces the map $\fm_x\to \DS_x(\cU(\fm))$ which gives the canonical map
$$\phi:\cU(\fm_x)\to \DS_x(\cU(\fm)).$$ 
Let us show that $\phi$ is bijective.
As in~\cite{DS}, this follows from the existence of the
following commutative diagram
$$\xymatrix{& \cU(\fm_x)\ar^{\sym_x}[d]\ar^{\phi}[r] & \DS_x(\cU(\fm))\ar^{\sym'}[d] & \\
& \cS(\fm_x)\ar^{\phi'}[r] & \DS_x(\cS(\fm)) & }
$$
The map $\sym_x$ is the usual symmetrization map. The map
$\sym'$ is induced by the symmetrization map
$\sym: \cU(\fm)\iso \cS(\fm)$, which is a bijection of $\fg^d$-modules.
The map $\phi'$
is the natural map. By~\cite{DS},
 $\DS_x$ is a tensor functor and $\phi'$ is bijective. Since
 $\sym,\sym'$ are also bijective, $\phi$ is bijective.

Now let us  prove the formula~(\ref{IndDS}).
We can assume that $d$ acts on $V$ by $aId_V$ ($a\in\mathbb{C}$).
Then the $d$-eigenvalues of $\Ind(V)$ are of the form
$a-\mathbb{N}$ and $V$ coincides with the $a$th eigenspace.
Therefore the $d$-eigenvalues of $\DS_x(\Ind (V))$ are of the form
$a-\mathbb{N}$ and the  $a$th eigenspace coincides with 
$\DS_x(V)$. 
This gives a ${\fg}_x$-homomorphism
$$\iota: \Ind_x (\DS_x(V))\to \DS_x(\Ind  (V)).$$
Since $\DS_x$ is a tensor functor we have
$$\DS_x(\Ind  (V))=\DS_x(\cU(\fm)\otimes V)=\DS_x(\cU(\fm))\otimes\DS_x(V)=
\cU(\fm_x)\otimes\DS_x(V).$$
Since $\Ind_x (\DS_x(V))=\cU(\fm_x)\otimes\DS_x(V)$, the map $\iota$ is bijective.
\end{proof}

\section{The category $\CO^{inf}_h(\fg)$}\label{sectO}
Let $\fg$ be a symmetrizable indecomposable Kac-Moody superalgebra
with a base $\Sigma$. We retain notation of~\ref{cc}.
In this section we consider a certain subcategory
of $\CO^{inf}(\fg)$ which contains the BGG category $\CO(\fg)$.
We will use this category in the proof of~\Thm{thmcore}.

\subsection{Definition}\label{Oinfh}
We fix $h\in\fh$ with the following properties  
\begin{equation}\label{eqh}\begin{array}{l}
\alpha(h)\in\mathbb{N} \text{ for all } \alpha\in\Sigma;\\
\{\alpha\in \Sigma|\ \alpha(h)=0\}\ \text{ is an iso-set.}
\end{array}\end{equation}

For each $\fg$-module  $N$ we denote by $\Omega_h(N)$ the set 
of $h$-eigenvalues in $N$:
$$\Omega_h(N):=\{b\in \mathbb{C}|\ \exists v\in N\setminus\{0\}\ 
hv=bv\}.$$

 We denote by $\CO^{inf}_h(\fg)$ the full subcategory of $\CO^{inf}(\fg)$  with the modules $N$ satisfying the following properties: all $h$-eigenspaces are finite-dimensional and 
there exists a finite set $\{c_i\}_{i=1}^s\subset\mathbb{C}$ such that
$$N=\displaystyle\bigoplus_{i=1}^s N_i,\ \ \ \Omega_h(N_i)\subset c_i-\mathbb{N}.$$
Note that $N$ has
 finite-dimensional weight spaces.

\subsection{Properties of $\CO^{inf}_h(\fg)$}
Clearly, $\CO^{inf}_h(\fg)$ is a ``locally small subcategory'', i.e.
for each exact sequence
$$0\to A\to B\to C\to 0$$
one has $B\in \CO^{inf}_h(\fg)$ if and only if $A,C\in \CO^{inf}_h(\fg)$.

\subsubsection{}\label{resultsOinf}
We will show that the category $\CO^{inf}_h(\fg)$ has the following properties:
\begin{itemize}
\item
The BGG category $\CO(\fg)$ lies in $\CO^{inf}_h(\fg)$.

\item The modules in $\CO^{inf}_h(\fg)$ have ``local composition series''
introduced in~\Prop{localseries} (this statement is a modification of Prop. 3.2 in~\cite{DGK}).
\end{itemize}

\subsubsection{}
The existence of  the "local  composition series'' for  $N\in\CO^{inf}(\fg)$ allows to define
the multiplicities $[N:L(\lambda)]$ using the formalism of Section 3 in~\cite{DGK}. One has
$$ \ch N=\sum_{\lambda} [N:L(\lambda)]\ch L(\lambda),\ \ \ [N:L(\lambda)]\in\mathbb{N}.$$


\subsection{Applications to $\DS$-functor}\label{infgx}
Let $x\in X(\fg)$ be such that  $\supp(x)\subset \Sigma\cup (-\Sigma)$.
It is easy to see that $S:=\supp(x)$ is an iso-set. By~\ref{gx},\ref{Hmn},
 ${\fg}_x:=\DS_x({\fg})$ is a Kac-Moody 
superalgebra
which can be viewed as a subalgebra of ${\fg}$.
The Cartan subalgebra ${\fh}_x$ lie in $\fh^x$; the root system $\Delta_x\subset\Delta$
is given by the same formula as in the finite-dimensional case described in~\ref{gx}.

 We consider the triangular decomposition of ${\fg}_x$ 
which is induced by the triangular decomposition of ${\fg}$.
Clearly, $\DS_x(\CO^{inf}(\fg))\subset\CO^{inf}(\fg_x)$.

\subsubsection{}
Fix $h$ satisfying~(\ref{eqh}) and such that $[h,x]=0$.
(For instance, take $h$ with $\alpha(h)=0$ for $\alpha\in S$
and $\alpha(h)=1$ for $\alpha\in\Sigma\setminus (-S\cup S)$).
 We denote by
$h_x$ the image of $h$ in $\fh_x$. Clearly, $\alpha(h_x)\in \mathbb{N}_{\geq 0}$
for each $\alpha\in\Delta^+_x$ and 
$$\{\alpha\in\Delta^+_x|\ \alpha(h_x)=0\}\subset \{\alpha\in\Delta^+|\ \alpha(h)=0\}$$
is an iso-set.  Hence $h_x\in\fh_x$ satisfies~(\ref{eqh}).

\subsubsection{}
Take $N\in\CO^{inf}_h(\fg)$. Since the $h_x$-eigenspaces of $\DS_x(N)$
are the images of the $h$-eigenspaces of $N$ (with the same eigenvalues),
we have
\begin{equation}\label{eqOga}
\DS_x(\CO^{inf}_h(\fg))\subset \CO^{inf}_{h_x}(\fg_x)
\end{equation}
and, in particular, $\DS_x(\CO(\fg))\subset \CO^{inf}_{h_x}(\fg_x)$.

\subsubsection{}
Take $N\in\CO^{inf}_h(\fg)$. By above, the multiplicities
$[N:L(\lambda)]$ and $[\DS_x(N):L_{\fg_x}(\nu)]$
are well-defined (where $L_{\fg_x}(\nu)$ stands for the corresponding
simple $\fg_x$-module). In~\ref{compmultproof} we will prove the following useful formula
\begin{equation}\label{compmult}\begin{array}{l}
\displaystyle\sum_{\lambda\in\fh^*}[N:L(\lambda)]\cdot
[\DS_x(L(\lambda)):L_{\fg_x}(\nu)]\in [\DS_x(N):L_{\fg_x}(\nu)]+2\mathbb{N}.\end{array}\end{equation}
(In particular, we will show that the left-hand side is finite).

\subsection{Proof of the properties~\ref{resultsOinf}}
We start from the following lemma.

\subsubsection{}
\begin{lem}{lemDeltaj}
For each $j\in\mathbb{N}$ the set 
$$\Delta(j):=\{\alpha\in\Delta^+|\ \alpha(h)=j\}$$ 
is finite.
\end{lem}
\begin{proof}
Recall that $S:=\{\beta\in\Delta^+|\ \beta(h)=0\}=\{\beta\in\Sigma|\ \beta(h)=0\}
$ is an iso-set and write
 $S=:\{\beta_i\}_{i=1}^r$.
By~\cite{Sint}, for each $i=1,\ldots, r-1$ the 
root $\beta_{i+1}$ is a simple root for 
for $r_{\beta_i}\ldots r_{\beta_1}(\Delta^+)$
(see~\ref{oddrefl} for notation), so $r_{\beta_{i+1}}\ldots r_{\beta_1}(\Delta^+)$
is well-defined.  Let $\Sigma'$ be the base
for  $r_{\beta_r}\ldots r_{\beta_1}(\Delta^+)$ and  
 $h'\in\fh$ be such that $\alpha(h')=1$ for each $\alpha\in\Sigma'$. 
By~\cite{Sint}, one has $-S\subset \Sigma'$,
so $\beta(h')=-1$ for $\beta\in S$. Set
$$a:=\max\{\alpha(h')|\ \alpha\in \Sigma\setminus S\}.$$

One has $\Delta(0)=S$. Take $j>0$. By~\ref{oddrefl}   
$$r_{\beta_r}\ldots r_{\beta_1}(\Delta^+)=(\Delta^+\setminus S)\cup(-S),$$ 
so this 
set contains $\Delta(j)$. Write
$\gamma\in\Delta(j)$ in the form
$\gamma=\sum_{\alpha\in\Sigma} m_{\alpha}\alpha$. One has $m_{\alpha}\in\mathbb{N}$ and $\sum_{\alpha\in\Sigma\setminus S} m_{\alpha}\leq j$.
Since $\gamma\in r_{\beta_r}\ldots r_{\beta_1}(\Delta^+)$ we have
$$0<\gamma(h')=\sum_{\alpha\in\Sigma\setminus S} m_{\alpha}\alpha(h')-
\sum_{\beta\in S}m_{\beta},$$
so $\sum_{\beta\in S}m_{\beta}< aj$. Hence $\Delta(j)$ lies in the set
$$ \{ \sum_{\alpha\in\Sigma} m_{\alpha}\alpha\ \text{
with }
m_{\alpha}\in\mathbb{N}\ \text{ for } \alpha\in \Sigma
 \text{ and }\sum_{\alpha\in\Sigma} m_{\alpha}\leq j(a+1)\}$$
which is finite.
\end{proof}

\subsubsection{}
Now let us prove the inclusion
\begin{equation}\label{OO}
\CO(\fg)\subset\CO^{inf}_h(\fg).\end{equation}

First, let us show that $\CO^{inf}_h(\fg)$ contains all Verma modules.
It is enough to show that
all $h$-eigenspaces in $\cU({\fn}^-)$ are finite-dimensional. 
Recall that $\cU({\fn}^-)\cong \cS({\fn}^-)$ as a $\fh$-module.
Fix $s\in\mathbb{N}$ and set
$$\begin{array}{l}
\fm:=\displaystyle\sum_{\alpha\in \Delta^+:\ 0<\alpha(h)\leq s}{\fg}_{-\alpha},\ \ \ 
\fs:=\displaystyle\sum_{\alpha\in \Delta^+:\ \alpha(h)=0} {\fg}_{-\alpha}.\end{array}$$
 One has
$$\{v\in \cS({\fn}^-)|\ hv=sv\}\subset\bigl(\sum_{i=0}^s \cS^i(\fm)\otimes \cS(\fs)\bigr).$$
The assumptions~(\ref{eqh}) imply that  $\fs$ is an odd finite-dimensional space.
By~\Lem{lemDeltaj}, $\dim\fm<\infty$. Therefore each summand
$\cS^i(\fm)\otimes \cS(\fs)$ is a finite-dimensional space. Hence
the $h$-eigenspaces in $\cS({\fn}^-)$ are finite-dimensional.

Denote by $Ver$ the set of isomorphism classes of all quotients
of Verma modules. Take $N\in \CO(\fg)$.  Since $N$ is finitely generated, 
$N$ admits a finite filtration with 
cyclic quotients. 
It is easy to see that each cyclic module in $\CO(\fg)$ 
admits a finite filtration with quotients in $Ver$. Hence $N$ admits a finite filtration with quotients in $Ver$. Since $\CO^{\inf}_h(\fg)$
is ``locally small''and contains all Verma modules,
it contains $N$. This completes the proof of~(\ref{OO}).

\subsubsection{}
\begin{prop}{localseries}
Take $N\in\CO^{inf}_h(\fg)$ and
 $a\in\mathbb{C}$. The module $N$ admits a ``local composition 
series at $a$'', which is
a finite filtration 
$$0=N^0\subset N^1\subset \ldots\subset N^t=N$$
with the following property: for every $i=1,\ldots,t$ 

either $N^i/N^{i-1}\cong L(\lambda_i)$ 
with $\lambda_i(h)\in a+\mathbb{N}$  

or $\ \Omega_h(N^i/N^{i-1})\cap (a+\mathbb{N})=\emptyset$.
\end{prop}
\begin{proof}
The proof is a slight modification of the proof of  Prop. 3.2 in~\cite{DGK}.
For every $N\in \CO^{inf}_h(\fg)$ we set
$$N_{\geq a}:=\displaystyle\sum_{i=0}^{\infty}\{v\in N|\ hv=(a+i)v\},\ \ \ m(N):=\dim N_{\geq a}.$$
From the definition of $\CO^{inf}_h(\fg)$ it follows that
$m(N)<\infty$.
We prove the statement by induction on 
$m(N)$. If $m(N)=0$, then $0=N^0\subset N^1=N$
is the required filtration. Let $m(N)>0$. Note that $N_{\geq a}$ is a finite-dimensional $(\fh+\fn)$-module. Let $v\in N_{\geq a}$ be a primitive vector
of weight $\lambda$ (one has $\lambda(h)\in a+\mathbb{N}$).
Let $M$ be a submodule generated by $v$
and $M'$ be the maximal submodule of $M$ satisfying $M'_{\lambda}=0$.
Then
$$0\subset M'\subset M\subset N\ \text{ with } M/M'\cong L(\lambda).$$
Since $\lambda(h)\in a+\mathbb{N}$, one has $m(L(\lambda))>0$ and thus
$$m(M')+m(N/M')=m(N)-m(M/M')<m(N).$$
By induction $M'$ and $N/M'$ admit suitable filtrations; combining these
filtrations we get the required filtration for $N$.
\end{proof}

\subsection{Proof of~(\ref{compmult})}\label{compmultproof}
Fix $\nu\in\fh_x^*$. For each $N\in\CO^{inf}_h(\fg)$ we set
$$n(N):=\sum_{\lambda}[N:L(\lambda)]\cdot
[\DS_x(L(\lambda)):L_{\fg_x}(\nu)].$$
Our goal is to show that $n(N)<\infty$ and 
that $[\DS_x(N):L_{\fg_x}(\nu)]\in n(N)-2\mathbb{N}$.

Consider the 
filtration as in~\Prop{localseries} for $a:=\nu(h_x)$. We prove the assertion
by induction on the length of this filtration $t$; we denote this length by $t$.
If the filtration has zero length, then $N=0$ and the assertion holds.

Recall that for any $M\in\CO^{inf}_h$ one has $\Omega_{h_x}(\DS_x(M))\subset\Omega_h(M)$. This implies
\begin{equation}\label{ana}
\Omega_h(M)\cap (a+\mathbb{N})=\emptyset\ \ \Longrightarrow\ \ 
[\DS_x(M'):L_{\fg_x}(\nu)]=0\text{ if $M'$ is a subquotient of $M$}\end{equation}

Consider the case when $\Omega_h(N/N^{t-1})\cap (a+\mathbb{N})=\emptyset$.
By~(\ref{ana}), 
$$[\DS_x(N/N^{t-1}):L_{\fg_x}(\nu)]=0$$
and $[N/N^{t-1}:L(\lambda)]=0$ if 
$[\DS_x(L(\lambda)):L_{\fg_x}(\nu)]\not=0$. 
Using Hinich's Lemma and the first formula we get
$$[\DS_x(N):L_{\fg_x}(\nu)]=[\DS_x(N^{t-1}):L_{\fg_x}(\nu)].$$
The second formula gives $n(N)=n(N^{t-1})$. By the induction hypothesis, the required assertions hold for $N^{t-1}$; thus the assertions hold for $N$.

Now consider the remaining case when $N/N^{t-1}\cong L(\lambda)$ with
$\lambda(h)\in a+\mathbb{N}$. One has
$$n(N)=n(N^{t-1})+[\DS_x(L(\lambda)):
L_{\fg_x}(\nu)].$$
By Hinich's Lemma
$$[\DS_x(N):L_{\fg_x}(\nu)]\in [\DS_x(N^{t-1}):L_{\fg_x}(\nu)]+[\DS_x(L(\lambda)):
L_{\fg_x}(\nu)]-2\mathbb{N},$$
so 
$$[\DS_x(N):L_{\fg_x}(\nu)]-n(N)\in [\DS_x(N^{t-1}):L_{\fg_x}(\nu)]-n(N^{t-1})-2\mathbb{N}.$$
By induction, $n(N)<\infty$ and $[\DS_x(N):L_{\fg_x}(\nu)]-n(N)\in -2\mathbb{N}$.\qed

\section{$\DS$-functor and cores in the affine case}\label{Thm91}
In this section $\fg$ is an indecomposable symmetrizable affine superalgebra or 
$\fgl(m|n)^{(1)}$.
We retain notation of Section~\ref{gxaff} and of~\ref{coreKMQ}.

Take $x\in X(\fg)$ such that $\supp(x)\subset (-\dot{\Sigma}\cup \dot{\Sigma})$;
set $S:=\supp(x)$. We identify $\fg_x$ with a subalgebra of 
$\fg$ as in~\ref{Hmn}. We
consider the triangular decomposition of ${\fg}_x$ 
which is induced by the triangular decomposition of ${\fg}$. 
Note that  $(\fg^d)_x$ coincides with $(\fg_x)^d$; we will denote this algebra
by $\fg_x^d$. One has
\begin{equation}\label{gdx}
{\fg}_x^d=\dot{\fg}_x\times\mathbb{C}K\times\mathbb{C}d,\ \ \ 
{\fh}_x=\dot{\fh}_x\oplus \mathbb{C}K\oplus \mathbb{C}d;\end{equation}
 the triangular decomposition of $\dot{\fg}_x$ is as in~\ref{gx}.
We choose the Weyl vectors
$\dot{\rho}$, $\rho$, $\dot{\rho}_x$, ${\rho}_x$ as in Section~\ref{gxaff}. 
The main result of this section 
is the following theorem.

\subsection{}
\begin{thm}{thmcore}
Fix $x\in X(\fg)$ with  $\supp(x)\subset (\Sigma\cup(-\Sigma))$.
Let $N\in\CO(\fg)$ be  an indecomposable $\fg$-module and let $L$
be a simple subquotient of $\DS_x(N)$.

\begin{enumerate}
\item 
$\atyp(L)=\atyp(N)-\rank x$.
\item
If $\dot{\fg}$ is not exceptional, 
 then
$\Core(N)=\Core(L)$.

\end{enumerate}
\end{thm}

\subsection{Remark}\label{examplesl21}
By constrast with $\Fin(\fg)$, $\DS_x$ can kill simple modules in $\CO(\fg)$
even in the case when the  
atypicality of the module is equal to $\rank x$.

For example, take $\fg:=\fsl(2|1)$ with the base $\Sigma=\{\beta_1,\beta_2\}$,
where $\beta_1,\beta_2$ are odd. If $(\lambda|\beta_2)=0$ and
$L(\lambda-\rho)$ is infinite-dimensional, then 
$$\DS_x(L(\lambda-\rho))=\left\{\begin{array}{llll}
\mathbb{C} & & \text{ if }& \supp(x)=\{-\beta_2\}\\
0 & & \text{ if } & \supp(x)=\{-\beta_1\}.
\end{array}\right.$$

\subsection{Proof of~\Thm{thmcore}}\label{proofthmcore}
By~\Cor{corcore},   (i) follows from (ii) if $\dot{\fg}$ is non-exceptional.
If $\dot{\fg}$ is exceptional, then $\defect\fg=1$ and (i) reduces to 
the following assertion:
$\DS_x(N)=0$ if $N$ is typical. Therefore we may (and will) substitute (i) by

$\ \ \ $ (i')
$N$ is atypical

(meaning that the existence of a subquotient $L$ implies that $N$ is atypical).

Assume that (i') or (ii) does not hold for a certain pair $(N, L)$.
Using~(\ref{compmult}), we conclude that the same assertion
does not hold for the pair $(L(\lambda),L)$,
where $L(\lambda)$ is a simple subquotient of $N$. Write
$L=L_{\fg_x}(\nu)$  and let $A\subset\fh^*$ 
be the set of $\lambda$s such that the assertion does not hold
for the pair $(L(\lambda),L_{\fg_x}(\nu))$. By above, $A$ is non-empty.

\subsubsection{}
We fix $h\in \fh^*$ satisfying
$$\alpha(h)=\left\{\begin{array}{ll}
0 & \text{ for }\ \alpha\in \supp(x),\\
1 & \ \text{ for} \ \alpha\in {\Sigma}\setminus (\supp(x)\cup(-\supp(x)).
\end{array}\right.
$$
Note that $h\in \fg^x$. Denote by $h_x$ the image of $h$ in $\fh_x$. 
Retain notation of Section~\ref{sectO} and recall that
for each $M\in \DS_x(\CO({\fg}))$ one has
$$\Omega_{h_x}(\DS_x(M))\subset \Omega_h(M),\ \ 
\DS_x(M)\in\CO^{inf}_{h_x}({\fg}_x).$$

By~(\ref{ana}),
for each $\lambda\in A$ one has $\lambda(h)\in \nu(h_x)+\mathbb{N}$.
We fix $\lambda\in A$ with the minimal value of $\lambda(h)$. Then
$L_{\fg_x}(\nu)$ is a subquotient of $\DS_x(L(\lambda))$ and 
either $\lambda+\rho$ is typical (for (i')) or 
$\Core(\lambda+\rho)\not=\Core(\nu+\rho_x)$ (for (ii)).

\subsubsection{}
Retain notation of~\ref{indgd} and consider the induced functors
$$\Ind: \fg^d-\Mod\to \fg-\Mod,\ \ \  \ \Ind_x: \fg^d_x-\Mod\to \fg_x-\Mod.$$

The module $L(\lambda)$ is a quotient of the induced module 
$\Ind(L_{\fg^d}(\lambda))$ by its maximal proper submodule, 
which we denote by $N'$. 

Let $L(\lambda')$ be a subquotient
of $N'$. We claim that 
\begin{equation}\label{lambdadd}
  \lambda'(h)<\lambda(h),\ \ \ \atyp \lambda=\atyp\lambda',\ \ \ 
\Core(\lambda+{\rho})=\Core(\lambda'+{\rho}).
\end{equation}
(using the notation $\Core$ we  always  assume that $\fg$ is non-exceptional).
Indeed, the first formula follows from the fact that 
$\lambda-\lambda'\in \mathbb{N}{\Sigma}\setminus\mathbb{N}\dot{\Sigma}$.
  The rest of the formulae  follow from~\ref{coreKMQ}
and the fact that $N$ is a subquotient of $M(\lambda)$.

Combining~(\ref{lambdadd}) and the  minimality of $\lambda(h)$, we get $\lambda'\not\in A$.
Using~(\ref{compmult}) we obtain
$$[N':L_{{\fg}_x}(\nu)]=0.$$
Recall that $L_{{\fg}_x}(\nu)$ is a subquotient of $\DS_x(L(\lambda))$.
 Hinich's Lemma give 
$$[\DS_x(\Ind(L_{\fg^d}(\lambda))):L_{{\fg}_x}(\nu)]=[\DS_x(L(\lambda)):L_{{\fg}_x}(\nu)]\not=0.$$
Using~(\ref{IndDS}) we conclude
\begin{equation}\label{eqDsnu}
\ [\Ind_x (\DS_x(L_{\fg^d}(\lambda))):L_{{\fg}_x}(\nu)]\not=0.
\end{equation}

In particular, $\DS_x(L_{\fg^d}(\lambda))\not=0$. Since
$\fg^d=\dot{\fg}\times\mathbb{C}K\times \mathbb{C}d$ and 
$\dot{\fg}$ is a finite-dimensional Kac-Moody superalgebra, 
$\DS_x(L_{\fg^d}(\lambda))\not=0$ implies the atypicality of 
the weight $\lambda|_{\dot{\fh}}+\dot{\rho}$.
Therefore $(\lambda|_{\dot{\fh}}+\dot{\rho}|\alpha)=0$
for some $\alpha\in\dot{\Delta}_{iso}$. By our choice of $\rho$, this gives
$(\lambda+{\rho}|\alpha)=0$, 
 so $\lambda+\rho$ is an atypical weight. This establishes (i').

\subsubsection{} 
By~\Prop{localseries}, the $\fg^d_x$-module $\DS_x(L_{\fg^d}(\lambda))$ admits 
a finite filtration 
$$0=M^0\subset M^1\subset \ldots\subset M^t=\DS_x(L_{\fg^d}(\lambda)),$$
where for each index $i$ the module $M^i/M^{i-1}$ is a simple $\fg_x^d$-module or
\begin{equation}\label{iOmgea}
\Omega_{h_x}(M^i/M^{i-1})\subset \nu(h_x)-1-\mathbb{N}.\end{equation}
Consider the corresponding filtration of the induced module
$\Ind_x (\DS_x(L_{\fg^d}(\lambda)))$. By~(\ref{eqDsnu}) for some index $j$ we have
$$[\Ind_x (M^j/M^{j-1}):L_{{\fg}_x}(\nu)]\not=0.$$
Since $\Omega_{h_x}(\cU(({\fg}_x)_{<0}))\subset -\mathbb{N}$,
for $M^i/M^{i-1}$ satisfying~(\ref{iOmgea}) we have 
$$\Omega_{h_x}(\Ind_x(M^i/M^{i-1}))\subset \Omega_{h_x}(M^i/M^{i-1})-
\mathbb{N}\subset \nu(h_x)-1-\mathbb{N}.$$
Thus $M^j/M^{j-1}$ does not satisfy~(\ref{iOmgea}), so
$M^j/M^{j-1}=L_{\fg^d_x}(\mu)$ for some $\mu\in \fh_x^*$. Then
$L_{{\fg}_x}(\nu)$ is a subquotient of 
$\Ind_x (L_{\fg^d_x}(\mu))$, which is a quotient of $M_{\fg_x}(\mu)$.
Hence $L_{{\fg}_x}(\nu)$ is a subquotient of $M_{\fg_x}(\mu)$.
Using~\ref{coreKMQ}  we get
$$
\Core(\mu+{\rho}_x)=\Core(\nu+{\rho}_x).$$
On the other hand, $L_{\fg^d_x}(\mu)$ is a subquotient of 
$\DS_x(L_{\fg^d}(\lambda))$. Using~(\ref{gdx}) we conclude that 
  the module $L_{\fg_x}(\mu|_{\dot{\fh}_x})$ is a subquotient of 
$\DS_x(L_{\fg}(\lambda|_{\dot{\fh}}))$.  By~\cite{DS} (see also~\ref{preservescore}) this gives
\begin{equation}\label{korka}
\Core(\lambda|_{\dot{\fh}}+\dot{\rho})=\Core(\mu|_{\dot{\fh}}+\dot{\rho}_x).
\end{equation}

  Clearly, $\lambda(K)=\nu(K)=\mu(K)$. One has $(\rho|\delta)=\rho(K)$.
Using~\Prop{propdualCox} (ii) we get
 $\rho(K)=\rho_x(K)$. Therefore
\begin{equation}\label{hatka}(\mu+{\rho}_x)(K)=
(\lambda+{\rho})(K).\end{equation}
Recall that 
$({\rho}-\dot{\rho}|\vareps_i)=({\rho}-\dot{\rho}|\delta_j)=0$
for all indices $i,j$. Combining this fact with~(\ref{korka})
and~(\ref{hatka})  we conclude that
$$\Core(\lambda+{\rho})=\Core(\mu+{\rho}_x)$$
which contradicts to 
$\Core(\lambda+{\rho})\not=\Core(\nu+{\rho}_x)$. This establishes (ii).
\qed

\subsection{$\DS_x$ for vertex algebras}
Let ${\fg}=\dot{\fg}^{(1)}$. Let $V^k(\fg)$ be the affine vertex superalgebra of level $k$ and let
 $V_k(\fg)$ be the corresponding simple affine vertex superalgebra.
A natural analogue of $\Fin(\dot{\fg})$ 
is the category $KL_k(\fg)$, which is the full subcategory 
of $\dot{\fg}$-locally finite  $V_k(\fg)$-modules in $\CO(\fg)$.
This category was studied in~\cite{Adam} and in other papers.

The set $\dot{X}:=\{g\in\dot{\fg}_1| x^2=0\}$ is a subset in $X_{iso}$. Take $x\in\dot{X}$.
By above,  $\DS_x(\fg)$ is the affinization of $\DS_x(\dot{\fg})$
and that $\DS_x$ gives the functor  from the category of $V^k(\fg)$-modules
 to the category of 
 $V^k(\fg_x)$-modules (see~\cite{GSvertex}).
Take a non-zero $x\in \dot{X}$.

\subsubsection{Conjecture}\label{conjKL1}
For ``admissible'' values of $k$ (introduced in~\cite{GSvertex}) one has 
$$\DS_x(V_k(\fg))=V_k(\fg_x),\ \ \ \ \DS_x(KL_k(\fg))\subset KL_k(\fg_x).$$

\subsubsection{Conjecture}\label{conjKL2}
Each block in $KL_k$ satisfies
the analogue of~(\ref{depthDsxL})
 and~(\ref{depthBL}) for $\dot{\depth}$, where
 $\dot{\depth}$ is defined by substituting  $X_{iso}$ by  $\dot{X}$
(so $\dot{\depth}\leq \depth$).

\subsubsection{Conjecture}\label{conjKL3}
For $k\in\mathbb{N}_{>0}$ and $x\in\dot{X}$ the functor  $\DS_x:
KL_k(\fg)\to KL_k(\fg_x)$ satisfies  the  following
properties: 
\begin{itemize}
\item
$\DS_x(L)$ is semisimple
for each $L\in\Irr(KL_k(\fg))$;
\item
the extension graphs of $KL_k(\fg),KL_k(\fg_x)$ admit
bipartitions 
$$\pari: \Irr(KL_k(\fg))\to \{\pm 1\},\ \ \Irr(KL_k(\fg_x))\to \{\pm 1\}$$
$ \ \ \ \ \ \ $ and $[\DS_x(L):L']=0$ if $\pari(L)\not=\pari(L')$.
\end{itemize}
(For $\Fin(\dot{\fg})$ such bipartition is studied in~\cite{Gdex}).

\subsubsection{Remark}
Consider the case $k\in\mathbb{N}_{>0}$.
By~\cite{GSvertex} Sect. 3, 
\Conj{conjKL1} holds for
 $\fg\not=\osp(2m+2|2m)$, $D(2|1;a)$.  
 
Take $\dot{\fg}=\fsl(1|n)$. One has $\DS_x(\dot{\fg})=\fsl_{n-1}$ and $\DS_x(\fg)=\fsl_{n-1}^{(1)}$. 
In this case $KL_k(\fg)$ is the subcategory of the  integrable modules
in $\CO^k$ and $\DS_x$ gives a functor from $KL_k(\fg)$ to the category of integrable $\fsl_{n-1}^{(1)}$-modules of level $k$. By~\cite{GSaff}, 
the conjectures~\ref{conjKL1}--\ref{conjKL3} hold in this case.

%

\end{document}